\documentclass[a4paper,11pt]{article}
\usepackage{latexsym,epsfig}
\usepackage{amsmath,amssymb,amsthm}
\usepackage{graphicx}
\usepackage{color}
\usepackage{subfigure}
\usepackage{listings}
\usepackage{courier}
\usepackage{verbatim}
\usepackage[margin=1.1in]{geometry}
\usepackage{tabularx}
\usepackage{epstopdf}

\newtheorem{theorem}{Theorem}

\newtheorem{thm}{Theorem}[section]
\newtheorem{prop}[thm]{Proposition}
\newtheorem{lem}[thm]{Lemma}
\newtheorem{cor}[thm]{Corollary}

\newtheorem{obs}[thm]{Observation}

\newenvironment{proofof}[1]
      {\medskip\noindent{\em Proof of #1:}\hspace{1mm}}
      {\hfill$\Box$\medskip}
      
\newenvironment{packed_enum}{

\begin{enumerate}
  \setlength{\itemsep}{1pt}
  \setlength{\parskip}{0pt}
  \setlength{\parsep}{0pt}
}{\end{enumerate}}

\newenvironment{packed_item}{

\begin{itemize}
  \setlength{\itemsep}{1pt}
  \setlength{\parskip}{0pt}
  \setlength{\parsep}{0pt}
}{\end{itemize}}

\def\X{{\mbox{cr}}}
\def\E{{\mathbb E}}

\begin{document}

\title{
     On topological graphs with at most four crossings per edge%\footnote{This paper is identical to the paper~\cite{CGTA} except for the application of the results here to Albertson conjecture that were observed after the original paper was already submitted. Since this application follows almost verbatim the arguments in~\cite{BT10} it does not deserve a paper of its own, however, for completeness and future reference it is included in this arXived version of the paper.}  
   }

\author{
Eyal Ackerman\thanks{
Department of Mathematics, Physics, and Computer Science, 
University of Haifa at Oranim, Tivon 36006, Israel.
{\tt ackerman@sci.haifa.ac.il}.}}

%\date{}
\maketitle

\begin{abstract}
We show that if a graph $G$ with $n \geq 3$ vertices can be drawn in the plane
such that each of its edges is involved in at most four crossings,
then $G$ has at most $6n-12$ edges.
This settles a conjecture of Pach, Radoi\v{c}i\'{c}, Tardos, and T\'oth,
and yields a better bound for the famous \emph{Crossing Lemma}:
The \emph{crossing number}, $\X(G)$, of a (not too sparse) graph $G$ with $n$ vertices and $m$ edges
is at least $c\frac{m^3}{n^2}$, where $c > 1/29$.
This bound is known to be tight, apart from the constant $c$ for which
the previous best lower bound was $1/31.1$.

As another corollary we obtain some progress on the \emph{Albertson conjecture}:
Albertson conjectured that if the chromatic number of a graph $G$ is $r$, then $\X(G) \geq \X(K_r)$.
This was verified by Albertson, Cranston, and Fox for $r \leq 12$, and for $r \leq 16$ by Bar\'at and T\'oth.
Our results imply that Albertson conjecture holds for $r \leq 18$.  
\end{abstract}

\section{Introduction}
\label{sec:intro}

Throughout this paper we consider graphs with no loops or parallel edges, unless stated otherwise.
A \emph{topological graph} is a graph drawn in the plane with its vertices
as points and its edges as Jordan arcs that connect the corresponding points
and do not contain any other vertex as an interior point.
Every pair of edges in a topological graph has a finite number of intersection points,
each of which is either a vertex that is common to both edges,
or a crossing point at which one edge passes from one side of the other edge to its other side.
Note that an edge may not cross itself, since edges are drawn as Jordan arcs.
If every pair of edges intersect at most once, then the topological graph is \emph{simple}. 

A \emph{crossing} in a topological graph consists of a pair of crossing edges and a point in which they cross.
We will be interested in the maximum number of crossings an edge is involved in,
and in the total number of crossings in a graph.
Therefore, we may assume henceforth that the topological graphs that we consider
do not contain three edges crossing at a single point.
Indeed, if more than two edges cross at a point $p$, then we can redraw
these edges at a small neighborhood of $p$ such that no three of them cross at a point
and without changing the number of crossings in which each of these edges is involved.
% We may further assume that the topological graphs that we consider
% do not contain self-crossing edges, for such loops can be deleted (again, without increasing
% the number of crossings per edge).
% NO NEED SINCE WE DEFINE EDGES AS JORDAN ARCS

Therefore, the \emph{crossing number} of a topological graph $D$, $\X(D)$, 
can be defined as the total number of crossing points in $D$.
The crossing number of an abstract graph $G$, $\X(G)$, is the minimum value of $\X(D)$ 
taken over all drawings $D$ of $G$ as a topological graph.
It is not hard to see that if $D$ is a drawing of $G$ as a topological graph
such that $\X(G) = \X(D)$, then $D$ is a simple topological graph.
Indeed, if $D$ has two edges $e_1$ and $e_2$ that intersect at two points $p_1$ and $p_2$,
then at least one of these intersection points is a crossing point, for otherwise
$e_1$ and $e_2$ are parallel edges.
By swapping the segments of $e_1$ and $e_2$ between $p_1$ and $p_2$
and modifying the drawing of $e_1$ in a small neighborhood of the crossing points
among $p_1$ and $p_2$, we obtain a drawing of $G$ as a topological graph with fewer crossings.

The following result was proved by Ajtai, Chv\'atal, Newborn, Szemer\'edi \cite{ACNS82} and, 
independently, Leighton \cite{L83}.

\begin{theorem}[\cite{ACNS82,L83}]
There is an absolute constant $c>0$ such that
for every graph $G$ with $n$ vertices and $m \geq 4n$ edges we have $\X(G) \geq c\frac{m^3}{n^2}$.
\end{theorem}

This celebrated result is known as the \emph{Crossing Lemma} and has numerous 
applications in combinatorial and computational geometry, number theory, and
other fields of mathematics.

The Crossing Lemma is tight, apart from the multiplicative constant $c$.
This constant was originally very small, and later was shown to be at least $1/64 \approx 0.0156$,
by a very elegant probabilistic argument due to Chazelle, Sharir, and Welzl~\cite{AZ04}.
Pach and T\'oth~\cite{PT97} proved that $0.0296 \approx 1/33.75 \leq c \leq 0.09$ (the lower bound applies for $m \geq 7.5n$).
Their lower bound was later improved by Pach, Radoi\v{c}i\'{c}, Tardos, and T\'oth~\cite{PR+06}
to $c \geq 1024/31827 \approx 1/31.1 \approx 0.0321$ (when $m \geq \frac{103}{16}n$).
Both improved lower bounds for $c$ were obtained using the same approach,
namely, finding many crossings in sparse graphs.
To this end, it was shown that topological graphs with few crossings per edge are sparse.

Denote by $e_k(n)$ the maximum number of edges in a topological graph with $n>2$ vertices
in which every edge is involved in at most $k$ crossings.
Let $e^*_k(n)$ denote the same quantity for \emph{simple} topological graphs.
It follows from Euler's Polyhedral Formula that $e_0(n) \leq 3n-6$.
Pach and T\'oth~\cite{PT97} showed that $e^*_k(n) \leq 4.108\sqrt{k}n$, for $k > 0$, and also gave
the following better bounds for $k \leq 4$.

\begin{theorem}[\cite{PT97}]\label{thm:PT97}
$e^*_k(n) \leq (k+3)(n-2)$ for $0 \leq k \leq 4$.
Moreover, these bounds are tight when $0 \leq k \leq 2$ for infinitely many values of $n$.
\end{theorem}

Pach et al.~\cite{PR+06} observed that the upper bound in Theorem~\ref{thm:PT97}
applies also for not necessarily simple topological graphs when $k \leq 3$,
and proved a better bound for $k=3$.

\begin{theorem}[\cite{PR+06}]\label{thm:PR+06}
$e_3(n) \leq 5.5n-11$.
This bound is tight up to an additive constant.
\end{theorem}

By Theorem~\ref{thm:PT97}, $e^*_4(n) \leq 7n-14$.
Pach et al.~\cite{PR+06} claim that similar arguments to their proof of Theorem~\ref{thm:PR+06}
can improve this bound to $(7-\frac{1}{9})n-O(1)$.
They also conjectured that the true bound is $6n-O(1)$.
Here we settle this conjecture on the affirmative. %, also for not necessarily simple topological graphs.

\begin{theorem}\label{thm:4crossings}
Let $G$ be a simple topological graph with $n \geq 3$ vertices.
If every edge of $G$ is involved in at most four crossings, 
then $G$ has at most $6n-12$ edges.
This bound is tight up to an additive constant.
\end{theorem}

In fact, Theorem~\ref{thm:4crossings} holds also for topological graphs that are not necessarily simple.
However, since this stronger statement is not needed to improve the Crossing Lemma,
and because proving it would complicate an already long and technical proof, 
we chose not to prove this generalization here.

Using the bound in Theorem~\ref{thm:4crossings} and following the footsteps of~\cite{PR+06,PT97} 
we obtain the following linear lower bound for the crossing number.

\begin{theorem}\label{thm:linear-new}
Let $G$ be a graph with $n>2$ vertices and $m$ edges.
Then $\X(G) \geq 5m-\frac{139}{6}(n-2)$.
\end{theorem}

This linear bound is then used to get a better constant factor for the bound in the Crossing Lemma,
by plugging it into its probabilistic proof, as in~\cite{Mon05,PR+06,PT97}. 

\begin{theorem}\label{thm:crossing-lemma}
Let $G$ be a graph with $n$ vertices and $m$ edges. 
Then $\X(G) \geq \frac{1}{29} \frac{m^3}{n^2}-\frac{35}{29}n$.
If $m \geq 6.95n$ then $\X(G) \geq \frac{1}{29} \frac{m^3}{n^2}$.
\end{theorem}

\paragraph{Albertson conjecture.}
The \emph{chromatic number} of a graph $G$, $\chi(G)$, is the minimum number of colors
needed for coloring the vertices of $G$ such that none of its edges has monochromatic endpoints.
In 2007 Albertson conjectured that if $\chi(G)=r$ then $\X(G) \geq \X(K_r)$.
That is, the crossing number of an $r$-chromatic graph is
at least the crossing number of the complete graph on $r$ vertices.

If $G$ contains a \emph{subdivision}\footnote{A subdivision of $K_{r}$ consists of $r$ vertices, 
each pair of which is connected by a path such that the paths are vertex disjoint (apart from their endpoints).} 
of $K_r$ as a subgraph, then clearly $\X(G) \geq \X(K_r)$.
A stronger conjecture (than Albertson conjecture and also than \emph{Hadwiger conjecture}\footnote{Hadwiger conjecture
states that if $\X(G)=r$ then $K_r$ is a minor of $G$.}) 
is therefore that if $\chi(G)=r$ then $G$ contains a subdivision of $K_r$.
However, this conjecture, which was attributed to Haj\'os, was refuted for $r \geq 7$~\cite{Cat79,EF81}.

Albertson conjecture is known to hold for small values of $r$:
For $r=5$ it is equivalent to the Four Color Theorem, whereas
for $r=6$, $r \leq 12$, and $r \leq 16$, it was verified by
Oporowskia and Zhao~\cite{OZ09}, Albertson, Cranston, and Fox~\cite{ACF09}, and Bar\'at and T\'oth~\cite{BT10}, respectively.  
By using the new bound in Theorem~\ref{thm:linear-new} and following the approach in~\cite{ACF09, BT10},
we can now verify Albertson conjecture for $r \leq 18$.

\begin{theorem}\label{thm:Albertson}
Let $G$ be an $n$-vertex $r$-chromatic graph.
If $r \leq 18$ or $r=19$ and $n \neq 37,38$, then $\X(G) \geq \X(K_r)$.
\end{theorem}

\paragraph{Organization.}
The bulk of this paper is devoted to proving Theorem~\ref{thm:4crossings} in Section~\ref{sec:4crossings}.
In Section~\ref{sec:applications} we recall how the improved crossing numbers are obtained,
and their consequences.

%%%%%%%%%%%%%%%%%%%%%%%%%%%%%%%%%%%%%%%%%%%%%%%%%%%%%%%%%%%%%%%%%%%%%%%%%%%%%%%%%%%%%%%%%
\section{Proof of Theorem~\ref{thm:4crossings}}
\label{sec:4crossings}
%%%%%%%%%%%%%%%%%%%%%%%%%%%%%%%%%%%%%%%%%%%%%%%%%%%%%%%%%%%%%%%%%%%%%%%%%%%%%%%%%%%%%%%%%

Most of this section (and of the paper) is devoted to proving the upper bound in Theorem~\ref{thm:4crossings}.
Let $G$ be a topological graph with $n\geq 3$ vertices and at most four crossings per edge.
We prove that $G$ has at most $6n-12$ edges by induction on $n$.
For $n \leq 10$ we have $6n-12 > {{n}\choose{2}}$ and thus the theorem trivially holds.
Therefore, we assume that $n \geq 11$.
Furthermore, we may assume that the degree of every vertex in $G$ is at least $7$,
for otherwise the theorem easily follows by removing a vertex of a small degree
and applying the induction hypothesis.

For a topological graph $G$ we denote by $M(G)$ the plane map induced by $G$.
That is, the vertices of $M(G)$ are the vertices and crossing points in $G$,
and the edges of $M(G)$ are the crossing-free segments of the edges of $G$
(where each such edge-segment connects two vertices of $M(G)$).
We will use capital letters to denote the vertices of $G$, and small letters to denote
crossing points in $G$ (that are vertices in $M(G)$).
%When the context is clear, we will sometimes refer to an edge of $G$ by its endpoints, e.g., $(A,B)$,
%although $G$ may have several parallel edges between the same pair of vertices.
An edge of $M(G)$ will usually be denoted by its endpoints, e.g., $xy$,
whereas for an edge of $G$ we will use the standard notation, e.g., $(A,B)$.\footnote{Unless stated otherwise, every edge is not oriented.}

We first show that we may assume that the vertex-connectivity of $M(G)$ is at least $2$.

\begin{prop}\label{prop:2-connected}
If $M(G)$ is not $2$-connected, then $G$ has at most $6n-12$ edges.
\end{prop}

\begin{proof}
Assume that $M(G)$ has a vertex $x$ such that $M(G) \setminus \{x\}$ is not connected.
The vertex $x$ is either a vertex of $G$ or a crossing point of two of its edges.
Suppose that $x$ is vertex of $G$. Then, $G \setminus \{x\}$ is also not connected.
Let $G_1,\ldots,G_k$ be the connected components of $G \setminus \{x\}$,
let $G'$ be the topological graph induced by $V(G_1) \cup \{x\}$
and let $G''$ be the topological graph induced by $V(G_2) \cup \ldots \cup V(G_k) \cup \{x\}$.
Note that $6 \leq |V(G')|,|V(G'')| < n$, since $\delta(G) \geq 7$.
Therefore, it follows from the induction hypothesis that $|E(G)| \leq 6|V(G')|-12 + 6|V(G'')|-12 = 6(n+1)-24 < 6n-12$.

Suppose now that $x$ is a crossing point of two edges $e_1$ and $e_2$.
Let $\hat{G}$ be the topological graph we obtain by adding $x$ as a vertex to $G$.
Therefore, $|V(\hat{G})|=n+1$ and $|E(\hat{G})|=|E(G)|+2$.
Let $G_1,\ldots,G_k$ be the connected components of $\hat{G} \setminus \{x\}$,
let $G'$ be the topological graph induced by $V(G_1) \cup \{x\}$
and let $G''$ be the topological graph induced by $V(G_2) \cup \ldots \cup V(G_k) \cup \{x\}$.
Note that $6 \leq |V(G')|,|V(G'')| < n$, since $\delta(G) \geq 7$.
Therefore, it follows from the induction hypothesis that $|E(G)| \leq 6|V(G')|-12 + 6|V(G'')|-12 - 2 = 6(n+2)-26 < 6n-12$.
\end{proof}

In light of Proposition~\ref{prop:2-connected}, we may assume henceforth that $M(G)$ is $2$-connected.
The \emph{boundary} of a face $f$ in $M(G)$ consists of all the edges of $M(G)$ that are incident to $f$.
Since $M(G)$ is $2$-connected, the boundary of every face in $M(G)$ is a simple cycle.
Thus, we can define the \emph{size} of a face $f$, $|f|$, as the number of edges of $M(G)$ on its boundary.
We will keep this fact in mind when analyzing some cases later.

\begin{obs}\label{obs:simple-cycle}
The boundary of every face in $M(G)$ is a simple cycle.
\end{obs}

We use the \emph{Discharging Method} to prove Theorem~\ref{thm:4crossings}.
This technique, that was introduced and used successfully for proving structural properties of planar graphs
(most notably, in the proof of the Four Color Theorem~\cite{AH77}),
has recently proven to be a useful tool also for solving 
several problems in geometric graph theory~\cite{Ac09,AT07,AFKMT12,KP11,RT08}.
In our case, we begin by assigning a \emph{charge} to every face of the planar map $M(G)$ such that the total charge is $4n-8$.
Then, we redistribute the charge in several steps such that eventually the charge of every face is non-negative
and the charge of every vertex $A \in V(G)$ is $\deg(A)/3$.
Hence, $2|E(G)|/3 = \sum_{A \in V(G)} \deg(A)/3 \leq 4n-8$ and we get the desired bound on $|E(G)|$.
Next we describe the proof in details. 
Unfortunately, as it often happens when using the discharging method,
the proof requires considering many cases and subcases.

\paragraph{Charging.}
Let $V'$, $E'$, and $F'$ denote the vertex, edge, and face sets of $M(G)$, respectively.
For a face $f \in F'$ we denote by $V(f)$ the set of vertices of $G$ that are incident to $f$.
It is easy to see that $\sum_{f \in F'} |V(f)| = \sum_{A \in V(G)} \deg(A)$ and that
$\sum_{f \in F'} |f| = 2|E'| = \sum_{u \in V'} \deg(u)$.
Note also that every vertex in $V' \setminus V(G)$ is a crossing point in $G$
and therefore its degree in $M(G)$ is four. Hence,
$$
    \sum_{f \in F'} |V(f)| = \sum_{A \in V(G)} \deg(A) =
        \sum_{u \in V'} \deg(u) - \sum_{u \in V' \setminus V(G)}\deg(u) =
    2|E'| - 4\left(|V'|-n\right).
$$
Assigning every face $f \in F'$ a charge of $|f|+|V(f)|-4$,
we get that total charge over all faces is
$$
    \sum_{f \in F'}\left(|f|+|V(f)|-4\right) = 2|E'|+ 2|E'| - 4\left(|V'|-n\right) - 4|F'| = 4n-8,
$$
where the last equality follows from Euler's Polyhedral Formula by which $|V'|+|F'|-|E'| = 2$
(recall that $M(G)$ is connected).

\paragraph{Discharging.}
We will redistribute the charges in several steps.
We denote by $ch_i(x)$ the charge of an element $x$ (either a face in $F'$ or a vertex in $V(G)$) after the $i$th step,
where $ch_0(\cdot)$ represents the initial charge function.
We will use the terms \emph{triangles}, \emph{quadrilaterals} and \emph{pentagons}
to refer to faces of size $3$, $4$ and $5$, respectively.
An integer before the name of a face denotes the number of original vertices it is incident to.
For example, a $2$-triangle is a face of size $3$ that is incident to $2$ original vertices.
% It follows from our choice of $G$ (using Lemma~\ref{lem:good}) that if $V(f)>1$ for a face $f$,
% then $f$ is a triangle.
Since $G$ is a simple topological graph, there are no faces of size $2$ in $F'$.
Therefore, initially, the only faces with a negative charge are $0$-triangles.

In order to describe the way the charge of $0$-triangles (and later also of $1$-triangles) becomes non-negative,
we will need the following definitions.
Let $f$ be a face, let $e$ be one of its edges, and let $f'$ be the other face that shares $e$ with $f$.
We say that $f'$ is the \emph{immediate neighbor} of $f$ at $e$.
Note that $f' \neq f$ since $M(G)$ is $2$-connected.

\smallskip\noindent\textbf{Wedge-neighbors.}
Let $f_0$ be a triangle in $M(G)$ and let $x_1$ and $y_1$ be two vertices of $f_0$ that are crossing points in $G$.
Denote by $e_x$ (resp., $e_y$) the edge of $G$ that contains $x_1$ (resp., $y_1$) and does not contain $y_1$ (resp., $x_1$).
Note that $e_x$ and $e_y$ intersect at the other vertex of $f_0$.
Let $f_1$ be the immediate neighbor of $f_0$ at $x_1y_1$.
For $i \geq 1$, if $f_{i}$ is a $0$-quadrilateral,
then denote by $x_{i+1}y_{i+1}$ the edge of $M(G)$ opposite to $x_iy_i$ in $f_{i}$,
such that $e_x$ contains $x_{i+1}$ and $e_y$ contains $y_{i+1}$,
and let $f_{i+1}$ be the immediate neighbor of $f_i$ at $x_{i+1}y_{i+1}$.
Observe that $f_i \neq f_j$ for $i < j$, for otherwise $x_j$ coincides with one of $x_i$ and $x_{i+1}$
(which implies that $e_x$ crosses itself) or with one of $y_i$ and $y_{i+1}$ (which implies that $e_x$ and $e_y$ intersect more than once).
Let $j$ be the maximum index for which $f_j$ is defined. 
Note that $j \leq 4$ since the existence of $f_5$ would imply that each of $e_x$ and $e_y$ are crossed more than four times.
The \emph{wedge} of $f_0$ consists of $\bigcup_{i=0}^{j-1} f_i$ and we call $f_j$ the \emph{wedge-neighbor} of $f_0$ at $x_1y_1$
(note that $f_j$ is uniquely defined).
We also say that $f_0$ is the wedge-neighbor of $f_j$ at $x_jy_j$
(see Figure~\ref{fig:first-steps}).

Observe that since the relations being an immediate neighbor at a certain edge of $M(G)$
and being an opposite edge in a $0$-quadrilateral are both one-to-one,
it follows that indeed there cannot be another triangle but $f_0$ that is a wedge-neighbor of $f_j$ at $x_jy_j$.
Note also that since $e_x$ and $e_y$ already intersect at a vertex of $f_0$, and by definition $f_j$ cannot be a $0$-quadrangle,
either $|f_j| \geq 5$ or $|f_j|=4$ and $|V(f_j)| \geq 1$.
Let us summarize these observations for future reference.

\begin{obs}\label{obs:one-wedge-neighbor}
Let $f$ be a face and let $e$ be one of its edges.
Then there is at most one triangle $t$ such that $t$ is a wedge-neighbor of $f$ at $e$.
If such a triangle exists, then either $|f| \geq 5$ or $|f|=4$ and $|V(f)| \geq 1$.
\end{obs}

\noindent\textbf{Step 1: Charging $0$-triangles.}
Let $f_0$ be a $0$-triangle, let $e_1=x_1y_1$ be one of its edges, and let $f_j$ be
its wedge-neighbor at $x_1y_1$ as defined above.
Then $f_j$ contributes $1/3$ units of charge to $f_0$ through $x_jy_j$ (see Figure~\ref{fig:step1}).
\begin{figure}[ht]
    \centering
    \subfigure[Step~1: The $0$-triangle $f_0$ receives $1/3$ units of charge from its wedge-neighbor $f_2$.]{\label{fig:step1}
    {\includegraphics[width=4cm]{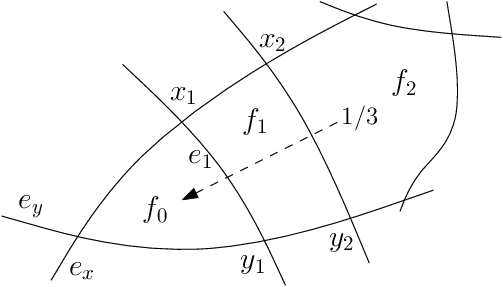}}}
        \hspace{5mm}
    \subfigure[Step~2: The vertex $A \in V(G)$ receives $1/3$ units of charge from each face that is incident to it.]{\label{fig:step2}
    {\includegraphics[width=4cm]{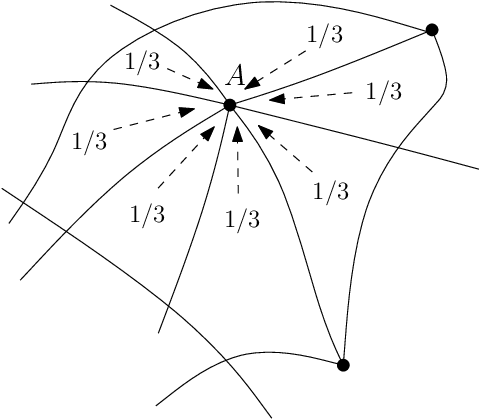}}}
        \hspace{5mm}
    \subfigure[Step~3: If the neighbors of the $1$-triangle $f_0$, $g_1$ and $g_2$, are $1$-triangles, then $f_0$ receives $1/3$ units of charge from its wedge-neighbor $f_3$ through $x_3y_3$.]{\label{fig:step3}
    {\includegraphics[width=4cm]{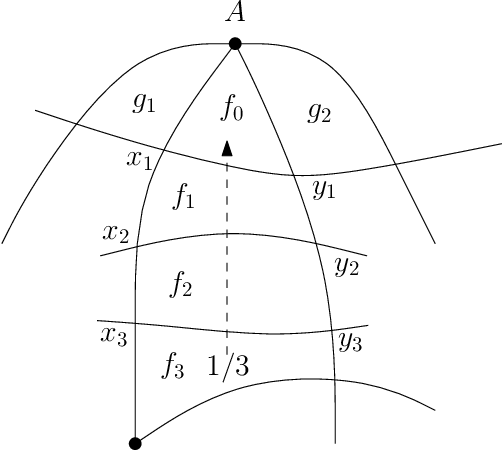}}}
	\caption{The first three discharging steps.}
	\label{fig:first-steps}
\end{figure}

In a similar way $f_0$ obtains $2/3$ units of charge from the two other `directions' that correspond to its two other edges.
Thus, after the first discharging step the charge of every $0$-triangle is zero. $\hfill\leftrightsquigarrow$

% Note that in at most one of the three `directions' in which a $0$-triangle $t$ `seeks' charge
% it can encounter a $0$-quadrilateral.
% Indeed, two neighboring $0$-quadrilaterals to $t$ would imply that the third neighboring face
% has two original vertices and size greater than three and hence is not good.

\medskip
Recall that according to our plan, the charge of every original vertex should be
one third of its degree. The next discharging step takes care of this.

\medskip
\noindent\textbf{Step 2: Charging the vertices of $G$.}
In this step every vertex of $G$ takes $1/3$ units of charge from each face it is incident to
(see Figure~\ref{fig:step2}).$\hfill\leftrightsquigarrow$

\medskip

Note that after Step~2, the charge of every $1$-triangle is $-1/3$.
In the following three steps, every $1$-triangle will obtain $1/3$ units of charge
from neighboring faces or its wedge-neighbor.
Let $f_0$ be a $1$-triangle and let $A \in V(G)$ be the vertex of $G$ that is incident to $f_0$.
Let $g_1$ and $g_2$ be the two faces that share an edge of $M(G)$ with $f_0$ and are also incident to $A$.
We call $g_1$ and $g_2$ the \emph{neighbors} of $f_0$ (see Figure~\ref{fig:step3} for an example).
Note that it is impossible that $g_1=g_2$, for otherwise $\deg(A)=2 < 7 \leq \delta(G)$ or $A$ is cut vertex in $M(G)$.
If both $g_1$ and $g_2$ are $1$-triangles, then $f_0$ must obtain the missing charge from its wedge-neighbor.

\medskip
\noindent\textbf{Step 3: Charging $1$-triangles with `poor' neighbors.} % or rich wedge-neighbors.}
If $f_0$ is a $1$-triangle whose two neighbors are $1$-triangles,
then the wedge-neighbor of $f_0$ contributes $1/3$ units of charge to $f_0$
through the edge of $M(G)$ that it shares with the wedge of $f_0$ (see Figure~\ref{fig:step3}).$\hfill\leftrightsquigarrow$
\medskip

From Observation~\ref{obs:one-wedge-neighbor} and the definition of Steps~1 and~3, we have:

\begin{obs}\label{obs:steps1and3}
Let $f$ be a face and let $e$ be one of its edges.
Then $f$ contributes charge at most once through $e$ during Steps~1 and~3.
Moreover, if $f$ contributes charge through $e$ in (either) Step~1 or~3, then the endpoints of $e$ are crossing points in $G$.
\end{obs}

The following facts will also be useful later on.

\begin{obs}\label{obs:extreme}
Let $f$ be a face, let $q$ be a vertex of $f$ that is a crossing point in $G$,
and let $v_1$, $v_2$, $v_3$ and $v_4$ be the neighbors of $q$ in $M(G)$
such that $v_1$ and $v_4$ are vertices of $f$.
Then if $v_2$ and $v_3$ are vertices of $G$, then $f$ does not contribute charge
through $v_1q$ and $v_4q$ neither in Step~1 nor in Step~3
(see Figure~\ref{fig:extreme} for an illustration).
\end{obs}

\begin{figure}[ht]
    \centering
    \subfigure[If $v_2,v_3 \in V(G)$ and $v_2q, v_3q \in E'$, then $f$ does not contribute charge through $v_1q$ and $v_4q$ in Steps~1 and~3.]{\label{fig:extreme}
    {\includegraphics[width=4cm]{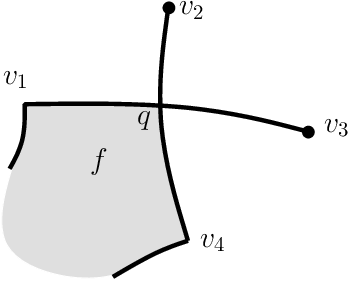}}}
    \hspace{10mm}
    \subfigure[$f$ does not contribute charge through $e$, $e_1$ and $e_2$ in Steps~1 and~3.]{\label{fig:1-triangle-neighbor}
    {\includegraphics[width=4cm]{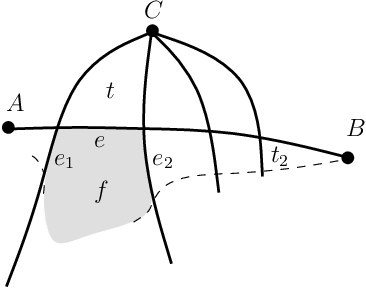}}}
	\caption{Certain cases where a face does not contribute charge in Steps~1 and~3.}
	\label{fig:extreme-and-cor}
\end{figure}

\begin{prop}\label{prop:1-triangle-neighbor}
Let $f$ be a face, let $e$ be one of its edges and let $t$ be a $1$-triangle that is the wedge-neighbor of $f$ at $e$.
If $t$ is a neighbor of a $1$-triangle that receives charge in Step~3,
then $f$ does not contribute charge through $e$ in Steps~1 and~3.
Furthermore, if the wedge of $t$ contains no $0$-quadrilaterals, then $f$ does not contribute charge in Steps~1 and~3 through
each of its edges that are incident to $e$
(see Figure~\ref{fig:1-triangle-neighbor} for an illustration).
\end{prop}

\begin{proof}
Since $t$ is a $1$-triangle $f$ cannot contribute charge through $e$ in Step~1.
Denote by $(A,B)$ the edge of $G$ that contains $e$ and let $C$ be the vertex of $G$ that is incident to $t$ (note that $C \neq A,B$ for otherwise $(A,B)$ crosses itself).
Since $t$ is a neighbor of a $1$-triangle that receives charge in Step~3, $(A,B)$ already contains four crossings
and therefore the other neighbor of $t$ must be incident to two original vertices.
Thus, $f$ does not contribute charge through $e$ in Step~3.

Assume that the wedge of $t$ contains no $0$-quadrilaterals and let $e_1$ and $e_2$ be
the edges of $f$ that are incident to $e$, as depicted in Figure~\ref{fig:1-triangle-neighbor}.
By Observation~\ref{obs:extreme} $f$ does not contribute charge in Steps~1 and~3 through $e_1$.
Suppose that $f$ contributes charge to a triangle $t_2$ through $e_2$ in Step~1 or~3.
Note that the wedge of this triangle must contain two $0$-quadrilaterals,
for otherwise $t_2$ would be a wedge-neighbor of a $1$-triangle (either the neighbor of $t$ or
a neighbor of this neighbor), which is impossible by Observation~\ref{obs:one-wedge-neighbor}.
If $t_2$ is a $0$-triangle, then $(A,B)$ would have five crossings.
Therefore, $t_2$ must be a $1$-triangle. 
However, one neighbor of this $1$-triangle is incident to $C$ and one endpoint of $(A,B)$
and this implies that $t_2$ does not receive charge in Step~3.
\end{proof}

\begin{prop}\label{prop:e_1-e_2}
Let $f$ be a face that contributes charge in Step~3 to a $1$-triangle $t$ through one of its edges $e$ such that $e$ is also an edge of $t$
(that is, the wedge of $t$ contains no $0$-quadrilaterals).
Then $f$ does not contribute charge in Step~1 or~3 through neither of its two edges that are incident to $e$. 
\end{prop}

\begin{proof}
Let $(A,B)$ be the edge of $G$ 
that contains $e$ and let $e'$ be an edge of $f$ that is incident to $e$.
Denote by $v$ the vertex of $f$ that is incident to both $e$ and $e'$.
Note that $(A,B)$ contains four crossing points: the endpoints of $e$
and two crossing points, one on each side of $e$ on $(A,B)$, since the neighbors of $t$ must be $1$-triangles.

Suppose that $f$ contributes charge through $e'$ to a triangle $t'$ in Step~1 or~3.
Let $g$ be the $1$-triangle that is a neighbor of $t$ and is incident to $v$,
and let $f'_1$ be the other face but $f$ that is incident to $e'$ (see Figure~\ref{fig:e_1-e_2}).
Observe that $f'_1$ and $g$ share an edge and therefore $f'_1$ cannot be a $1$-triangle or a $0$-triangle 
(the latter implies that two edges of $G$ intersect twice).
If the wedge of $t'$ contains more than one $0$-quadrilateral, then $(A,B)$ has more than four crossings.
Thus, $f'_1$ must be a $0$-quadrilateral that shares an edge with the triangle $t'$ to which $f$ contributes charge through $e'$.
If $t'$ is a $1$-triangle (refer to Figure~\ref{fig:e_1-e_2a}), then one of its neighbors shares edges with $t'$ and $g$,
which implies that this neighbor is incident to two original vertices (of $t$ and $t'$)
and is therefore not a $1$-triangle (recall that if $t'$ gets charge in Step~3, then its neighbors must be $1$-triangles).
Otherwise, if $t'$ is a $0$-triangle, then the edge $(A,B)$ has more than four crossings (see Figure~\ref{fig:e_1-e_2b}).
This implies that $f$ does not contribute charge through $e'$ in
Steps~1 and~3.
\begin{figure}[ht]
    \centering
    \subfigure[If $f$ contributes charge through $e'$ in Step~3, then $t'$ should have $1$-triangles for neighbors, however $g'$ is incident to two original vertices.]{\label{fig:e_1-e_2a}
    {\includegraphics[width=6cm]{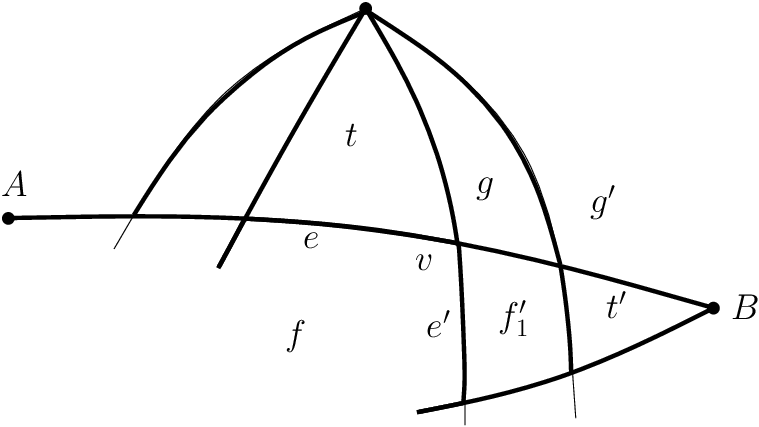}}}
    \hspace{5mm}
    \subfigure[If $f$ contributes charge through $e'$ in Step~1, then $(A,B)$ has five crossings.]{\label{fig:e_1-e_2b}
    {\includegraphics[width=6cm]{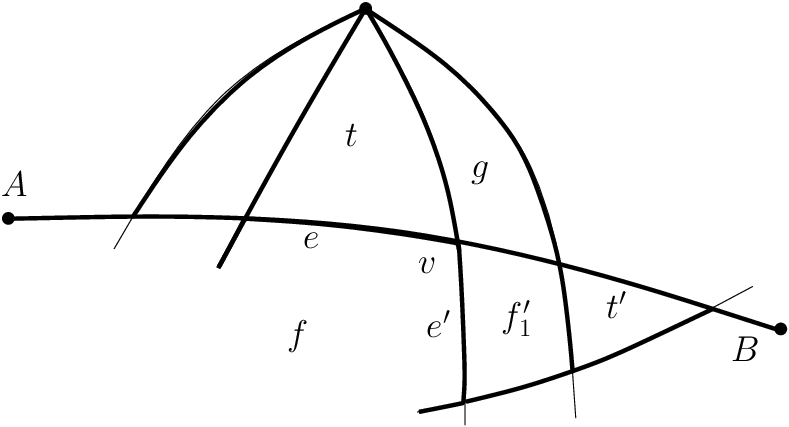}}}
	\caption{If $f$ contributes charge to $t$ through $e$ in Step~3 and $e$ is an edge of $t$,
then $f$ cannot contribute charge to a triangle $t'$ through an edge $e'$ that is incident to $e$.}
	\label{fig:e_1-e_2}
\end{figure}
\end{proof}

Let us analyze the charge of a face $f$ after the first three discharging steps.
It follows from Observation~\ref{obs:steps1and3} and the discharging steps
that $ch_3(f) \geq 2|f|/3+2|V(f)|/3-4$.
Therefore $ch_3(f)\geq 0$ if $|f| \geq 6$.

\begin{obs}\label{obs:ch_3}
Let $f$ be a face in $M(G$). Then
\begin{packed_item}
	\vspace{-3mm}
	\item if $|f| \geq 6$, then $ch_3(f) \geq 2|V(f)|/3$;
	\item if $f$ is a $0$-pentagon, then $ch_3(f) \geq -2/3$;	
	\item if $f$ is a $k$-pentagon such that $0<k<5$, then $ch_3(f) \geq k-1/3$;
	\item if $f$ is a $5$-pentagon, then $ch_3(f) = 4\frac{1}{3}$;
% 	\item if $f$ is a $0$-pentagon then $ch_2(f) \geq 1/3$;
	\item if $f$ is a $0$-quadrilateral, then  $ch_3(f)=0$;
	\item if $f$ is a $1$-quadrilateral, then $ch_3(f) \geq 0$;
	\item if $f$ is a $k$-quadrilateral such that $k>1$, then $ch_3(f) \geq k-4/3$;
	\item if $f$ is a $0$-triangle, then $ch_3(f)=0$;
	\item if $f$ is a $1$-triangle, then $ch_3(f)=-1/3$ or $ch_3(f)=0$;
	\item if $f$ is a $2$-triangle, then $ch_3(f)=1/3$; and
	\item if $f$ is a $3$-triangle, then $ch_3(f)=1$.
\end{packed_item}
\end{obs}

\begin{proof}
Let us consider, for example, the case that $f$ is a quadrilateral (the other cases are similar).
If $f$ is a $0$-quadrilateral, then it does not contribute any charge and thus $ch_3(f)=ch_0(f)=0$.
If $f$ is a $1$-quadrilateral, then it has two edges whose endpoints are crossing points, and
thus it follows from Observation~\ref{obs:steps1and3} that $f$ contributes at most $2/3$ units of charge in Steps~1 and~3.
Since $f$ contributes $1/3$ units of charge in Step~2, we have $ch_3(f) \geq 4+1-4-2\cdot\frac{1}{3}-\frac{1}{3} = 0$.
If $f$ is a $2$-quadrilateral, then it has at most one edge whose endpoints are crossing points,
and therefore, $ch_3(f) \geq 4+2-4-\frac{1}{3}-2\cdot\frac{1}{3} = 1$.
If $f$ is a $3$-quadrilateral, then it has no edge whose endpoints are crossing points,
and therefore, $ch_3(f) = 4+3-4-3\cdot\frac{1}{3} = 2$.
Finally, if $f$ is a $4$-quadrilateral, then $ch_3(f) = 4+4-4-4\cdot\frac{1}{3} = 2\frac{2}{3}$.
\end{proof}

% Note that we have not mentioned $0$-pentagons.
% Showing that $0$-pentagons end up with a non-negative charge will be the most challenging task,
% and therefore we postpone the analysis of their charge until after all the discharging steps are described.

By Observation~\ref{obs:ch_3} the only faces with a negative charge after Step~3 are $1$-triangles and $0$-pentagons.
%(we will see later in Proposition~\ref{prop:ch_1-0-pentagon} that the charge of $0$-pentagons after Step~1 is non-negative).
In the next two steps we redistribute the charges such that the charge of every $1$-triangle becomes zero.

Let $f$ be a $1$-triangle with a negative charge after Step~3. 
The missing charge of $f$ will come from either both of its neighbors or from one neighbor and the wedge-neighbor of $f$.
Next we show that if $f$ has a negative charge, then one of its neighbors has a positive charge.

\begin{prop}\label{prop:positive-neighbor}
Let $f$ be a $1$-triangle and let $g_1$ and $g_2$ be its neighbors.
If $ch_3(f)<0$, then $ch_3(g_1)>0$ or $ch_3(g_2)>0$.
\end{prop}

\begin{proof}
It follows from Observation~\ref{obs:ch_3} that if a face is incident to a vertex of $G$ and its charge after Step~3 is non-positive,
then this face is either a $1$-triangle or a $1$-quadrilateral.
Assume without loss of generality that $|g_1| \geq |g_2|$.
It follows that if $g_1$ is a $1$-triangle, then so is $g_2$, however, in this case the charge of $f$ should become zero in Step~3.
Therefore, $g_1$ must be a $1$-quadrilateral and $g_2$ is either a $1$-triangle or a $1$-quadrilateral.

Let $A$ be the vertex of $G$ that is incident to $f$, let $e$ be the edge of $f$ that is opposite to $A$
and let $(B,C)$ be the edge of $G$ that contains $e$. 
Let $e_1$ be the edge of $g_1$ that is contained in $(B,C)$ and let $e_2$ be its other edge that is not incident to $A$ (see Figure~\ref{fig:ch_3-neighbor}).
\begin{figure}[ht]
    \centering
	\includegraphics[width=6cm]{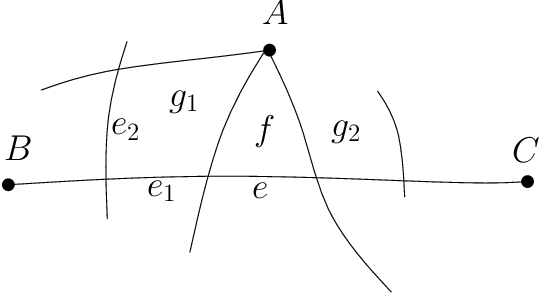}
	\caption{An illustration for the proof of Proposition~\ref{prop:positive-neighbor}: $g_2$ is either a $1$-quadrilateral or a $1$-triangle. 
			It is impossible that $g_1$ contributes charge through $e_2$ in Step~1, or through $e_2$ in Step~3 and through $e_1$ in Step~1 or ~3.}
	\label{fig:ch_3-neighbor}
\end{figure}
Observe that $ch_3(g_1) \in \{0,1/3,2/3\}$ and that if $ch_3(g_1)=0$, then $g_1$ must contribute charge through $e_2$ in either Step~1 or Step~3.
However, since $(B,C)$ already has four crossings among the endpoints of the edges of $g_1$, $f$, and $g_2$ that it contains,
it is impossible that $g_1$ contributes charge to a $0$-triangle through $e_2$ in Step~1.
Moreover, if $g_1$ contributes charge to a $1$-triangle through $e_2$ in Step~3, 
then $e_2$ must be an edge of this triangle, for otherwise $(B,C)$ would have more than four crossings.
But then, it follows from Proposition~\ref{prop:e_1-e_2} that $g_1$ does not contribute charge through $e_1$ in Step~1 or~3.
Therefore, $ch_3(g_1) \geq 1/3$.
\end{proof}

\medskip
\noindent\textbf{Step 4: Charging $1$-triangles with a positive neighbor.}
Let $t$ be a $1$-triangle such that $ch_3(t)<0$, let $g$ be a neighbor of $t$ such that $ch_3(g)>0$,
and let $g'$ be the other neighbor of $t$.
Then $g$ contributes $1/6$ units of charge to $t$ through the edge of $M(G)$ that they share if:
(a)~$g$ is not a $1$-quadrilateral or $ch_3(g) > 1/3$; or 
(b)~$g$ is a $1$-quadrilateral, $ch_3(g)=1/3$, and $g'$ is either a $1$-triangle or a $1$-quadrilateral with $ch_3(g')=1/3$. 
See Figure~\ref{fig:steps45} for an illustration.$\hfill\leftrightsquigarrow$

\begin{figure}[ht]
    \centering
    \subfigure[Steps~4 and 5: In Step~4 the $1$-triangle $t$ receives $1/6$ units of charge from its neighbor $g$.
				In Step~5 $t$ receives $1/6$ units of charge from its wedge-neighbor $f$.]{\label{fig:steps45}
    {\includegraphics[width=5cm]{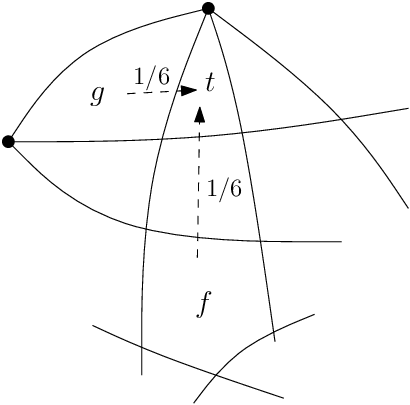}}}
        \hspace{5mm}
    \subfigure[Step~6: The $0$-pentagon $f$ sends $2/3$ units of charge in Step~1, $1/3$ unit of charge in Step~3,
				and $1/6$ units of charge in Step~5, thus $ch_5(f)<0$. Its vertex-neighbor $f'$ has $ch_5(f') = 1/6 \geq 0$.
				In Step~6 $f'$ sends $1/6$ units of charge to $f$ (its only vertex-neighbor that is a $0$-pentagon with a negative charge).]{\label{fig:step6}
    {\includegraphics[width=6cm]{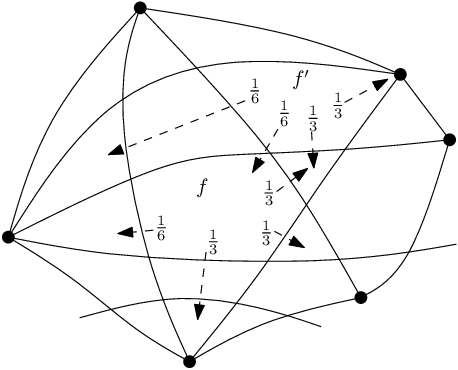}}}
	\caption{The last three discharging steps.}
	\label{fig:last-steps}
\end{figure}

\medskip

Considering Observation~\ref{obs:ch_3} it follows that if $f$ is a neighbor of a $1$-triangle $t$ and $f$ is not a $1$-triangle or a $1$-quadrilateral with
$ch_3(f)=1/3$, then $f$ contributes $1/6$ units of charge to $t$ in Step~4.
By Proposition~\ref{prop:positive-neighbor} it also follows that for every $1$-triangle $t$ we have $ch_4(t) \geq -1/6$.

% Note that it might happen that $g=g'$.
% In this case $|g| \geq 6$ and $g$ contributes $1/6$ units of charge
% to $f$ through each of the two edges that it shares with $f$.

\begin{prop}\label{prop:ch_4}
There is no face $f$ such that $ch_3(f) \geq 0$ and $ch_4(f) < 0$.
\end{prop}

\begin{proof}
Let $f$ be a face.
Note that $f$ may contribute charge in Step~4 through one of its edges, only if the
endpoints of this edge consist of one original vertex of $G$ and one crossing point.
Therefore $f$ cannot contribute charge through the same edge in Steps 1, 3, or 4,
and we only have to consider faces containing original vertices of $G$.

If $f$ is a $1$-triangle, then $ch_3(f) \leq 0$ and so it cannot contribute charge in Step~4.
If $f$ is a $2$-triangle, then $ch_3(f)=1/3$ and $f$ is the neighbor of at most two $1$-triangles and so $ch_4(f)\geq 0$.
If $f$ is a $3$-triangle then $f$ is not a neighbor of any $1$-triangle and so $ch_4(f)=ch_3(f)=1$.
If $f$ is a $1$-quadrilateral then it contributes charge (to at most two $1$-triangles) only if $ch_3(f) \geq 1/3$ and therefore $ch_4(f)\geq 0$.
If $|V(f)|>1$ or $|f|>4$, then it is easy to verify that the charge of $f$ remains positive.
\end{proof}

% \begin{prop}\label{prop:ch_4-1-triangle}
% If $f$ is a $1$-triangle, then $ch_4(f) \geq -1/6$.
%%Furthermore, if $ch_4(f) = -1/6$ then exactly one of the neighbors of $f$ is a $1$-triangle.
% \end{prop}

% \begin{proof}
% Let $f$ be a $1$-triangle and let $g_1$ and $g_2$ be its neighbors.
% If $ch_4(f) < -1/6$ it means that $f$ did not receive any charge from neither $g_1$ nor $g_2$ in Step~4.
% The only faces that contain an original vertex and may have a non-positive charge after Step~3 are $1$-triangles and $1$-quadrilaterals.
% Therefore, each of $g_1$ and $g_2$ is either a $1$-triangle or a $1$-quadrilateral.
% If both of them are $1$-triangles, then $f$ is charged in Step~3 and $ch_4(f)=ch_3(f)=0$.

% Suppose without loss of generality that $|g_1| \geq |g_2|$, and therefore $g_1$ is a $1$-quadrilateral.
% Considering Step~4, observe that $f$ receives charge from $g_1$ or $g_2$ 
% unless $ch_3(g_1)=0$ or $ch_3(g_1)=1/3$ and $g_2$ is a $1$-quadrilateral such that $ch_3(g_2)=0$.
% In the first case $|V(g_2)| \geq 2$ by Proposition~\ref{prop:ch_3-1-quad} and so $ch_3(g_2) > 0$.
% The second case is actually symmetric to the first case (switching the names of $g_1$ and $g_2$).
% Therefore, in each case at least one neighbor of $f$ contributes charge to $f$ in Step~4 and thus $ch_4(f) \geq -1/6$.
% \end{proof}

\begin{prop}\label{prop:ch_4-1-quad}
If $f$ is a $1$-quadrilateral and $ch_3(f)=1/3$, then $ch_4(f) \geq 1/6$. % contributes charge to at most one $1$-triangle in Step~4.
\end{prop}

\begin{proof}
Suppose that $f$ is a $1$-quadrilateral such that $V(f)=\{A\}$, $ch_3(f)=1/3$ and $f$ contributes
charge to two $1$-triangles $t_1$ and $t_2$ in Step~4.
Let $g_1$ and $g_2$ be the other neighbors of $t_1$ and $t_2$, respectively 
(it is impossible that $g_1=g_2$, for otherwise $\deg(A)<7 \leq \delta(G)$ or $A$ is a cut vertex).
Note that according to Step~4(b),
each of $g_1$ and $g_2$ must be either a $1$-triangle or a $1$-quadrilateral whose charge is $1/3$ after Step~3.
Observe also that it is impossible that $ch_2(f)<2/3$, that is, that $f$ contributes charge
to at least one $0$-triangle in Step~1.
Indeed, this would imply that an edge of $G$ that contains an edge of $f$ that is not incident to $A$
already has four crossings among the vertices of $f$
and the $0$- and $1$-triangles to which $f$ contributes charge.
It then follows that $|V(g_1)| \geq 2$ or $|V(g_2)| \geq 2$ (see Figure~\ref{fig:ch_4-1-quad-1} for an illustration).
\begin{figure}[ht]
    \centering
    \subfigure[If $ch_2(f)<2/3$ then $|V(g_1)| \geq 2$ (as in this figure) or $|V(g_2)| \geq 2$.]{\label{fig:ch_4-1-quad-1}
    {\includegraphics[width=4cm]{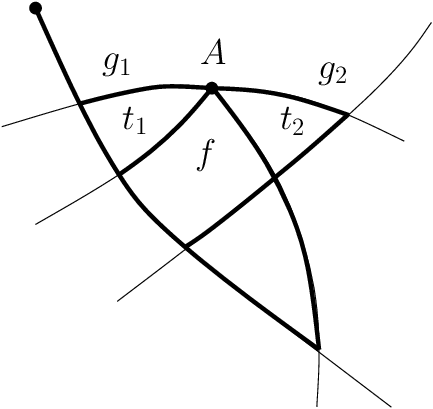}}}
	\hspace{10mm}
    \subfigure[If $ch_2(f)=2/3$ and $f$ contributes charge to $t'$ in Step~3, then
				both neighbors of $t'$ should be $1$-triangles. This implies that there are two parallel edges whose endpoints are $A$ and $B$.]{\label{fig:ch_4-1-quad-2}
    {\includegraphics[width=4cm]{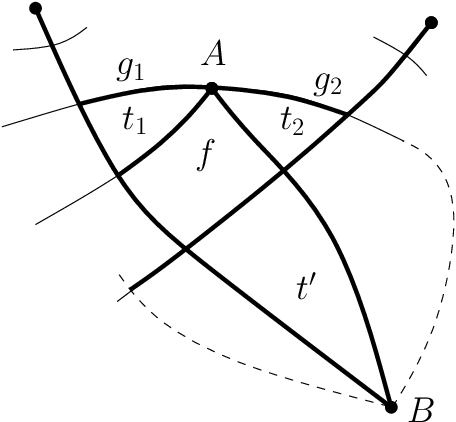}}}
	\caption{Illustrations for the proof of Proposition~\ref{prop:ch_4-1-quad}: There is no $1$-quadrilateral $f$ 
             such that $ch_3(f)=1/3$ and $f$ contributes charge to two $1$-triangles $t_1$ and $t_2$ in Step~4.}
	\label{fig:ch_4-1-quad}
\end{figure}

Therefore, assume that $ch_2(f)=2/3$ and denote by $t'$ the $1$-triangle to which $f$ has contributed charge in Step~3.
Suppose without loss of generality that $f$ contributes charge to $t'$ through the edge that is opposite to the edge
through which $f$ contributes charge to $t_1$ (see Figure~\ref{fig:ch_4-1-quad-2}).
Note that $t'$ must share an edge (of $M(G)$) with $f$, for otherwise, as before, $|V(g_1)| \geq 2$.
Let $B$ be the vertex of $G$ that is incident to $t'$.
Since $t'$ receives charge in Step~3, both of its neighbors are $1$-triangles.
However, this implies that there are two parallel edges whose endpoints are $A$ and $B$.
%(and also implies that the edge of $G$ that contains the edge opposite to $B$ in $t'$ has five crossings).
\end{proof}

\noindent\textbf{Step 5: Finish charging $1$-triangles.}
Let $t$ be a $1$-triangle, let $f$ be the wedge-neighbor of $t$ 
and let $e$ be the edge of $f$ at which $t$ is a wedge-neighbor of $f$.
If $ch_4(t)<0$ then $f$ contributes $1/6$ units of charge to $t$ through $e$ 
(see Figure~\ref{fig:steps45} for an illustration).$\hfill\leftrightsquigarrow$

\medskip

It follows from Observation~\ref{obs:one-wedge-neighbor} and the definition of Steps 1--5
that a face contributes charge at most once through each of its edges.

\begin{obs}\label{obs:once}
Let $f$ be a face and let $e$ be an edge of $f$.
Then $f$ contributes charge through $e$ at most once during Steps 1--5.
\end{obs}

\begin{prop}\label{prop:ch_5}
For every face $f$ that is not a $0$-pentagon we have $ch_5(f) \geq 0$.
\end{prop}

\begin{proof}
It follows from Observation~\ref{obs:ch_3}, Proposition~\ref{prop:positive-neighbor}, and Steps 3--5 that the charge
of every $1$-triangle is zero after the fifth discharging step.
Suppose that $f$ is a $1$-quadrilateral. Then $ch_3(f)$ is either $0$, $1/3$, or $2/3$.
If $ch_3(f)=0$ then it follows from Observation~\ref{obs:once} that $f$ does not contribute charge in Steps~4 and~5, and therefore $ch_5(f)=0$.
If $ch_3(f)=1/3$ then it follows from Observation~\ref{obs:once} and Proposition~\ref{prop:ch_4-1-quad} that $ch_4(f) \geq 1/6$ and so $ch_5(f) \geq 0$.
If $ch_3(f)=2/3$ then clearly $ch_5(f) \geq 0$.
It is not hard to see, recalling Observations~\ref{obs:ch_3} and~\ref{obs:once}, 
that the charge of every other face but a $0$-pentagon cannot be negative.
\end{proof}

Let $x$ be a crossing point in $G$ and let $f_1$, $f_2$, $f_3$ and $f_4$ be the four faces that are incident to $x$,
listed in their clockwise order around $x$.
Note that these faces are distinct, since $M(G)$ is $2$-connected.
We say that $f_1$ and $f_3$ (resp., $f_2$ and $f_4$) are \emph{vertex-neighbors} at $x$.
For a face $f$ such that $ch_5(f) \geq 0$, we denote by $\mathcal{P}(f)$ the set of $0$-pentagons $f'$ such that $ch_5(f') < 0$ and $f'$ is a vertex-neighbor of $f$.
We also denote by $\mathcal{P}'(f)$ the set of vertices of $f$ at which it is a vertex-neighbor of $0$-pentagons with a negative charge after the fifth discharging step.

\medskip
\noindent\textbf{Step 6: Charging $0$-pentagons.}
For every face $f$ such that $ch_5(f) \geq 0$,
if $\mathcal{P}(f) \neq \emptyset$, then in the sixth discharging step $f$ sends $ch_5(f)/|\mathcal{P}(f)|$ units of charge 
to every $0$-pentagon in $\mathcal{P}(f)$ through the vertex by which they are vertex-neighbors.
See Figure~\ref{fig:step6} for an illustration. $\hfill\leftrightsquigarrow$

% \paragraph{Remark:} If a $0$-pentagon $f'$ with $ch_5(f') < 0$ is a vertex-neighbor 
% of a face $f$ with $ch_5(f)>0$ at $k$ vertices, then $f'$ gets $k \cdot ch_5(f)/|B(f)|$ units of charge from $f$
% (that is, $ch_5(f)/|B(f)|$ units of charge through each vertex by which they are vertex-neighbors).
%\medskip
% \begin{prop}\label{prop:step6}
% Let $f$ be a face, let $v_1v_2$ be an edge of $f$, and let $f_i$
% be the face of $M(G)$ such that $f \cap f_i = v_i$, for $i=1,2$.
% If $\{f_1,f_2\} \subseteq B(f)$ then $f$ does not contribute charge through $v_1v_2$.
% \end{prop}

\medskip
We summarize all the discharging steps in Table~1. %\ref{tab:steps}.

\begin{table}[h] \label{tab:steps}
\begin{center}
\begin{tabularx}{\textwidth}{|l|l|X|} \hline
    Step & Charging & Details \\ \hline\hline
    1 & $0$-triangles & Each $0$-triangle gets $1/3$ units of charge from each of its three wedge-neighbors. \\ \hline
    2 & Vertices of $G$ & Each vertex of $G$ gets $1/3$ units of charge from each face of $M(G)$ it is incident to.  \\ \hline
    3 & $1$-triangles & Each $1$-triangle whose neighbors are $1$-triangles gets $1/3$ units of charge from its wedge-neighbor.  \\ \hline
    4 & $1$-triangles & A face $g$ with a positive charge that is a neighbor of a $1$-triangle $t$ with a negative charge sends $1/6$ units of charge to $t$ if:
						(a)~$g$ is not a $1$-quadrilateral or $ch_3(g) > 1/3$; or 
						(b)~$g$ is a $1$-quadrilateral with $ch_3(g)=1/3$ and the other neighbor of $t$ is either a $1$-triangle or a $1$-quadrilateral whose charge is $1/3$. \\ \hline
    5 & $1$-triangles & Each $1$-triangle with a negative charge gets $1/6$ units of charge from its wedge-neighbor.  \\ \hline
    6 & $0$-pentagons & Each face with a non-negative charge distributes it uniformly to the $0$-pentagons with a negative charge among its vertex-neighbors.  \\ \hline
\end{tabularx}
  \caption{Summary of the discharging steps.}
\end{center}
\end{table}

\begin{prop}\label{prop:step6}
Let $f$ be a face such that $|f| \geq 4$, $|V(f)|>0$ and $|\mathcal{P}(f)| > 0$.
If $|V(f)|=1$, then each $0$-pentagon in $\mathcal{P}(f)$ receives at least $\frac{2|f|/3-3}{|\mathcal{P}(f)|}$ units of charge from $f$ in Step~6.
If $|V(f)|\geq 2$, then each $0$-pentagon in $\mathcal{P}(f)$ receives at least $\frac{2|f|/3-2}{|\mathcal{P}(f)|} \geq 1/3$ units of charge from $f$ in Step~6.
\end{prop}

\begin{proof}
$f$ is not a $0$-pentagon since $|V(f)|>0$, and therefore by Proposition~\ref{prop:ch_5} it has
a non-negative charge after Step~5 which is distributed among the $0$-pentagons in $\mathcal{P}(f)$.
Since by replacing (for the sake of charge calculation) an original
vertex with a crossing point decreases the charge of a face,
we may assume that $f$ has at most two original vertices.

Suppose that $V(f) = \{A\}$.
Then $f$ may contribute at most $1/6$ units of charge through each of its two edges that
are incident to $A$, and at most $1/3$ units of charge through every other edge.
Therefore $ch_5(f) \geq |f|+1-4 - 1/3 - 2/6 - (|f|-2)/3 = 2|f|/3-3$.

Suppose now that $V(f) = \{A,B\}$.
If $A$ and $B$ are consecutive on the boundary of $f$,
then $f$ does not contribute charge through the edge $(A,B)$,
contributes at most $1/6$ units of charge through two edges,
and at most $1/3$ units of charge through every other edge.
Therefore, $ch_5(f) \geq |f|+2-4 - 2/3 -2/6 - (|f|-3)/3 = 2|f|/3-2$. 

If $A$ and $B$ are not consecutive on the boundary of $f$,
then $f$ contributes at most $1/6$ units of charge through four edges,
and at most $1/3$ units of charge through every other edge.
Therefore, $ch_5(f) \geq |f|+2-4 - 2/3 -4/6 - (|f|-4)/3 = 2|f|/3-2$.
Note that $A,B \notin \mathcal{P}'(f)$, thus we have
$\frac{2|f|/3-2}{|\mathcal{P}(f)|} \geq \frac{2|f|/3-2}{|f|-2} \geq 1/3$, for $|f| \geq 4$.
\end{proof}

% \begin{cor}\label{cor:step6-1}
% If $f$ is a face such that $|V(f)| = 1$, $|f| \geq 5$, and $f$ contributes charge in Step~6 to at most $|f|-4$ pentagons,
% then each $0$-pentagon in $B(f)$ receives at least $1/3$ units of charge from $f$.
% \end{cor}

%\begin{cor}\label{cor:step6-2}
%Let $f$ be a $0$-pentagon with $ch_5(f)<0$ and let $f'$ be a %vertex-neighbor of $f$ such that $|V(f')| \geq 2$ and $|f'| \geq 4$.
%Then $f$ receives at least $1/3$ units of charge from $f'$ in Step~6.
%\end{cor}

% \begin{cor}\label{cor:step6}
% Let $f$ be a face that contributes charge in Step~6. Then:
% \begin{packed_enum}
% \item if $|V(f)| \geq 1$ and $|f| \geq 5$, then $f$ sends at least $1/6$ units of charge to each $0$-pentagon in $B(f)$;
% \item if $|V(f)| \geq 1$ and $|f| \geq 6$, then $f$ sends at least $1/3$ units of charge to each $0$-pentagon in $B(f)$; and
% \item if $|V(f)| \geq 2$ and $|f| \geq 4$, then $f$ sends at least $1/3$ units of charge to each $0$-pentagon in $B(f)$.
% \end{packed_enum}
% \end{cor}

It follows from Proposition~\ref{prop:ch_5} and Step~6 that it remains to show
that after the last discharging step the charge of every $0$-pentagon is non-negative.
A $0$-pentagon can contribute either $1/3$ or $1/6$ units of charge (to a triangle) at most once through each of its edges.
We analyze the charge of $0$-pentagons according to their charge after Step~1 in Lemmas~\ref{lem:ch_6-0-pentagon-negative},
\ref{lem:ch_6-0-pentagon-0}, \ref{lem:ch_6-0-pentagon-1/3}, \ref{lem:ch_6-0-pentagon-2/3} and \ref{lem:ch_6-0-pentagon-1}.
In all cases we conclude that the charge of the $0$-pentagons after Step~6 is non-negative.
Recall that since $M(G)$ is $2$-connected, 
the boundary of every $0$-pentagon is a simple $5$-cycle.

Before proving the above-mentioned lemmas, we introduce some useful notations and prove a few auxiliary propositions. 
Let $f$ be a $0$-pentagon.
The vertices of $f$ are denoted by $v_0,\ldots,v_4$, listed in their clockwise cyclic order.
The edges of $f$ are $e_i=v_{i-1}v_i$, for $i=0,\ldots,4$ (addition and subtraction are modulo $5$).
For every edge $e_i=v_{i-1}v_i$ we denote by $(A_i,B_i)$ the edge of $G$ that contains $e_i$,
such that $v_{i-1}$ is between $A_i$ and $v_{i}$ on $(A_i,B_i)$.
Denote by $t_i$ the triangle to which $f$ sends charge through $e_i$, if such a triangle exists.
Note that if $t_i$ is a $1$-triangle then one of its vertices is $A_{i+1}=B_{i-1}$.
Its other vertices will be denoted by $x_i$ and $y_i$ such that $x_i$ is contained in $(A_{i-1},B_{i-1})$
and $y_i$ is contained in $(A_{i+1},B_{i+1})$.
If $t_i$ receives charge from $f$ in Step~3, then its neighbors are $1$-triangles.
In this case we denote by $x'_i$ (resp., $y'_i$) the third vertex of the neighbor whose other two vertices are $A_{i+1}$ and $x_i$ (resp., $y_i$).
If $t_i$ is a $0$-triangle, then $w_i$ denotes its vertex which is the crossing point of $(A_{i-1},B_{i-1})$
and $(A_{i+1},B_{i+1})$, and, as before, $x_i$ and $y_i$ denote its other vertices.
Finally, we denote by $f_i$ the vertex-neighbor of $f$ at $v_i$.
See Figure~\ref{fig:0-pentagon-notations} for an example of these notations. 
\begin{figure}[ht]
    \centering
    \includegraphics[width=6cm]{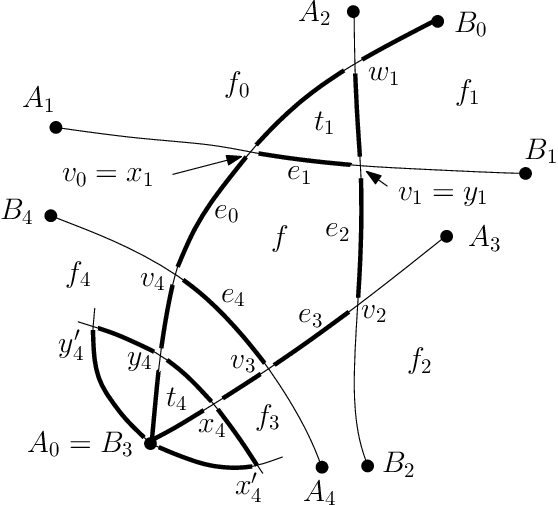}
	\caption{The notations used for vertices, edges, and faces near a $0$-pentagon $f$.
				 Bold edge-segments mark edges of $M(G)$.}
	\label{fig:0-pentagon-notations}
\end{figure}
Note also that in all the figures bold edge-segments mark edges of $M(G)$
and filled circles represent vertices of $G$.

It might happen that different names refer to the same point (e.g., $A_0$ and $B_3$ in Figure~\ref{fig:0-pentagon-notations}).
However, since $(A_i,B_i)$ crosses $(A_{i+1},B_{i+1})$ at $v_{i}$ we have:
% and crosses $(A_{i+1},B_{i+1})$ at $v_i$, we have:

\begin{obs}\label{obs:A_i,B_i}
For every $0 \leq i \leq 4$ we have $\{A_{i},B_{i}\} \cap \{A_{i+1},B_{i+1}\} = \emptyset$.
\end{obs}

\begin{prop}\label{prop:3-non-consecutive}
Let $f$ be a $0$-pentagon that contributes charge in Step~1 through $e_i$ and $e_{i+1}$, $0 \leq i \leq 4$,
and also contributes charge through $e_{i+3}$ to a $1$-triangle whose wedge contains one $0$-quadrilateral or to a $0$-triangle.
If $ch_5(f) < 0$, then $f$ receives at least $2/3$ units of charge from $f_i$ in Step~6 (and therefore $ch_6(f) \geq 0$).
\end{prop}

\begin{proof}
Assume without loss of generality that $i=0$.
Each of $(A_2,B_2)$ and $(A_4,B_4)$ supports one $0$-triangle, one $0$-pentagon, 
and either another $0$-triangle or a $0$-quadrilateral.
Thus, they already have four crossings among the vertices of these faces (see Figure~\ref{fig:3-non-consecutive}).
\begin{figure}[ht]
    \centering
    \includegraphics[width=7cm]{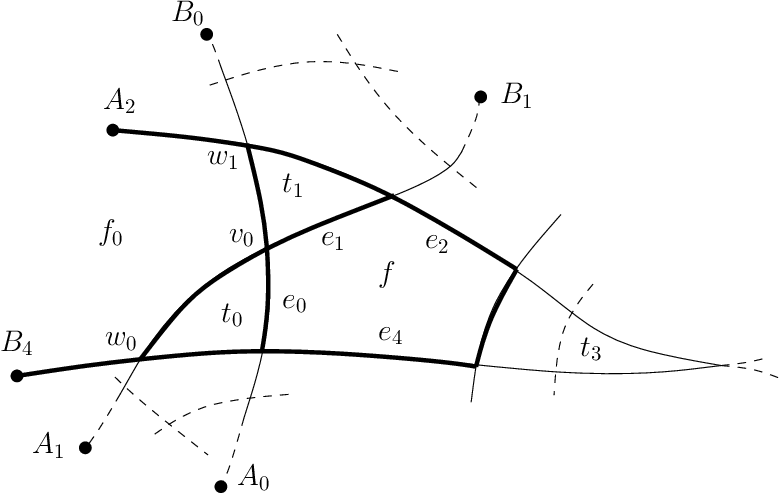}
	\caption{$f$ contributes charge through $e_0$ and $e_1$ in Step~1 and also contributes charge through $e_3$ either
				to a $0$-triangle or to a $1$-triangle whose wedge contains one $0$-quadrilateral.
				If $f$ is a vertex-neighbor of $0$-pentagons at both $w_0$ and $w_1$, then each of $(A_0,B_0)$ and $(A_1,B_1)$ has five crossings.}
	\label{fig:3-non-consecutive}
\end{figure}
Therefore, $A_2w_1$, $w_1v_0$, $v_0w_0$ and $w_0B_4$ are edges of $f_0$.
Note that $A_2 \neq B_4$ for otherwise $(A_2,B_2)$ and $(A_4,B_4)$ would intersect twice.
Thus, $|f_0| \geq 5$ and $|V(f_0)|\geq 2$.
Clearly $A_2,B_2 \notin \mathcal{P}'(f_0)$.
Observe that at least one of $w_0$ and $w_1$ is not in $\mathcal{P}'(f_0)$ either, for otherwise each of $(A_0,B_0)$ and $(A_1,B_1)$ has five crossings (see Figure~\ref{fig:3-non-consecutive}).
Thus, $|\mathcal{P}'(f_0)| \leq |f_0|-3$.
The claim now follows from Proposition~\ref{prop:step6}.
% Suppose that the clockwise chain on the boundary of $f_0$ from $B_4$ to $A_2$ contains $k\geq 1$ edges.
% Each of these edges increases the charge of $f_0$ after Step~5 by at least $2/3$,
% since it increases the initial charge by $1$ and at most $1/3$ unit of charge are contributed through it
% (if the edge contains an original vertex then it increases $ch_5(f)$ by at least $4/3$). 
% Note that $f_0$ contributes at most $1/6$ units of charge through each of $A_2w_1$ and $w_0B_4$.
% It follows that $ch_5(f_0) \geq 4+2k/3+2-4-2/3-2/3-2/6 = (2k+1)/3$.
% Since $f_0$ does not contribute charge in Step~6 through $A_2$ and $B_4$ we have $|B(f_0)| \leq 3+k-1=k+2$.
% Therefore $f_0$ contributes at least $\frac{2k+1}{3(k+2)} \geq \frac{1}{3}$ units of charge to each $0$-pentagon in $B(f_0)$, including $f$.
\end{proof}

\begin{prop}\label{prop:2/3}
Let $f$ be a $0$-pentagon that contributes charge through $e_i$ ($0 \leq i \leq 4$) in Step~1 such that the wedge of $t_i$ contains no $0$-quadrilaterals.
\begin{packed_enum}
\item If $f$ contributes charge through $e_{i+1}$ in Step~3 such that the wedge of $t_{i+1}$ contains exactly one $0$-quadrilateral
and $B_{i-1} \in V(f_i)$, then $f_i$ sends at least $2/3$ units of charge to $f$ in Step~6.
\item If $f$ contributes charge through $e_{i-1}$ in Step~3 such that the wedge of $t_{i-1}$ contains exactly one $0$-quadrilateral
and $A_{i+1} \in V(f_{i-1})$, then $f_{i-1}$ sends at least $2/3$ units of charge to $f$ in Step~6.
\end{packed_enum}
\end{prop}

\begin{proof}
By symmetry we may assume without loss of generality that $f$ contributes charge through $e_1$ in Step~1 and through $e_2$ in Step~3, 
the wedge of $t_2$ contains one $0$-quadrilateral, and $B_0 \in V(f_1)$.
It follows that $x_2,v_1,w_1,B_0 \in V(f_1)$ and $|f_1| \geq 5$ (see Figure~\ref{fig:2thirds}).
%Denote by $z$ the vertex of $f_1$ that precedes $x_2$ (see Figure~\ref{fig:2thirds}).
\begin{figure}[ht]
    \centering
    \subfigure[$f_1$ contributes at most $1/6$ units of charge through $x_2v_1$.]{\label{fig:2thirds}
    {\includegraphics[width=4.5cm]{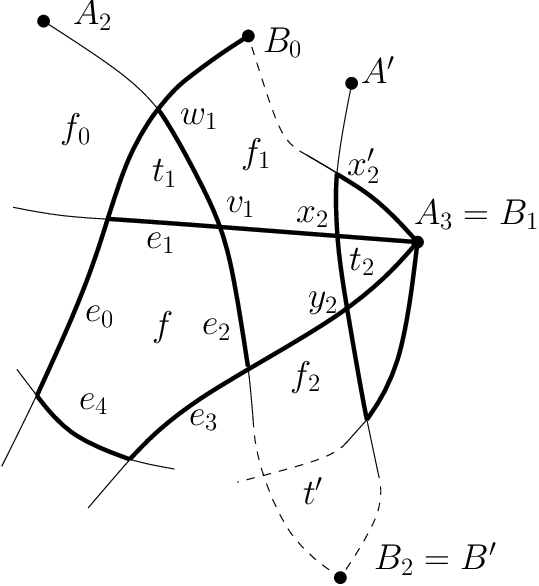}}}
	\hspace{2mm}
    \subfigure[If $f_0$ is a $0$-pentagon, then it does not contribute charge through $rw_1$.]{\label{fig:2thirds-c}
    {\includegraphics[width=4.5cm]{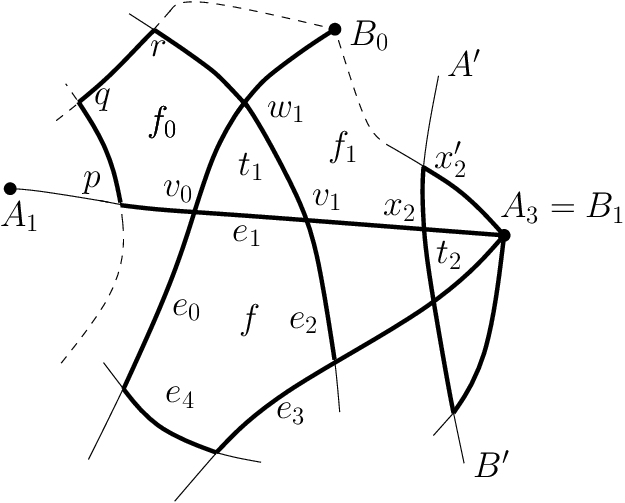}}}
	\hspace{2mm}
    \subfigure[If $f_0$ is a $0$-pentagon that contributes charge through $qr$ in Step~3, then it does not contribute charge through $pq$.]{\label{fig:2thirds-a}
    {\includegraphics[width=4.5cm]{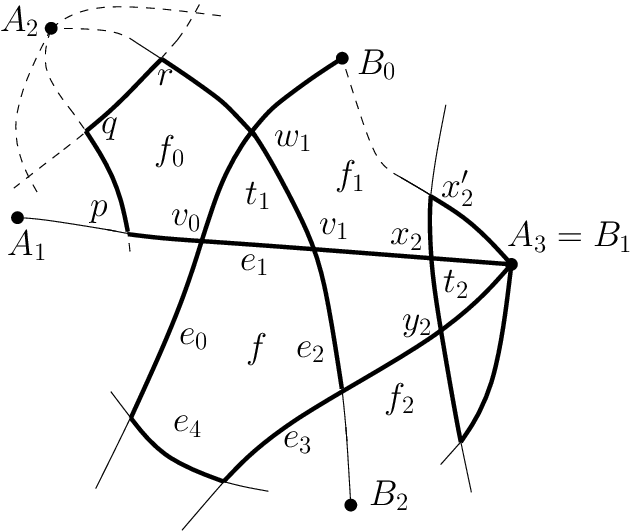}}}
	\caption{Illustrations for Proposition~\ref{prop:2/3}:$f$ contributes charge through $e_1$ in Step~1 and through $e_2$ in Step~3,
			the wedges of $t_1$ and $t_2$ contain zero and one $0$-quadrilaterals, respectively, and $B_0 \in V(f_1)$.}
	\label{fig:2thirds-prop}
\end{figure}
Observe that $f_1$ contributes at most $1/6$ units of charge through each of its edges that are incident to $B_0$
and also through $x'_2x_2$ and $x_2v_1$ (by Proposition~\ref{prop:1-triangle-neighbor}).
% Indeed, let $(A',B')$ be the edge of $G$ that contains $x_2y_2$.
% Observe that this edge has four crossings since it supports three $1$-triangles ($t_2$ and its neighbors).
% If $f$ contributes charge through $x_2v_1$ to a triangle $t'$, then $(A',B')$ and $(A_2,B_2)$ must intersect and $f_2$ must be a $0$-quadrilateral.
% If $t'$ is $0$-triangle then $(A',B')$ has five crossings.
% Therefore, $t'$ is a $1$-triangle and an endpoint of $(A',B')$, say $B'$, coincides with $B_2$.
% However, then one neighbor of $t'$ is incident to two original vertices ($A_3$ and $B_2$),
% and therefore $t'$ does not receive charge in Step~3.
Note also that $B_0,x'_2,x_2,w_1 \notin \mathcal{P}'(f_1)$:
It is clear that $B_0,x'_2,x_2 \notin \mathcal{P}'(f_1)$ since each of the vertex-neighbors at these vertices contains an original vertex of $G$
and therefore cannot be a $0$-pentagon.
Suppose that $f_0$, the vertex-neighbor of $f_1$ at $w_1$ is a $0$-pentagon such that $ch_5(f_0) < 0$,
and let $w_1,v_0,p,q,r$ be its vertices listed in a clockwise order (see Figure~\ref{fig:2thirds-c}).

Consider the edge $rw_1$.
Since $B_0w_1$ is an edge of $M(G)$, if $f_0$ contributes charge through $rw_1$,
then it must be to a $1$-triangle whose vertices are $r,w_1,B_0$.
However, one neighbor of this triangle is $f_1$ and its other neighbor is incident to two original vertices,
therefore this $1$-triangle receives charge from both of its neighbors in Step~4 and no charge from $f_0$.

Consider the edge $pq$.
Since $(A_1,B_1)$ has four crossings among the vertices of $f_0$ and $f_1$,
it follows that if $f_0$ contributes charge through $pq$, then it must be to a $1$-triangle which is its immediate neighbor at $pq$.
Thus, if $f_0$ contributes $1/3$ units of charge through $pq$, then it follows from Proposition~\ref{prop:e_1-e_2} that it contributes at most 
$1/6$ units of charge through $qr$ and through $v_0p$ and hence $ch_5(f_0) \geq 0$.

Consider the edge $qr$.
Since $(A_2,B_2)$ has four crossings among the vertices of $f_0$ and $f$,
it follows that if $f_0$ contributes charge through $qr$, then it must be to a $1$-triangle which is its immediate neighbor at $qr$.
Note that if $f_0$ contributes $1/3$ units of charge through $qr$, then its immediate neighbor at $pq$ cannot be a triangle
or a $0$-quadrilateral (otherwise $(A_1,B_1)$ would have more than four crossings, see Figure~\ref{fig:2thirds-a}).
Therefore it such a case $f_0$ does not contribute charge through $pq$.

It follows that $ch_5(f_0) \geq 0$ and therefore $f_0 \notin \mathcal{P}(f_1)$.
Therefore, $f_1$ contributes to $f$ in Step~6 at least $\frac{|f_1|+1-4-1/3-1/3-4/6-(|f_1|-5)/3}{|f_1|-4} \geq 2/3$
units of charge.
\end{proof}

\begin{prop}\label{prop:1/6}
Let $f$ be a $0$-pentagon that contributes charge through $e_i$ ($0 \leq i \leq 4$) in Step~3 such that the wedge of $t_i$ contains exactly one $0$-quadrilateral.
\begin{packed_enum}
\item If $f$ contributes charge through $e_{i+1}$ in Step~5 such that the wedge of $t_{i+1}$ contains no $0$-quadrilaterals,
then $f_i$ sends at least $1/6$ units of charge to $f$ in Step~6.
\item If $f$ contributes charge through $e_{i-1}$ in Step~5 such that the wedge of $t_{i-1}$ contains no $0$-quadrilaterals,
then $f_{i-1}$ sends at least $1/6$ units of charge to $f$ in Step~6.
\end{packed_enum}
\end{prop}

\begin{proof}
By symmetry we may assume without loss of generality that $f$ contributes charge through $e_1$ in Step~3 and through $e_2$ in Step~5, 
the wedge of $t_1$ contains exactly one $0$-quadrilateral, and the wedge of $t_2$ contains no $0$-quadrilaterals
(see Figure~\ref{fig:1/6}).
\begin{figure}[ht]
    \centering
    \includegraphics[width=4.5cm]{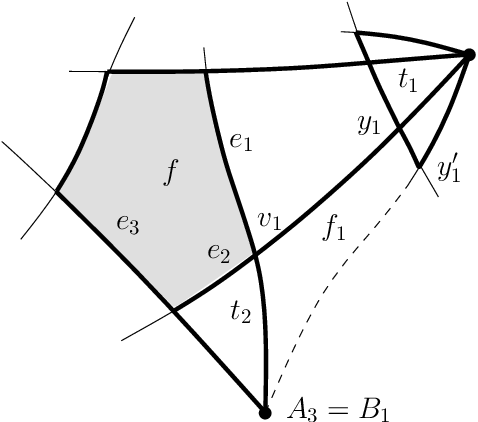}
	\caption{An illustration for Proposition~\ref{prop:1/6}: $f$ contributes charge through $e_1$ in Step~1 and through $e_2$ in Step~5,
			the wedges of $t_1$ and $t_2$ contain one and zero $0$-quadrilaterals, respectively.}
	\label{fig:1/6}
\end{figure}
Consider $f_1$ and observe that $A_3,v_1,y_1,y'_1 \in V(f_1)$.
By Proposition~\ref{prop:1-triangle-neighbor} $f_1$ contributes at most $1/6$ units of charge through each of $v_1y_1$ and $y_1y'_1$.
Note that $f_1$ also contributes at most $1/6$ units of charge through each of edges that are incident to $A_3$.
Since $A_3,y_1,y'_1 \notin \mathcal{P}'(f_1)$, if the size of $f_1$ is at least five, then it contributes at least
$\frac{|f_1|+1-4-1/3-4/6-(|f_1|-4)/3}{|f_1|-3} \geq 1/6$ units of charge to $f$ in Step~6.

If $f_1$ is a $1$-quadrilateral, then observe that it does not contribute charge through $y'_1A_3$,
since its immediate neighbor at this edge is incident to two original vertices.
In this case the charge of $f_1$ after Step~5 is at least $1/6$ and $f$ gets all this excess charge.
\end{proof}

% lemmas
%%%%%%%%%%%%%%%%%%%%%%%%%%%%%%%%%%%%%%%%%%%%%%%%%%%%%%%

\begin{lem}\label{lem:ch_6-0-pentagon-negative}
Let $f$ be a $0$-pentagon such that $ch_1(f)<0$ and $ch_5(f)<0$. Then $ch_6(f) \geq 0$.
\end{lem}

\begin{proof}
If $ch_1(f)<0$, then $f$ contributes charge to at least four $0$-triangles in Step~1.
Assume without loss of generality that $f$ contributes charge in Step~1 through $e_0,e_1,e_2,e_3$.
It follows from Proposition~\ref{prop:3-non-consecutive} that $f$ receives at least $2/3$
units of charge from each of $f_0$ and $f_2$ and therefore $ch_6(f) \geq 0$.
\end{proof}

\begin{lem}\label{lem:ch_6-0-pentagon-0}
Let $f$ be a $0$-pentagon such that $ch_1(f)=0$ and $ch_5(f)<0$. Then $ch_6(f) \geq 0$.
\end{lem}

\begin{proof}
Since $ch_1(f)=0$, $f$ contributes charge through exactly three of its edges in Step~1.
We consider two cases, according to whether these edges are consecutive on the boundary of $f$.

\smallskip
\noindent\underline{Case 1:} Assume without loss of generality 
that $f$ contributes charge in Step~1 through $e_0,e_1$ and $e_2$.
If $ch_5(f)<0$, then $f$ contributes charge through at least one more edge.
Suppose without loss of generality that $f$ contributes charge through $e_3$.

\smallskip
\noindent\underline{Subcase 1.1:} If $f$ contributes charge through $e_3$ in Step~3,
then the wedge of $t_3$ must contain exactly one $0$-quadrilateral.
Indeed, more than one $0$-quadrilateral implies five crossings on $(A_2,B_2)$,
and if $t_3$ shares an edge with $f$, then it follows from Proposition~\ref{prop:e_1-e_2}
that $f$ cannot contribute charge through $e_2$ in Step~1.
It follows from Proposition~\ref{prop:3-non-consecutive} that $f$ receives at least
$2/3$ units of charge from $f_0$ and and hence $ch_6(f) \geq 0$.

\smallskip
\noindent\underline{Subcase 1.2:} If $f$ contributes $1/6$ units of charge through $e_3$ in Step~5,
then again if the wedge of $t_3$ contains one $0$-quadrilateral, then
by Proposition~\ref{prop:3-non-consecutive} $f_0$ sends at least $2/3$ units of charge to $f$ in Step~6.
Assume therefore that $t_3$ and $f$ share an edge (see Figure~\ref{fig:3-in-1b}).
\begin{figure}[ht]
    \centering
    \includegraphics[width=6cm]{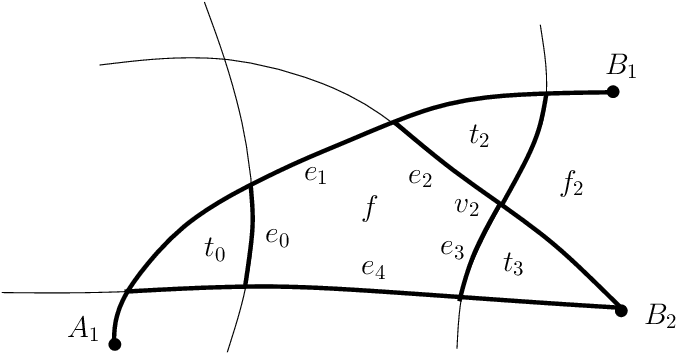}
	\caption{$f$ contributes charge to $t_0,t_1,t_2$ in Step~1 and to $t_3$ in Step~5.}
	\label{fig:3-in-1b}
\end{figure}
% \begin{figure}[ht]
%     \centering
%     \subfigure[$f$ contributes charge to $t_0,t_1,t_2$ in Step~1 and to $t_3$ in Step~5.]{\label{fig:3-in-1b}
%     {\includegraphics[width=6cm]{3-in-1b}}}
% 	\hspace{10mm}
%     \subfigure[$f$ contributes charge to $t_0,t_1,t_3$ in Step~1, to $t_2$ in Step~3, and to $t_4$ in either Step~3 or~5.]{\label{fig:3-in-1c}
%     {\includegraphics[width=6cm]{3-in-1c}}}
% 	\caption{$f$ contributes charge to three $0$-triangles in Step~1.}
% 	\label{fig:3-in-1}
% \end{figure}
It follows that $B_2$ is incident to $f_2$.
Since $(A_1,B_1)$ supports $t_0,f$ and $t_2$, it already has four crossings,
and therefore $B_1$ is also a vertex of $f_2$.
Note that $B_1 \neq B_2$ by Observation~\ref{obs:A_i,B_i}.
Thus, $|V(f_2)|\geq 2$ and $|f_2| \geq 4$, and therefore, by Proposition~\ref{prop:step6}, 
$f_2$ sends at least $1/3$ units of charge to $f$ in Step~6.
% Suppose that the clockwise chain on the boundary of $f_2$ from $B_1$ to $B_2$ contains $k\geq 1$ edges.
% $f_2$ may contribute at most $1/3$ units of charge though each of these edges,
% but each of these edges also increases the initial charge of $f_2$ by one.
% Note that $f_2$ contributes $1/3$ units of charge through $v_2w_2$ and at most $1/6$ units of charge through each of $w_2B_1$ and $B_2v_2$.
% It follows that $ch_5(f_2) \geq 3+2k/3+2-4-2/3-1/3-2/6 = (2k-1)/3$.
% Since $f_2$ does not contribute charge in Step~6 through $B_1$ and $B_2$ we have $|B(f_2)| \leq 2+k-1=k+1$.
% Therefore $f_2$ contributes at least $\frac{2k-1}{3(k+1)} \geq \frac{1}{6}$ units of charge to each $0$-pentagon in $B(f_2)$, including $f$.
% Thus, $f_2$ compensates for the charge taken by $t_3$.

Note that if $f$ contributes charge through $e_4$,
then by symmetry either $f_1$ or $f_4$ compensates for this charge in Step~6.
This concludes Case~1.

\smallskip
\noindent\underline{Case 2:} Assume without loss of generality 
that $f$ contributes charge in Step~1 through $e_0,e_1$ and $e_3$.
It follows from Proposition~\ref{prop:3-non-consecutive} that $f_0$ sends at least $2/3$ units of charge to $f$ in Step~6
and therefore $ch_6(f) \geq 0$.
\end{proof}

\begin{lem}\label{lem:ch_6-0-pentagon-1/3}
Let $f$ be a $0$-pentagon such that $ch_1(f)=1/3$ and $ch_5(f)<0$. Then $ch_6(f) \geq 0$.
\end{lem}

\begin{proof}
Suppose that $ch_1(f)=1/3$ and $ch_5(f)<0$.
Assume without loss of generality that $f$ contributes $1/3$ units of charge in Step~1 to $t_0$ through $e_0$.
By symmetry, there are two cases to consider, according to whether the other edge through which $f$
sends charge in Step~1 is $e_1$ or $e_2$.

\bigskip
\noindent\underline{Case 1:} $f$ contributes charge through $e_0$ and $e_1$ in Step~1.
We will use the following propositions.

\begin{prop}\label{prop:2-in-1a}
Let $f$ be a $0$-pentagon such that $ch_5(f) <0$ and $f$ contributes charge through $e_i$ and $e_{i+1}$ in Step~1.
\begin{packed_enum}
\item Suppose that $f$ contributes charge through $e_{i+2}$ in Step~3.
If $A_i \in V(f_{i+4})$, then $f_{i+4}$ contributes at least $1/3$ units of charge to $f$ in Step~6.
Otherwise, if $A_i \notin V(f_{i+4})$, then $f_{i+1}$ contributes at least $1/3$ units of charge to $f$ in Step~6.
\item Suppose that $f$ contributes charge through $e_{i-1}$ in Step~3.
If $B_{i+1} \in V(f_{i+1})$, then $f_{i+1}$ contributes at least $1/3$ units of charge to $f$ in Step~6.
Otherwise, if $B_{i+1} \notin V(f_{i+1})$, then $f_{i+4}$ contributes at least $1/3$ units of charge to $f$ in Step~6.
\end{packed_enum}
\end{prop}

\begin{proof}
Assume without loss of generality that $i=0$.
By symmetry, it is enough to consider the first case in which $f$ contributes charge through $e_2$ in Step~3.
It follows from Proposition~\ref{prop:e_1-e_2} and the maximum number of crossings per edge that the wedge
of $t_2$ contains exactly one $0$-quadrilateral.
Since $(A_1,B_1)$ has four crossings, it follows that $A_1$ and $w_0$ are vertices of $f_4$.
Note that $|f_4| \geq 4$, for otherwise if $|f_4|=3$ then this implies that $A_1=A_0$
which is impossible by Observation~\ref{obs:A_i,B_i}.

If $A_0 \in V(f_4)$ (see Figure~\ref{fig:2-in-1a}), then it follows from Proposition~\ref{prop:step6} that $f_4$
sends at least $1/3$ units of charge to $f$ in Step~6.
\begin{figure}[ht]
    \centering
    \subfigure[If $A_0 \in V(f_4)$, then $f_4$ sends at least $1/3$ unit of charge to $f$ in Step~6.]{\label{fig:2-in-1a}
    {\includegraphics[width=5cm]{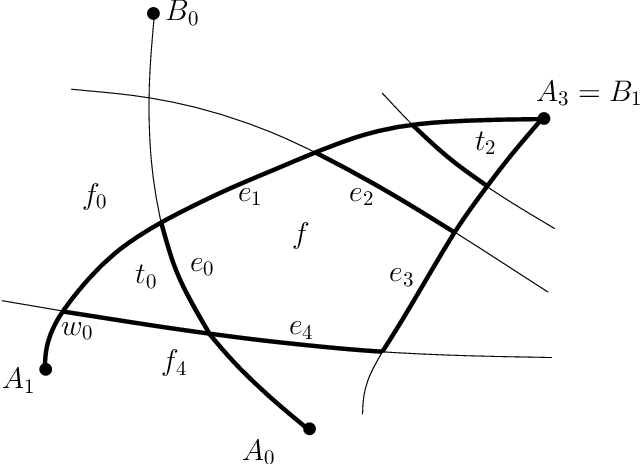}}}
	\hspace{10mm}
    \subfigure[If $A_0 \notin V(f_4)$, then $f_1$ sends at least $1/3$ unit of charge to $f$ in Step~6.]{\label{fig:2-in-1b}
    {\includegraphics[width=5cm]{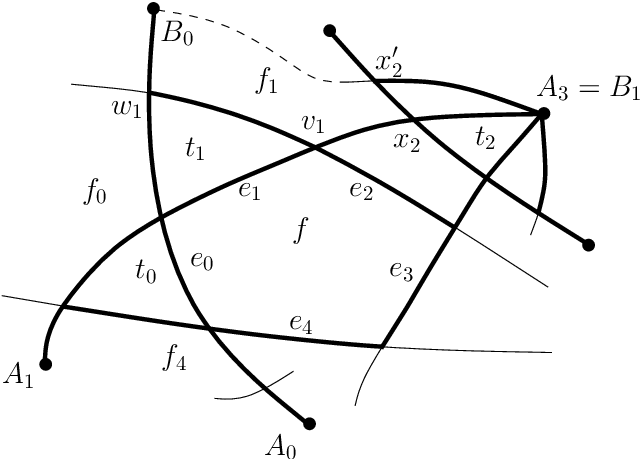}}}
	\caption{Illustrations for the proof of Proposition~\ref{prop:2-in-1a}: $f$ contributes charge through $e_0$ and $e_1$ in Step~1 and
				through $e_2$ in Step~3.}
	\label{fig:2-in-1-1/3}
\end{figure}

Otherwise, if $A_0$ is not a vertex of $f_4$, then it follows that $B_0$ and $w_1$ are vertices of $f_1$ and $|f_1| \geq 5$ (see Figure~\ref{fig:2-in-1b}).
% Let $z$ be the vertex of $f_1$ that precedes $x_2$ in the clockwise order of the vertices on the boundary of $f_1$.
% Recall that the two neighbors of $t_2$ must be $1$-triangles and therefore $z$ is a crossing point in $G$.
Observe that $f_1$ contributes at most $1/6$ units of charge through each of $x'_2x_2$ and $x_2v_1$ by Proposition~\ref{prop:1-triangle-neighbor},
and also through each of its edges that are incident to $B_0$.
Note also that $x'_2,x_2,B_0 \notin \mathcal{P}'(f_1)$.
Therefore $f_1$ contributes at least $\frac{|f_1|+1-4-1/3-1/3-4/6-(|f_1|-5)/3}{|f_1|-3} \geq 1/3$ units of charge to $f$ in Step~6.
% Let $k \geq 1$ be the number of edges on the clockwise chain from $B_0$ to $z$ on the boundary of $f_1$.
% Then $ch_5(f_1) \geq 4+2k/3+1-4-1/3-2/3-2/6 = (2k-1)/3$.
% Since $f_1$ is not a vertex-neighbor of a $0$-pentagon through $B_0,x_2$ and $z$, we have $|\mathcal{P}(f_1)| \leq k+1$.
% Therefore, if $k \geq 2$, then every $0$-pentagon in $\mathcal{P}(f_1)$ (including $f$) receives at least $\frac{2k-1}{3(k+1)} \geq 1/3$
% units of charge from $f_1$ in Step~6.
% If $k=1$, that is, $B_0z$ is an edge of $f_1$, then notice that the other face that shares this edge with $f_1$
% is incident to at least two original vertices.
% Therefore, $f_1$ does not contribute charge through $B_0z$ and hence $ch_5(f_1) \geq 5+1-4-1/3-2/3-2/6 = 2/3$.
% In this case $|\mathcal{P}(f_1)| \leq 2$ and therefore, as before, $f$ receives at least $1/3$ units of charge from $f_1$
% in Step~6.
\end{proof}

\begin{prop}\label{prop:2-in-1b}
Let $f$ be a $0$-pentagon such that $ch_5(f) <0$, $f$ contributes charge through $e_i$ and $e_{i+1}$ in Step~1,
and through $e_{i+3}$ in Step~3 or~5.
\begin{packed_enum}
\item Suppose that $f$ contributes charge through $e_{i+2}$ in Step~3 or 5.
Suppose further that the wedge of $t_{i+2}$ contains one $0$-quadrilateral and the wedge of $t_{i+3}$ contains no $0$-quadrilaterals.
If $A_i \in V(f_{i+4})$, then $f_{i+4}$ contributes at least $1/3$ units of charge to $f$ in Step~6.
Otherwise, if $A_i \notin V(f_{i+4})$, then $f_{i+4}$ contributes at least $1/6$ units of charge to $f$ in Step~6.
\item Suppose that $f$ contributes charge through $e_{i-1}$ in Step~3 or 5.
Suppose further that the wedge of $t_{i-1}$ contains one $0$-quadrilateral and the wedge of $t_{i+3}$ contains no $0$-quadrilaterals.
If $B_{i+1} \in V(f_{i+1})$, then $f_{i+1}$ contributes at least $1/3$ units of charge to $f$ in Step~6.
Otherwise, if $B_{i+1} \notin V(f_{i+1})$, then $f_{i+1}$ contributes at least $1/6$ units of charge to $f$ in Step~6.
\end{packed_enum}
\end{prop}

\begin{proof}
Assume without loss of generality that $i=0$.
By symmetry, it is enough to consider the first case in which $f$ contributes charge through $e_2$ in Step~3 or~5.
It follows from the maximum number of crossings per edge that the wedge
of $t_2$ contains exactly one $0$-quadrilateral.
Since $(A_1,B_1)$ has four crossings, it follows that $A_1$ and $w_0$ are vertices of $f_4$.
Note that $|f_4| \geq 4$, for otherwise if $|f_4|=3$ then this implies that $A_1=A_0$
which is impossible by Observation~\ref{obs:A_i,B_i}.

If $A_0 \in V(f_4)$ (see Figure~\ref{fig:2-in-1ee}), then it follows from Proposition~\ref{prop:step6} that $f_4$
sends at least $1/3$ units of charge to $f$ in Step~6.
\begin{figure}[ht]
    \centering
    \subfigure[If $A_0 \in V(f_4)$, then $f_4$ sends at least $1/3$ unit of charge to $f$ in Step~6.]{\label{fig:2-in-1ee}
    {\includegraphics[width=6cm]{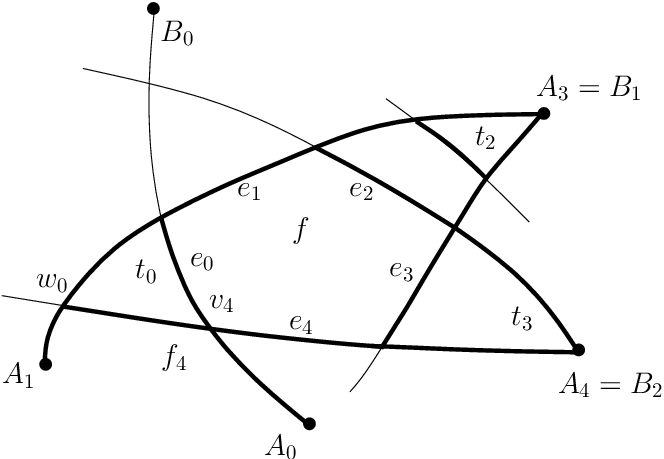}}}
	\hspace{10mm}
    \subfigure[If $A_0 \notin V(f_4)$, then $f_4$ sends at least $1/6$ unit of charge to $f$ in Step~6.]{\label{fig:2-in-1ff}
    {\includegraphics[width=6cm]{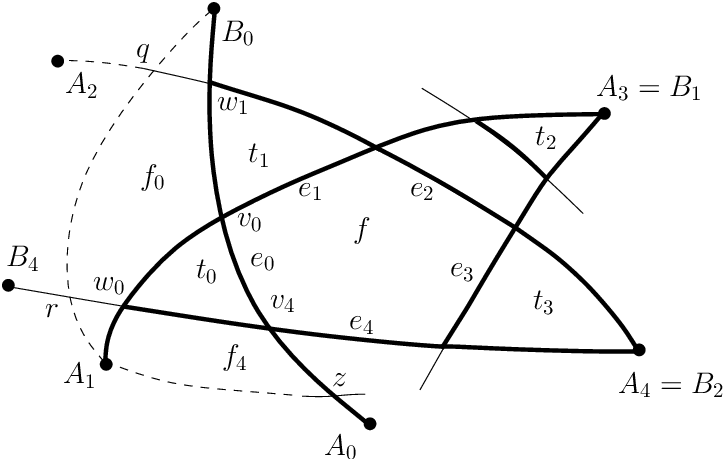}}}
	\caption{Illustrations for the proof of Proposition~\ref{prop:2-in-1b}: 
				$f$ contributes charge through $e_0$ and $e_1$ in Step~1 and
				through $e_2$ and $e_3$ in Steps~3 or~5.
				The wedge of $t_2$ contains one $0$-quadrilateral.
				The wedge of $t_3$ contains no $0$-quadrilaterals.}
	\label{fig:2-in-1bb}
\end{figure}

Suppose that $A_0 \notin V(f_4)$ and let $z$ be the crossing point on $(A_0,B_0)$ between $A_0$ and $v_4$
(see Figure~\ref{fig:2-in-1ff}).
If $f_4$ contributes charge through $v_4z$, then the recipient of this charge must be a $1$-triangle
whose neighbors are $t_3$ and a face that is incident to at least two original vertices.
Therefore $f_4$ contributes at most $1/6$ units of charge through $v_4z$.
Note that it also contributes at most $1/6$ units of charge through $A_1w_0$ and its other edge that is incident to $A_1$.
Observe also that $A_1,z \notin \mathcal{P}'(f_4)$.
Therefore, if $|f_4| \geq 5$, then $f_4$ contributes at least $\frac{|f_4|+1-4-1/3-1/3-3/6-(|f_4|-4)/3}{|f_4|-2} \geq 1/6$
units of charge to $f$ in Step~6.
%
% Let $k \geq 1$ be the number of edges on the clockwise chain from $z$ to $A_1$ on the boundary of $f_4$.
% Note that $f_4$ contributes at most $1/6$ units of charge through the edge in this chain that is incident to $A_1$
% and at most $1/3$ units of charge through every other edge.
% Thus, $ch_5(f_4) \geq 3+k+1-4-\frac{1}{3}-\frac{1}{3}-\frac{3}{6}-\frac{k-1}{3} = \frac{2k}{3}-\frac{5}{6}$.
% Note that $|\mathcal{P}(f_4)| \leq k+1$ since the vertex-neighbor of $f_4$ at $z$ is incident to $A_0$.
% Therefore, if $k \geq 2$ then every $0$-pentagon in $\mathcal{P}(f_4)$ (including $f$) receives at least $\frac{2k/3-5/6}{k+1} \geq 1/6$
% units of charge from $f_4$ in Step~6.

If $|f_4|=4$, that is, $zA_1$ is an edge of $f_4$, then notice that the other face that shares this edge with $f_4$
is incident to at least two original vertices.
Therefore, $f_4$ does not contribute charge through $zA_1$.
Observe that $f_4$ does not contribute charge through $A_1w_0$ in Step~4.
Indeed, if $f_4$ shares the edge $A_1w_0$ with a $1$-triangle $w_0A_1r$, then
since $(A_4,B_4)$ already has four crossings, the other neighbor of this $1$-triangle is incident to two original vertices.
It follows from the statement of Step~4 that the $1$-quadrilateral $f_4$ whose charge after Step~3 is $1/3$
does not contribute charge through $A_1w_0$.
Thus, $ch_5(f_4) \geq 1/6$.
Suppose that $f_4$ is a vertex-neighbor of a $0$-pentagon through $w_0$.
Then this $0$-pentagon is $f_0$.
Denote the vertices of $f_0$ by $v_0,w_0,r,q,w_1$ listed in clockwise order (see Figure~\ref{fig:2-in-1ff}).
%Denote by $q$ its fifth vertex. That is, the vertices of $f_0$ are $v_0,w_0,r,q,w_1$ (see Figure~\ref{fig:2-in-1ff}).
Then $f_0$ does not contribute charge through $rq$, since $(A_2,B_2)$ and $(A_4,B_4)$ already intersect at $A_4=B_2$.
If $f_0$ contributes charge through $qw_1$, then the recipient of this charge must be a $1$-triangle whose vertices are $q$, $B_0$ and $w_1$,
since $(A_0,B_0)$ already has four crossings.
Because one neighbor of this $1$-triangle is incident to two original vertices ($A_2$ and $B_0$),
it follows that $f_0$ contributes at most $1/6$ units of charge through $qw_1$.
For similar reasons $f_0$ also contributes at most $1/6$ units of charge through $w_0r$.
Therefore $ch_5(f_0) \geq 0$ and $f_0 \notin \mathcal{P}(f_4)$.
Since the vertex-neighbor of $f_4$ through $z$ is not a $0$-pentagon (it is incident to $A_0$),
we conclude that $\mathcal{P}(f_4)=\{f\}$ and therefore $f$ receives at least $1/6$ units of charge from $f_4$ in Step~6.
\end{proof}

Recall that we consider now the case that $f$ contributes charge through $e_0$ and $e_1$ in Step~1 and $ch_5(f)<0$.
By symmetry, we may assume without loss of generality that the charge that $f$ contributes through $e_4$
is not greater than the charge it contributes through $e_2$.
We proceed by considering the charge that $f$ contributes through $e_3$ (either $1/3$ in Step~3, $1/6$ in Step~5, or $0$).

\bigskip
\noindent\underline{Subcase 1.1:} $f$ contributes $1/3$ units of charge through $e_3$ in Step~3.
Note that the wedge of $t_3$ contains at most one $0$-quadrilateral, 
for otherwise each of $(A_2,B_2)$ and $(A_4,B_4)$ would have more than four crossings.
If the wedge of $t_3$ contains exactly one $0$-quadrilateral, then
it follows from Proposition~\ref{prop:3-non-consecutive} that $f$ receives at least
$2/3$ units of charge from $f_0$ in Step~6 and thus $ch_6(f) \geq 0$.
% Therefore, if $ch_6(f) < 0$, then $f$ must contribute charge in Step~3 through
% at least one of $e_2$ and $e_4$.
% However, in this case it also gets at least $1/3$ units of charge from $f_1$ or $f_4$,
% by Proposition~\ref{prop:2-in-1a}, and ends up with a non-negative charge.

Suppose that the wedge of $t_3$ contains no $0$-quadrilaterals, and refer to Figure~\ref{fig:2-in-1-1.1}.
\begin{figure}[ht]
     \centering
     \subfigure[Subcase 1.1: $f$ contributes charge through $e_3$ in Step~3. If the wedge of $t_3$ contains no $0$-quadrilaterals,
				then $f$ receives at least $1/6$ units of charge from $f_4$ in Step~6.]{\label{fig:2-in-1-1.1}
     {\includegraphics[width=6cm]{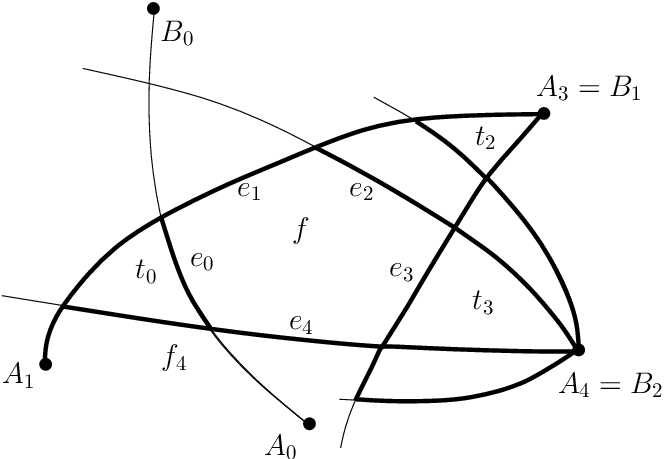}}}
 	\hspace{10mm}
     \subfigure[Subcase 1.2: $f$ contributes charge through $e_3$ in Step~5. If $f$ contributes charge through $e_2$ and $e_4$ in Step~3, 
				and through $e_3$ in Step~5, then it receives at least $1/3$ units of charge from each of $f_4$ and $f_1$ in Step~6.]{\label{fig:2-in-1g}
     {\includegraphics[width=6cm]{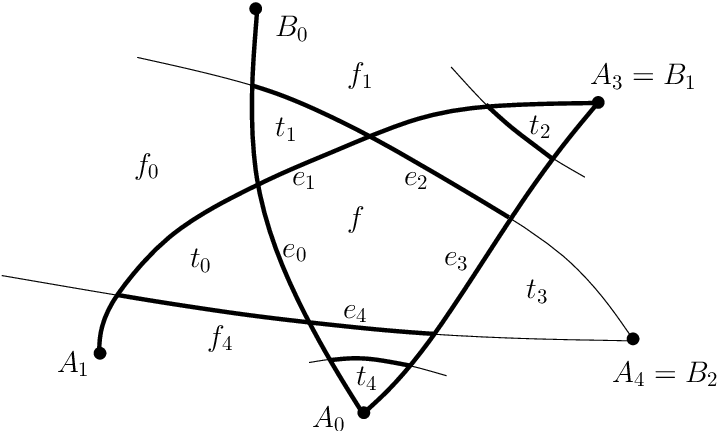}}}
 	\caption{Illustration for Case~1 in the proof of Lemma~\ref{lem:ch_6-0-pentagon-1/3}: 
 				$f$ contributes charge through $e_0$ and $e_1$ in Step~1.}
 	\label{fig:2-in-1-1.1-2}
\end{figure}
Then by Proposition~\ref{prop:e_1-e_2} $f$ does not contribute charge through neither $e_2$ nor $e_4$ in Step~3.
Therefore, $f$ must contribute $1/6$ units of charge through $e_2$ in Step~5.
Observe that is this case the wedge of $t_2$ must contain exactly one $0$-quadrilateral.
% If the wedge of $t_2$ contains no $0$-quadrilaterals, then $f_2$ is incident to
% at least two original vertices ($A_3=B_1$ and $A_4=B_2$).
% However, since $f_2$ is also a neighbor of $t_3$, it must be a $1$-triangle.
% Therefore, the wedge of $t_2$ contains exactly one $0$-quadrilateral (two $0$-quadrilaterals would imply
% that $(A_1,B_1)$ and $(A_3,B_3)$ are crossed more than four times).
Thus, by Proposition~\ref{prop:2-in-1b}, $f$ receives at least $1/6$ units of charge from $f_4$ in Step~6.
% It follows that $A_1 \in V(f_4)$.
% \begin{figure}[ht]
%     \centering
%     \subfigure[If $A_0 \in V(f_4)$, then $f_4$ sends at least $1/3$ unit of charge to $f$ in Step~6.]{\label{fig:2-in-1e}
%     {\includegraphics[width=6cm]{2-in-1e}}}
% 	\hspace{10mm}
%     \subfigure[If $A_0 \notin V(f_4)$, then $f_4$ sends at least $1/6$ unit of charge to $f$ in Step~6.]{\label{fig:2-in-1f}
%     {\includegraphics[width=6cm]{2-in-1f}}}
% 	\caption{Illustrations for Subcase~1.1 in the proof of Lemma~\ref{lem:ch_6-0-pentagon-1/3}: 
% 				$f$ contributes charge through $e_0$ and $e_1$ in Step~1 and
% 				through $e_3$ in Step~3. The wedge of $t_3$ contains no $0$-quadrilaterals.}
% 	\label{fig:2-in-1-1.1}
% \end{figure}

% If $A_0 \in V(f_4)$ (see Figure~\ref{fig:2-in-1e}), then it follows from
% Proposition~\ref{prop:step6} that $f$ receives at least $1/3$ units of charge from $f_4$
% in Step~6, and therefore $ch_6(f) \geq 0$.

If $f$ also contributes $1/6$ units of charge through $e_4$ in Step~5,
then by Proposition~\ref{prop:2-in-1b} it receives at least $1/6$ units of charge also from $f_1$ in Step~6.
Therefore, $ch_6(f) \geq 0$.

\bigskip
\noindent\underline{Subcase 1.2:} $f$ contributes $1/6$ units of charge through $e_3$ in Step~5.
Since $ch_5(f)<0$ and we have assumed that the amount of charge that $f$ contributes
through $e_2$ is at least the amount of charge it contributes through $e_4$,
it follows that $f$ contributes at least $1/6$ units of charge through $e_2$.

If $f$ contributes charge through $e_2$ in Step~3, then
it follows from Proposition~\ref{prop:2-in-1a} that $f$ receives at least $1/3$ units of charge
from either $f_4$ or $f_1$ in Step~6.
Therefore, if $f$ contributes at most $1/6$ units of charge through $e_4$,
then it ends up with a non-negative charge.

If $f$ also contributes charge through $e_4$ in Step~3,
then it follows from Proposition~\ref{prop:e_1-e_2} that $A_0 \notin f_4$ and $B_1 \notin f_1$ (see Figure~\ref{fig:2-in-1g}).
Therefore, it follows from Proposition~\ref{prop:2-in-1a} that
$f$ receives at least $1/3$ units of charge from each of $f_4$ and $f_1$,
and therefore $ch_6(f) \geq 0$.
% \begin{figure}[ht]
%     \centering
%     \includegraphics[width=6cm]{2-in-1g}
% 	\hspace{10mm}
% 	\caption{An illustration for Subcase~1.2 in the proof of Lemma~\ref{lem:ch_6-0-pentagon-1/3}: 
% 				$f$ contributes charge through $e_0$ and $e_1$ in Step~1, through $e_2$ and $e_4$ in Step~3, 
% 				and through $e_3$ in Step~5.}
% 	\label{fig:2-in-1g}
% \end{figure}

It remains to consider the case that $f$ contributes $1/6$ through each of $e_2$ and $e_4$.
If the wedge of $t_3$ contains one $0$-quadrilateral, then it follows from
Proposition~\ref{prop:3-non-consecutive} that $f$ receives at least $2/3$ units
of charge from $f_0$ and therefore $ch_6(f) \geq 0$.
Assume therefore that that wedge of $t_3$ contains no $0$-quadrilaterals.
If the wedges of $t_2$ and $t_4$ also contain no $0$-quadrilaterals,
then each of the neighbors of $t_3$ is incident to at least two original vertices,
and so $f$ does not contribute charge to $t_3$.
Assume, without loss of generality, that the wedge of $t_2$ contains one $0$-quadrilateral
(note that two $0$-quadrilateral would imply more than four crossings for $(A_1,B_1)$).
By Proposition~\ref{prop:2-in-1b} $f$ receives at least $1/6$ units of charge from $f_4$
and so $ch_6(f) \geq 0$.

\bigskip
\noindent\underline{Subcase 1.3:} $f$ does not contribute charge through $e_3$.
Since $ch_5(f) < 0$ and the amount of charge that $f$ contributes
through $e_2$ is at least the amount of charge it contributes through $e_4$,
$f$ must contribute charge through $e_2$ in Step~3.
However, by Proposition~\ref{prop:2-in-1a} $f$ receives at least $1/3$ units of charge
from $f_1$ or $f_4$ in Step~6 and so $ch_6(f) \geq 0$.

\bigskip
\noindent\underline{Case 2:} $f$ contributes charge through $e_0$ and $e_2$ in Step~1.
Note that since $(A_1,B_1)$ has four crossings (because it supports $f$ and two $0$-triangles),
it follows that $e_0$ and $e_2$ are edges of $t_0$ and $t_2$, respectively.
Since $ch_5(f)<0$, $f$ must contribute charge through $e_3$ or $e_4$.
By symmetry we may assume without loss of generality that the amount of charge that
$f$ contributes through $e_3$ is at least the amount of charge it contributes through $e_4$.
Therefore, it is enough to consider the following three subcases:
$f$ contributes $1/3$ units of charge through $e_3$; 
$f$ contributes $1/3$ units of charge through $e_1$ and $1/6$ units of charge through $e_3$; and
$f$ contributes $1/6$ units of charge through each of $e_1$, $e_3$ and $e_4$.

\bigskip
\noindent\underline{Subcase 2.1:} $f$ contributes $1/3$ units of charge through $e_3$.
It follows from Proposition~\ref{prop:e_1-e_2} and the maximum number of crossings per edge 
that the wedge of $t_3$ contains exactly one $0$-quadrilateral and that $B_1 \in V(f_2)$.
Therefore, by Proposition~\ref{prop:2/3} $f$ receives at least $2/3$ units of charge from $f_2$
and ends up with a non-negative charge.

\bigskip
\noindent\underline{Subcase 2.2:} $f$ contributes $1/3$ units of charge through $e_1$ and $1/6$ units of charge through $e_3$.
It follows from Proposition~\ref{prop:e_1-e_2} that the wedge of $t_1$ contains at least one $0$-quadrilateral.
Suppose that it contains two $0$-quadrilaterals. 
Then $\{A_0,A_1\} \subseteq V(f_4)$ and $\{B_1,B_2\} \subseteq f_2$ (see Figure~\ref{fig:2-in-1ia}).
Therefore, by Proposition~\ref{prop:step6} $f$ receives at least $1/3$ units of charge from each of $f_4$ and $f_2$ in Step~6
and so $ch_6(f) \geq 0$.
Assume therefore that the wedge of $t_1$ contains exactly one $0$-quadrilateral.

If the wedge of $t_3$ contains no $0$-quadrilaterals, then $\{B_1,B_2\} \subseteq V(f_2)$.
% It follows that $f$ contributes at most $1/6$ units of charge through $e_3$ and $e_4$
% since by Proposition~\ref{prop:e_1-e_2} $f$ does not contribute charge to $t_3$ in Step~3.
Therefore, $f$ receives at least $1/3$ units of charge from $f_2$ by Proposition~\ref{prop:step6}, and hence $ch_6(f) \geq 0$.

Assume therefore that the wedge of $t_3$ contains exactly one $0$-quadrilateral,
and consider the face $f_0$ (refer to Figure~\ref{fig:2-in-1ib}).
Observe that $x_1$, $x'_1$, $v_0$, $w_0$ and $B_4$ are vertices of $f_0$.
%and that $|f_0| \geq 5$, since the two neighbors of $t_1$ are $1$-triangles.
%Denote by $z$ the vertex of $f_0$ that precedes $x_1$ in a clockwise order of the vertices of $f_0$.
\begin{figure}[ht]
    \centering
    \subfigure[If the wedge of $t_1$ contains two $0$-quadrilaterals,
				then $f$ receives at least $1/3$ units of charge from each of $f_4$ and $f_2$.]{\label{fig:2-in-1ia}
    {\includegraphics[width=6cm]{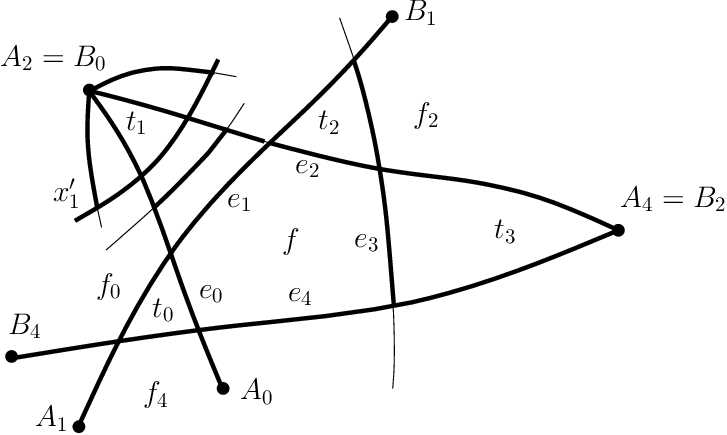}}}
	\hspace{5mm}    
    \subfigure[If the wedges of $t_1$ and $t_3$ contain one $0$-quadrilateral,
				then $f$ receives at least $1/3$ units of charge from $f_0$.]{\label{fig:2-in-1ib}
    {\includegraphics[width=6cm]{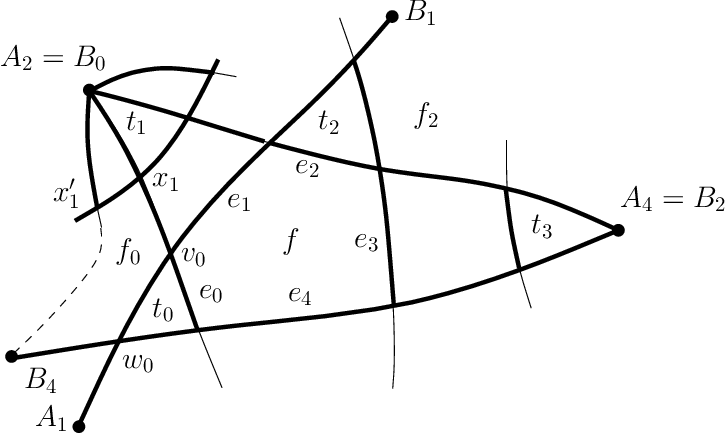}}}
	\caption{Illustrations for Subcase 2.2 in the proof of Lemma~\ref{lem:ch_6-0-pentagon-1/3}: 
				$f$ contributes charge through $e_0$ and $e_2$ in Step~1, 
				through $e_1$ in Step~3, and through $e_3$ in Step~5.}
	\label{fig:2-in-1i}
\end{figure}
Note that $f_0$ contributes no charge through $w_0B_4$ (since its immediate neighbor at this edge is incident to $B_4$ and $A_1$) and at most $1/6$ units of charge through each of $x'_1x_1$ and $x_1v_0$ (by Proposition~\ref{prop:1-triangle-neighbor})
and its other edge that is incident to $B_4$.
% and at most $1/3$ units of charge through $x_1v_0$ and $v_0w_0$.
% Observe also that $f_0$ cannot be a vertex-neighbor of a $0$-pentagon through $B_4$, $z$, $x_1$, and $w_0$.
Observe also that $B_4,w_0,x_1,x'_1 \notin \mathcal{P}'(f_0)$.
Therefore, $f_0$ contributes at least $\frac{|f_0|+1-4-1/3-1/3-3/6-(|f_0|-5)/3}{|f_0|-4} \geq 1/3$ units of charge to $f$ in Step~6.
% Let $k \geq 1$ be the number of edges on the clockwise chain from $B_4$ to $z$ on the boundary of $f_0$.
% Thus, $|f_0|=4+k$.
% Then $ch_5(f_0) \geq 4+2k/3+1-4-1/3-2/3-1/6 = (4k-1)/6$.
% Note that $|\mathcal{P}(f_2)| \leq k$ since $f_0$ is not a vertex-neighbor of a $0$-pentagon at $B_4$, $z$, $x_1$, and $w_0$.
% Therefore, every $0$-pentagon in $\mathcal{P}(f_0)$ (including $f$) receives at least $\frac{4k-1}{6k} \geq 1/3$
% units of charge from $f_0$ in Step~6.
Since $f$ contributes at most $1/6$ units of charge through $e_3$ and $e_4$, we have $ch_6(f) \geq 0$.
% If $f$ contributes $1/3$ units of charge through $e_3$, then it follows from
% Proposition~\ref{prop:2-in-1x} that it receives at least $1/3$ units of charge
% from $f_2$ as well, and therefore $ch_6(f) \geq 0$.
%
% \bigskip
% \noindent\underline{Subcase 2.2:} $f$ contributes $1/6$ units of charge through $e_1$ and $1/3$ units of charge through $e_3$.
% In this case, it follows from Proposition~\ref{prop:2-in-1x} that $f$ receives at least $1/3$ units of charge from $f_2$ in Step~6.
% If $f$ also contributes $1/3$ units of charge through $e_4$, then it also receives at least $1/3$ units of charge from $f_4$.
% Therefore, $ch_6(f) \geq 0$.

\bigskip
\noindent\underline{Subcase 2.3:} $f$ contributes $1/6$ units of charge through each of $e_1$, $e_3$ and $e_4$.
Recall that the wedges of $t_0$ and $t_2$ contain no $0$-quadrilaterals.
It follows that $A_1 \in V(f_4)$ and $B_1 \in V(f_2)$. 
If the wedge of $t_3$ contains no $0$-quadrilaterals, then $f_2$ is also incident to $A_4$ and therefore
contributes at least $1/3$ units of charge to $f$ in Step~6 by Proposition~\ref{prop:step6}.
Therefore, assume that the wedge of $t_3$ contains at least one $0$-quadrilateral.
For similar reasons, we may assume that the wedges of each of $t_4$ and $t_1$ contain at least one $0$-quadrilateral,
for otherwise $f_4$ or $f_0$ contributes at least $1/3$ units of charge to $f$ in Step~6.
It follows that each of the wedges of $t_1$, $t_3$ and $t_4$ contains exactly one $0$-quadrilateral (see Figure~\ref{fig:2-in-1q}).
\begin{figure}[ht]
    \centering
    \includegraphics[width=6cm]{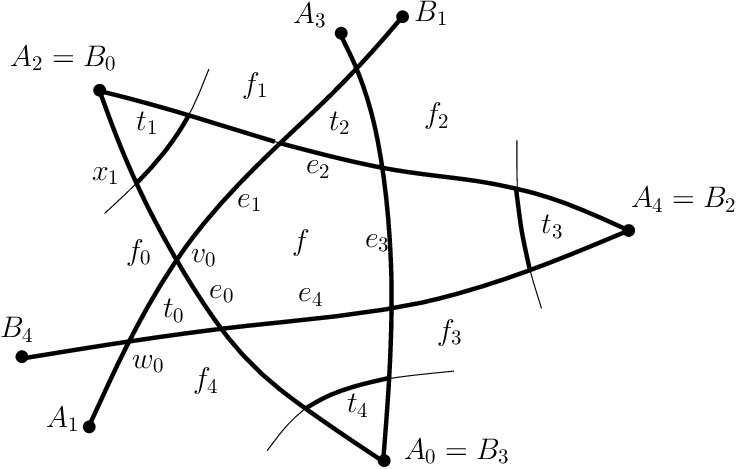}
	\caption{An illustration for Subcase 2.3 in the proof of Lemma~\ref{lem:ch_6-0-pentagon-1/3}: 
				$f$ contributes charge through $e_0$ and $e_2$ in Step~1, 
				and through $e_1$, $e_3$ and $e_4$ in Step~5.}
	\label{fig:2-in-1q}
\end{figure}
Note that each of $f_0$ and $f_1$ is of size at least four.
However, not both of them can be of size four, since then each neighbor of $t_1$ would be incident to two original vertices
and then $t_1$ would not get any charge from $f$.
Assume without loss of generality that $|f_0| \geq 5$ and observe that $f_0$
contributes no charge through $x_1v_0$ and at most $1/6$ units of charge through each of its edges that are incident to $B_4$.
Note also that $B_4,w_0,x_1 \notin \mathcal{P}'(f_0)$.
Therefore, $f_0$ contributes at least $\frac{|f_0|+1-4-1/3-1/3-2/6-(|f_0|-4)/3}{|f_0|-3} \geq 1/6$ units of charge to $f$ in Step~6
and thus $ch_6(f) \geq 0$.
This concludes the last subcase and the proof of Lemma~\ref{lem:ch_6-0-pentagon-1/3}.
\end{proof}

The following claim will be useful when considering $0$-pentagons
that contribute charge to at least two $1$-triangles in Step~3.

\begin{prop}\label{prop:e_1-e_2-step3}
Let $f$ be a $0$-pentagon that contributes charge in Step~3 through $e_i$ and $e_{i+1}$, for some $0 \leq i \leq 4$,
such that the wedges of $t_i$ and $t_{i+1}$ each contain exactly one $0$-quadrilateral.
If $ch_5(f)<0$ and $f_i$ is not a $0$-quadrilateral, then $f$ receives at least $1/3$ units of charge from $f_i$ in Step~6.
\end{prop}

\begin{proof}
Assume without loss of generality that $i=1$ and refer to Figure~\ref{fig:e_1-e_2-step3}.
\begin{figure}[ht]
    \centering
    \includegraphics[width=5.5cm]{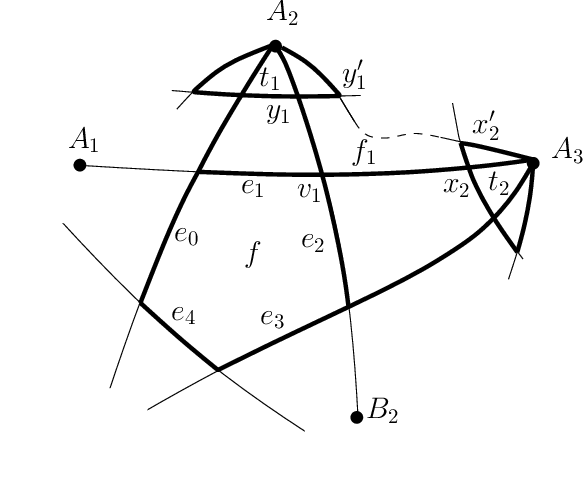}
	\caption{An illustration for the proof of Proposition~\ref{prop:e_1-e_2-step3}:
			$t_1$ and $t_2$ receive charge from $f$ in Step~3 and each of their wedges contains exactly one $0$-quadrilateral.}
	\label{fig:e_1-e_2-step3}
\end{figure}
Since $t_1$ and $t_2$ receive charge in Step~3, both of their neighbors are $1$-triangles.
% Let $A_2y_1p$ be the neighbor of $t_1$ that shares an edge with $f_1$ and
% let $A_3x_2q$ be the neighbor of $t_2$ that shares an edge with $f_1$
% (each of these triangles shares an edge with $f_1$ since each of the wedges of $t_1$ and $t_2$ contains exactly one $0$-quadrilateral).
Note that $y_1,y'_1,v_1,x_2,x'_2 \in V(f_1)$ and that $y'_1 \neq x'_2$ since $f_1$ is not a $0$-quadrilateral.
Observe that $f_1$ contributes at most $1/6$ units of charge through each of $v_1y_1$, $y_1y'_1$, $x'_2x_2$ and $x_2v_1$ by Proposition~\ref{prop:1-triangle-neighbor}.
Note that $f_1$ is not a vertex-neighbor of a $0$-pentagon at the vertices $x'_2$, $x_2$, $y_1$ and $y'_1$.
Therefore, if $|f_1| \geq 6$, then $f_1$ contributes at least $\frac{|f_1|-4-4/6-(|f_1|-4)/3}{|f_1|-4} \geq 1/3$
units of charge to $f$ in Step~6.

If $|f_1|=5$, then $f_1$ does not contribute charge through $y'_1x'_2$.
Indeed, each of the edges of $G$ that contain $y_1y'_1$ and $x'_2x_2$ already has four crossings,
and this implies that if the face that shares $y'_1x'_2$ with $f_1$ is a $1$-triangle (it cannot be a $0$-triangle or a $0$-quadrilateral),
then both of its neighbors are incident to two original vertices
and therefore this $1$-triangle gets its missing charge from them.
Therefore, $ch_5(f_1) \geq 1/3$ and $f_1$ sends all its extra charge to $f$ in Step~6.
%
% If $|f_1| \geq 6$ then the clockwise chain from $p$ to $q$ contains $|f_1|-4$ edges 
% and at most $|f_1|-5$ vertices through which $f_1$ might contribute charge in Step~6.
% Since $f_1$ contributes at most $1/6$ units of charge through each of $py_1,v_1y_1,v_1x_2,qx_2$
% and at most $1/3$ units of charge through each other edge, we have $ch_5(f_1)\geq |f_1|-4-4/6-(|f_1|-4)/3$.
% Therefore every face in $\mathcal{P}(f_1)$ receives at least
% $\frac{2|f_1|/3-10/3}{|f_1|-4} \geq 1/3$ units of charge from $f_1$ in Step~6.
% If $|f| > 5$ then every `extra' edge might add a $0$-pentagon to $\mathcal{P}(f_1)$ and $f_1$ might contribute at most $1/3$ through it,
% but it also increases the initial charge of $f_1$ by one, and therefore, still, every $0$-pentagon in $\mathcal{P}(f_1)$
% (including $f$) will receive at least $1/3$ units of charge from $f_1$ in Step~6.
\end{proof}

\begin{cor}\label{cor:f_1-or-f_2}
Let $f$ be a $0$-pentagon that contributes charge in Step~3 through $e_i$, $e_{i+1}$ and $e_{i+2}$, for some $0 \leq i \leq 4$,
such that the wedge of $t_{i+1}$ contains exactly one $0$-quadrilateral.
If $ch_5(f)<0$, then $f$ receives at least $1/3$ units of charge from $f_i$ or $f_{i+1}$ in Step~6.
\end{cor}

\begin{proof}
We may assume without loss of generality that $i=1$.
It follows from Proposition~\ref{prop:e_1-e_2} that each of the wedges of $t_1$, $t_2$ and $t_3$
contains at least one $0$-quadrilateral.
Moreover, each of the wedges of $t_1$ and $t_3$ must contain exactly one $0$-quadrilateral,
for otherwise $(A_2,A_4)$ would have more than four crossings.
If $f_1$ (resp., $f_2$) is not a $0$-quadrilateral, then it follows from Proposition~\ref{prop:e_1-e_2-step3}
that $f$ receives at least $1/3$ units of charge from this face in Step~6.
Suppose therefore that both $f_1$ and $f_2$ are $0$-quadrilaterals and let $f'$ be the face that shares an edge
with each of them and also shares $e_2$ with $f$ (see Figure~\ref{fig:f_1-or-f_2}).
\begin{figure}[ht]
    \centering
	\includegraphics[width=5cm]{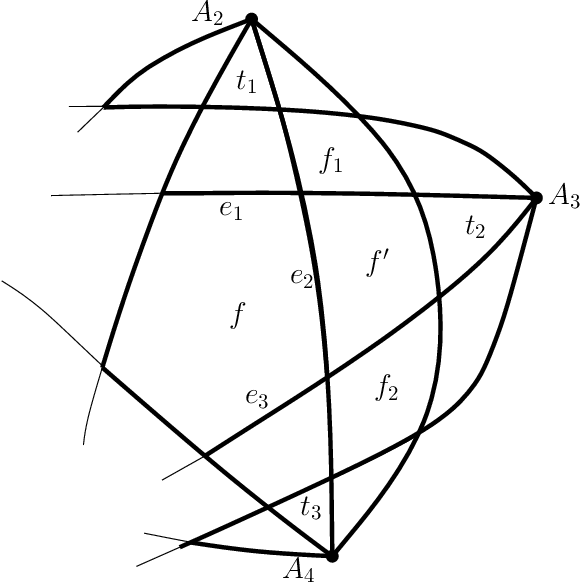}
	\caption{Illustration for the proof of Corollary~\ref{cor:f_1-or-f_2}: $f$ contributes charge through $e_1$, $e_2$ and $e_3$ in Step~3,
			and the wedge of $t_2$ contains exactly on $0$-quadrilateral.
			Then it is impossible that both $f_1$ and $f_2$ are $0$-quadrilaterals.}
	\label{fig:f_1-or-f_2}
\end{figure}
Then $f'$ must be a $0$-quadrilateral since $f$ sends charge to $t_2$ as well.
However, this implies that there are two parallel edges between $A_2$ and $A_4$ in $G$.
\end{proof}

\begin{lem}\label{lem:ch_6-0-pentagon-2/3}
Let $f$ be a $0$-pentagon such that $ch_1(f)=2/3$ and $ch_5(f)<0$. Then $ch_6(f) \geq 0$.
\end{lem}

\begin{proof}
Assume without loss of generality that $f$ contributes $1/3$ units of charge in Step~1 to $t_1$ through $e_1$.
There are three cases to consider, based on whether $f$ contributes $1/3$ units of charge
to exactly one, exactly two or at least three $1$-triangles in Step~3.

\smallskip\noindent\underline{Case 1:} $ch_3(f)=1/3$ and $ch_5(f) = -1/6$.
That is, $f$ contributes $1/3$ units of charge to exactly one $1$-triangle $t'$ in Step~3,
and $1/6$ units of charge to three $1$-triangles in Step~5.
We need to show that $f$ receives at least $1/6$ units of charge from its vertex-neighbors in Step~6.
Without loss of generality we may assume that either $t'=t_2$ or $t'=t_3$.

\smallskip\noindent\underline{Subcase 1.1:} $f$ sends $1/3$ units of charge to $t_2$ in Step~3.
It follows from Proposition~\ref{prop:e_1-e_2} that the wedge of $t_2$ contains at least one $0$-quadrilateral.
We observe first that the wedge of $t_2$ cannot contain two $0$-quadrilaterals.
Indeed, suppose it does and refer to Figure~\ref{fig:1-in-1-1-in-3_0}.
\begin{figure}[ht]
    \centering
    \subfigure[If the wedge of $t_2$ contains two $0$-quadrilaterals,
				then $(A_2,B_2)$ has five crossings.]{\label{fig:1-in-1-1-in-3_0}
    {\includegraphics[width=5cm]{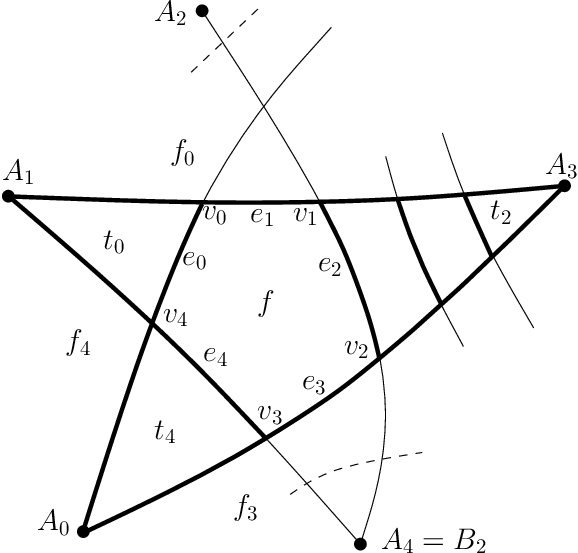}}}
    \hspace{10mm}
    \subfigure[If $e_1$ is not an edge of $t_1$, then $t_3$ receives $1/6$ units of charge
				from each of its neighbors in Step~4.]{\label{fig:1-in-1-1-in-3_1}
    {\includegraphics[width=5cm]{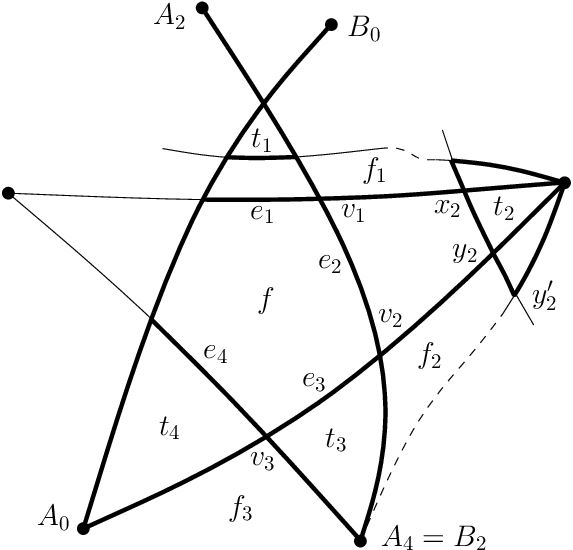}}}
%     \hspace{10mm}
%     \subfigure[If $A_2w_1$ is not an edge in $M(G)$, then $e_3$ is an edge of $t_3$,
%                and $f_2$ contributes at least $1/6$ units of charge to $f$.]{\label{fig:1-in-1-1-in-3_2}
%     {\includegraphics[width=5cm]{1-in-1-1-in-3_2a}}}
    \hspace{10mm}
%     \subfigure[$A_2w_1$ is an edge in $M(G)$ and $|f_1|=5$. $f_1$ cannot send charge through $pq$ and $w_1q$.]{\label{fig:1-in-1-1-in-3_3}
%     {\includegraphics[width=5cm]{1-in-1-1-in-3_3}}}
    \subfigure[$A_2w_1$ is an edge in $M(G)$ and $B_0w_1$ is not.]{\label{fig:1-in-1-1-in-3_3k}
    {\includegraphics[width=5cm]{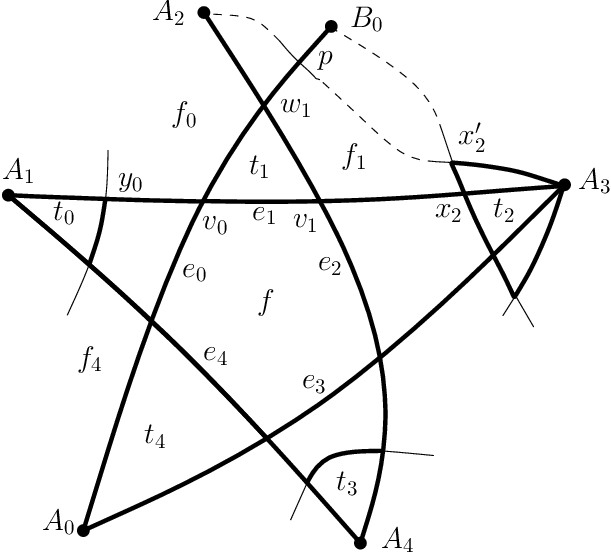}}}
	\caption{Subcase 1.1 in the proof of Lemma~\ref{lem:ch_6-0-pentagon-2/3}: $f$ sends $1/3$ units of charge to $t_1$ in Step~1,
			$1/3$ units of charge to $t_2$ in Step~3, and $1/6$ units of charge to $t_0$, $t_3$ and $t_4$ in Step~5.}
	\label{fig:1-in-1-1-in-3}
\end{figure}
Since each of the edges $(A_1,A_3)$ and $(A_3,A_0)$ has four crossings it follows that $e_0$ is an edge of $t_0$
and $e_4$ is an edge of $t_4$.
Thus $f_4$ is incident to $A_0$ and $A_1$ and is a neighbor of $t_0$ and $t_4$.
Therefore the other neighbor of $t_0$ (resp., $t_4$) cannot be incident to
any other original vertex but $A_1$ (resp., $A_0$).
This implies that $(A_2,B_2)$ has two additional crossings,
and totally five crossings.

Next, we observe that we may assume that $e_1$ is an edge of $t_1$.
Indeed, suppose it is not and refer to Figure~\ref{fig:1-in-1-1-in-3_1}.
Since each of the edges $(A_2,B_2)$ and $(A_0,B_0)$ has four crossings, $e_3$ is an edge of $t_3$ and $e_4$ is an edge of $t_4$.
It follows from Proposition~\ref{prop:1/6} that $f_2$ sends at least $1/6$ units of charge to $f$ in Step~6.
% It follows that the neighbors of $t_3$ are $f_2$ and $f_3$.
% Since $f_3$ is incident to $A_0$ and $B_2$, it follows from Observation~\ref{obs:ch_3} that it contributes charge to $t_3$ in Step~4.
% Consider $f_2$, the other neighbor of $t_3$ and observe that it also contributes charge to $t_3$ in Step~4.
% Indeed, if it does not, then since $A_4 \in V(f_2)$, it must be a $1$-quadrilateral with $ch_3(f_2)=1/3$
% (note that $f_2$ is incident to $A_4$, $v_2$, $y_2$ and $y'_2$).
% However, by Proposition~\ref{prop:1-triangle-neighbor} $f_2$ does not contribute charge through $v_2y_2$ and $y_2y'_2$ in Steps~1 or~3,
% and therefore $ch_3(f_2) \geq 2/3$.
% This implies that $f_2$ contributes charge through $v_2y_2$ in Steps~1 or~3,
% since it cannot contribute charge in these steps through its other edge with endpoints that are crossing points in $G$ 
% by Observation~\ref{obs:1-triangle-neighbor}.
% However, if $f_2$ contributes charge through $v_2y_2$ in Steps~1 or~3,
% then it follows that $f_1$ must be a $0$-quadrilateral.
% But then the edge of $f_1$ that is opposite to $v_1x_2$ is incident to a face of size at least four which is also incident to $B_0$.
% Therefore, $f_2$ contributes charge to $t_3$ in Step~4, and thus $f$ does not contribute charge to $t_3$, a contradiction.
% Hence, $e_1$ is an edge of $t_1$.

We assume therefore that $e_1$ is an edge of $t_1$ and that
the wedges of $t_2$ and $t_3$ contain exactly one $0$-quadrilateral.
Observe that this implies that $|f_1| \geq 5$.
Indeed, if $|f_1|=4$, then $B_0$ and $B_1$ must coincide, which is impossible.
If $B_0 \in V(f_1)$, then by Proposition~\ref{prop:2/3} $f_1$ sends at least $2/3$ unit of charge to $f$ in Step~6.
%
% Consider now the edge-segment $A_2w_1$.
% Suppose that it contains a crossing point between $A_2$ and $w_1$ and refer to Figure~\ref{fig:1-in-1-1-in-3_2}.
% It follows that $e_3$ is an edge of $t_3$ and that $f_2$ is a neighbor of $t_3$.
% Since $|f_1| \geq 5$, $f_2$ does not contribute charge through $v_2y_2$.
% Let $q$ be the vertex of $f_2$ that follows $y_2$ in a clockwise order of the vertices on the boundary of $f_2$.
% Note that $f_2$ contributes at most $1/6$ units of charge through each of $y_2q$ and the two edges that are incident to $A_4$.
% Since $f_2$ is not a vertex-neighbor of a $0$-pentagon at $A_4$, $y_2$ and $q$, it follows that
% $f_2$ contributes at least $\frac{|f_2|+1-4-1/3-3/6-(|f_2|-4)/3}{|f_2|-3}\geq 1/6$ units of charge to every face in $\mathcal{P}(f_2)$ (including $f$) in Step~6.
% If $|f_2|=4$, then $f_2$ does not contribute charge through $qA_4$, and therefore $ch_5(f_2) \geq 1/3$.
% In this case $f_2$ sends at least $1/6$ units of charge to $f$ in Step~6.
% It is also not hard to see that this remains true if $|f_2| \geq 5$. 
%Therefore, $ch_6(f) \geq 0$ if $A_2w_1$ is not an edge of $M(G)$.
%
% Suppose that $A_2w_1$ is an edge of $M(G)$.
% If $B_0w_1$ is also an edge of $M(G)$, then $B_0$ is a vertex of $f_1$.
% Note that in this case $f_1$ is not a vertex-neighbor of a $0$-pentagon at $B_0$, $x_2$ and $w_1$.
% Recalling that $|f_1| \geq 5$, it follows that $f_1$ contributes at least $\frac{|f_1|-4+1-1/3-2/6-(|f_1|-2)/3}{|f_1|-3} \geq 1/6$ units of charge
% to $f$ in Step~6 and so $ch_6(f) \geq 0$.
%
Suppose therefore that there is a crossing point $p$ between $B_0$ and $w_1$ on $B_0w_1$, and refer to Figure~\ref{fig:1-in-1-1-in-3_3k}.
$f_1$ contributes at most $1/6$ units of charge through $x'_2x_2$ and $x_2v_1$ by Proposition~\ref{prop:1-triangle-neighbor}.
Note that it also contributes at most $1/6$ units of charge through $w_1p$,
since the recipient of such a charge must be a $1$-triangle (recall $w_1A_2$ is an edge of $M(G)$)
that has a neighbor which is incident to two original vertices ($A_2$ and $B_0$).
Since $w_1,p,x'_2,x_2 \notin \mathcal{P}'(f_1)$, if the size of $f_1$ is at least six,
then it contributes at least $\frac{|f_1|-4-3/6-(|f_1|-3)/3}{|f_1|-4} \geq 1/6$ units of charge to $f$ in Step~6.

If $|f_1|=5$, then note that $f_1$ contributes at most $1/6$ units of charge through $px'_2$, 
since such a contribution must be to the $1$-triangle $pB_0x'_2$ that has a neighbor which is incident to
two original vertices ($B_0$ and $A_3$).
Moreover, if $f_1$ contributes $1/6$ units of charge to this triangle,
then it does not contribute charge through $w_1p$,
for otherwise both neighbors of the $1$-triangle $pB_0x'_2$ would be incident to two original vertices (see Figure~\ref{fig:1-in-1-1-in-3_3k}).
Therefore, in this case as well $f$ receives at least $1/6$ units of charge from $f$ in Step~6 and thus $ch_6(f) \geq 0$.

\medskip

\noindent\underline{Subcase 1.2:} $f$ sends $1/3$ units of charge to $t_3$ in Step~3.
We first observe that $e_1$ must be an edge of $t_1$.
Indeed, suppose it does not and refer to Figure~\ref{fig:1-in-1-1-in-3_4}.
\begin{figure}[ht]
    \centering
    \subfigure[If $e_1$ is not an edge of $t_1$, then $f_3$ is a neighbor of $t_3$ and $|V(f_3)| \geq 2$.]{\label{fig:1-in-1-1-in-3_4}
    {\includegraphics[width=5cm]{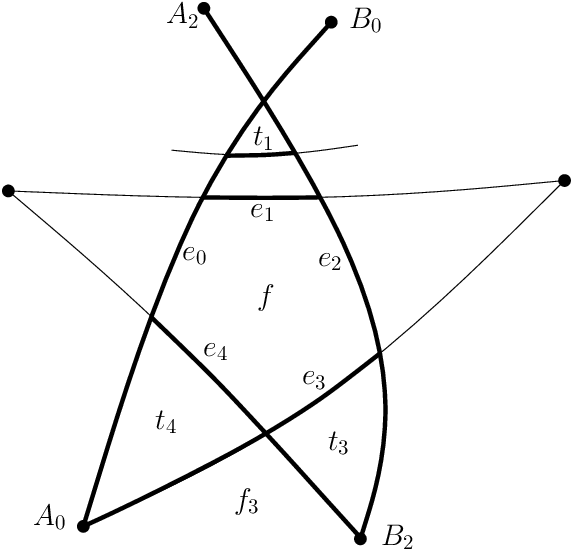}}}
    \hspace{10mm}
    \subfigure[$e_3$ is an edge of $t_3$. $f_1$ sends charge to $f$ in Step~6.]{\label{fig:1-in-1-1-in-3_5}
    {\includegraphics[width=5cm]{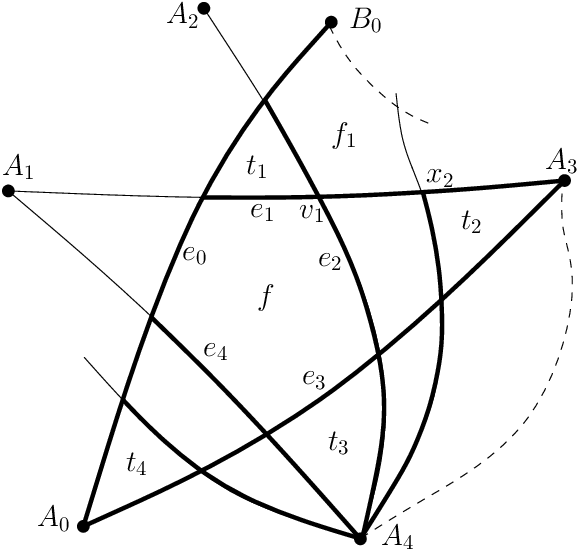}}}
    \hspace{10mm}
    \subfigure[$e_3$ is not an edge of $t_3$ and $e_4$ is an edge of $t_4$. $f_3$ sends charge to $f$ in Step~6.]{\label{fig:1-in-1-1-in-3_6}
    {\includegraphics[width=5cm]{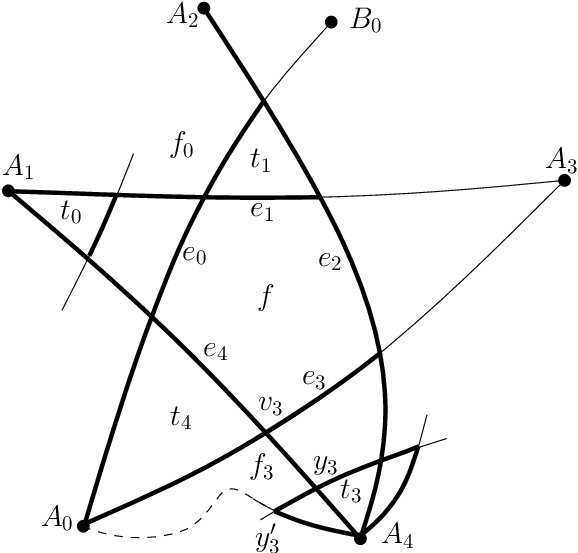}}}
    \hspace{10mm}
    \subfigure[$e_3$ is not an edge of $t_3$ and $e_4$ is not an edge of $t_4$. $f_1$ sends charge to $f$ in Step~6.]{\label{fig:1-in-1-1-in-3_7}
    {\includegraphics[width=5cm]{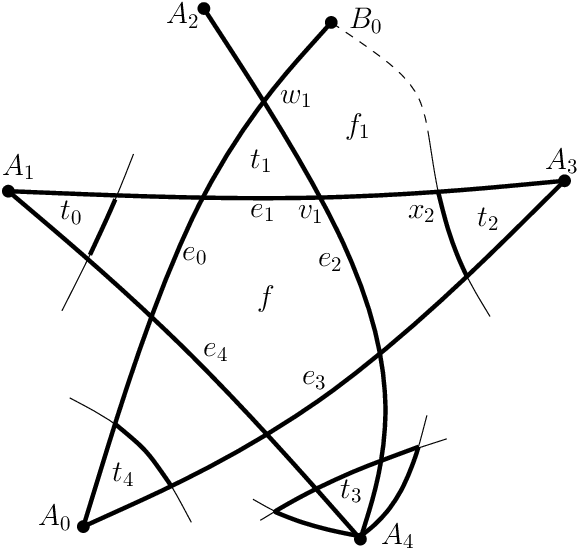}}}
	\caption{Subcase 1.2 in the proof of Lemma~\ref{lem:ch_6-0-pentagon-2/3}: $f$ sends $1/3$ units of charge to $t_1$ in Step~1
			and $1/3$ units of charge to $t_3$ in Step~3.}
	\label{fig:1-in-1-1-in-3-1.2}
\end{figure}
Since each of $(A_2,B_2)$ and $(A_0,B_0)$ contains four crossings, $e_3$ is an edge of $t_3$ and $e_4$ is an edge of $t_4$.
But then one neighbor of $t_3$ (the face $f_3$) is incident to two original vertices ($A_0$ and $B_2$) and therefore $f$ could not have
contributed charge to $t_3$ in Step~3.

Suppose that $e_3$ is an edge of $t_3$ and refer to Figure~\ref{fig:1-in-1-1-in-3_5}.
Since the neighbors of $t_3$ are $1$-triangles it follows that each of the wedges of $t_2$ and $t_4$ 
contains exactly one $1$-quadrilateral. 
Observe also that one neighbor of $t_2$ is incident to two original vertices ($A_3$ and $A_4$), 
which implies that its other neighbor is either a $1$-quadrilateral or a $1$-triangle.
This in turn implies that $f_1$ cannot be a quadrilateral (because then both neighbors of $t_2$ would be incident
to two original vertices). 
Thus, the size of $f_1$ is at least five and it contains one original vertex.
Since $f_1$ is not a vertex-neighbor of a $0$-pentagon at $B_0$ and $x_2$,
and it contributes at most $1/6$ units of charge through $x_2v_1$ and each of its edges that are incident to $B_0$,
it follows that $f_1$ contributes at least $\frac{|f_1|+1-4-1/3-3/6-(|f_0|-3)/3}{|f_0|-2} \geq 1/6$ units of charge
to every face in $\mathcal{P}(f_1)$ (including $f$) in Step~6.

Therefore, assume that $e_3$ is not an edge of $t_3$.
This implies that $f_0$ is incident to $A_2$.
If $e_0$ is an edge of $t_0$, then $|V(f_0)| \geq 2$ and $|f_0|\geq 4$ and by Proposition~\ref{prop:step6}
$f_0$ contributes at least $1/3$ units of charge to $f$ in Step~6.
We may assume therefore that the wedge of $t_0$ contains exactly one $1$-quadrilateral
(more than one would imply five crossings on $(A_4,B_4)$).

If $e_4$ is an edge of $t_4$ (see Figure~\ref{fig:1-in-1-1-in-3_6}), then it follows 
from Proposition~\ref{prop:1/6} that $f_3$ compensates for the missing charge of $f$.
Thus, we may assume that $e_4$ is not an edge of $t_4$, and, for similar reasons, $e_2$ is not an edge of $t_2$ (refer to Figure~\ref{fig:1-in-1-1-in-3_7}).
Recall also that $e_0$ is not an edge of $t_0$.
Note that $x_2,v_1,w_1,B_0 \in V(f_1)$.
Observe that $f_1$ contributes $1/3$ units of charge to $B_0$ and to $t_1$ 
and at most $1/6$ units of charge through $v_1x_2$ (since that recipient of such a charge must be a neighbor of $t_3$).
Furthermore, $f_1$ does not contribute any charge through $B_0w_1$,
since the other face that is incident to this edge is incident to two original vertices.
If $f_1$ is a $1$-quadrilateral, then its immediate neighbor at $B_0x_2$ is incident to two original vertices, and therefore it also does not contribute any charge through this edge,
and thus $ch_5(f_1) \geq 1/6$.
We also have $\mathcal{P}(f_1)=\{f\}$ in this case, and so $f_1$ sends at least $1/6$ units of charge to $f$ in Step~6.
If $|f_1| \geq 5$, then $f_1$ contributes at least $\frac{|f_1|+1-4-1/3-3/6-(|f_1|-3)/3}{|f_1|-3} \geq 1/6$ units of charge to $f$ in Step~6.
%
% consider the clockwise chain from $B_0$ to $x_2$, 
% and observe that it contains $|f_1|-3$ edges
% and at most $|f_1|-4$ vertices through which $f_1$ sends charge in Step~6.
% Therefore, every face in $\mathcal{P}(f_1)$ (including $f$) receives from $f_1$ in Step~6 at least
% $\frac{|f_1|+1-4-2/3-1/6-(|f_1|-3)/3}{|f_1|-3} \geq 1/6$ units of charge.
%It is not hard to see that this is true also in the case that $|f_1| \geq 5$.

\bigskip
\noindent\underline{Case 2:} $ch_1(f) = 2/3$, $ch_3(f)=0$ and $ch_5(f) < 0$.
That is, $f$ contributes $1/3$ units of charge to exactly two $1$-triangles in Step~3.
Recall that we assume without loss of generality that $f$ sends $1/3$ units of charge to $t_1$ in Step~1.
By symmetry, it is enough to consider the cases that the edges through which
$f$ contributes charge in Step~3 are $e_2$ and $e_3$,
$e_2$ and $e_4$, $e_2$ and $e_0$, and $e_3$ and $e_4$.

\medskip
\noindent\underline{Subcase 2.1:} $f$ contributes charge through $e_2$ and $e_3$ in Step~3.
It follows from Proposition~\ref{prop:e_1-e_2} and the maximum number
of crossings per edge that the wedge of $t_3$ contains exactly one $0$-quadrilateral,
the wedge of $t_1$ contains no $0$-quadrilaterals
and the wedge of $t_2$ contains one or two $0$-quadrilaterals.
If the wedge of $t_2$ contains two $0$-quadrilaterals,
then it follows that the size of $f_0$ is at least four and this face is incident to $A_1$ and $A_2$
(see Figure~\ref{fig:1-in-1-3-in-3c}).
% \begin{figure}[ht]
%     \centering
%     \includegraphics[width=5cm]{1-in-1-3-in-3c}
% 	\caption{An illustration for Subcase 2.1 in the proof of Lemma~\ref{lem:ch_6-0-pentagon-2/3}: If $f$ contributes through $e_2$ and $e_3$ in Step~3,
% 				such that the wedge of $t_2$ contains two $0$-quadrilaterals,
% 				then $f$ receives at least $1/3$ units of charge from $f_0$ in Step~6.}
% 	\label{fig:1-in-1-3-in-3c}
% \end{figure}
\begin{figure}[ht]
    \centering
    \subfigure[If the wedge of $t_2$ contains two $0$-quadrilaterals,
				then $f$ receives at least $1/3$ units of charge from $f_0$ in Step~6.]{\label{fig:1-in-1-3-in-3c}
    {\includegraphics[width=4.5cm]{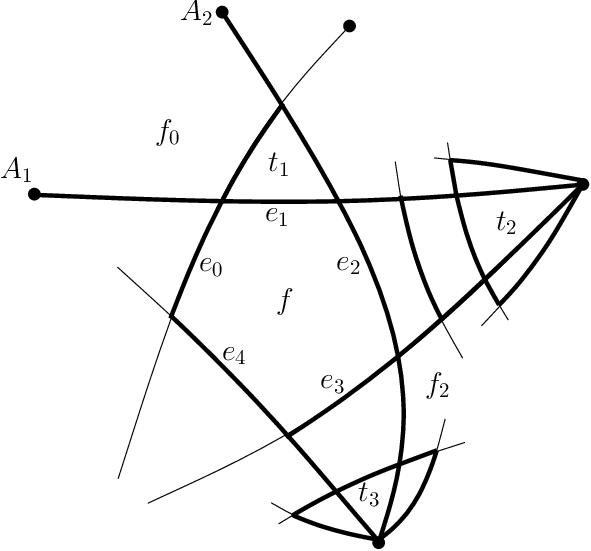}}}
    \hspace{10mm}
    \subfigure[$f_2$ is a $0$-quadrilateral. If $f$ contributes charge through $e_0$, 
				then $f_0$ contributes at least $1/6$ units of charge to $f$ in Step~6.]{\label{fig:1-in-1-2-in-3_100}
    {\includegraphics[width=5cm]{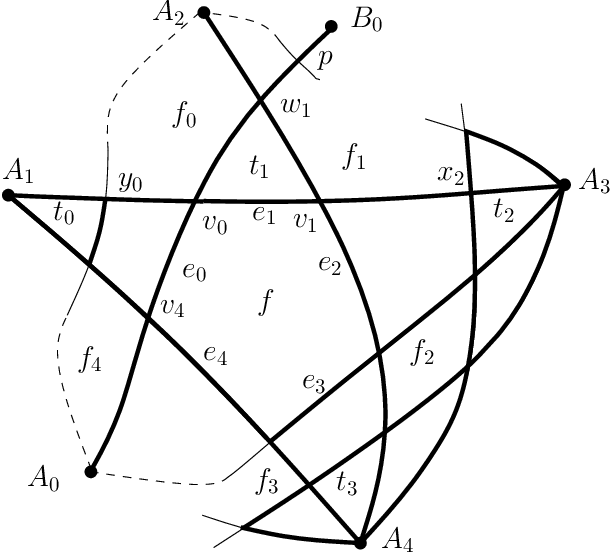}}}
	\caption{Subcase 2.1 in the proof of Lemma~\ref{lem:ch_6-0-pentagon-2/3}: $f$ sends $1/3$ units of charge to $t_1$ in Step~1,
			$1/3$ units of charge to $t_2$ and $t_3$ in Step~3.}
	%\label{fig:1-in-1-2-in-3-2.1}
\end{figure}
By Proposition~\ref{prop:step6} $f_0$ contributes at least $1/3$ units of charge
to $f$ in Step~6, and thus $ch_6(f) \geq 0$.

Suppose that the wedge of $t_2$ contains exactly one $0$-quadrilateral.
If $B_0 \in V(f_1)$, then it follows from Proposition~\ref{prop:2/3} that $f_1$ sends at least $2/3$ units of charge to $f$ in Step~6
and so $ch_6(f) \geq 0$.
Assume therefore that $B_0 \notin V(f_1)$ and let $p$ be the (only) crossing point between $B_0$ and $w_1$ on $(A_0,B_0)$.

Note that $|f_2| \geq 4$.
If $|f_2| \geq 5$, then by Proposition~\ref{prop:e_1-e_2-step3} $f_2$ contributes % *** vvv
at least $1/3$ units of charge to $f$ and thus $ch_6(f) \geq 0$.
Assume therefore that $|f_2|=4$, and refer to Figure~\ref{fig:1-in-1-2-in-3_100}.

If $f$ contributes charge to $t_4$ in Step~5, then $t_4$ must share $e_4$ with $f$,
since $(A_0,B_0)$ has already four crossings points ($v_4$, $v_0$, $w_1$ and $p$).
Therefore, it follows from Proposition~\ref{prop:1/6} that $f_3$ contributes at least $1/6$ units of charge to $f$ in such a case.
Thus, if $ch_6(f)<0$, then $f$ must contribute charge to $t_0$ in Step~5.

Assume that it does, consider the face $f_0$ and observe that it is incident to $A_2$ and that its size is at least four.
Therefore, if the wedge of $t_0$ contains no $0$-quadrilaterals, then $A_1 \in V(f_0)$ and $f_0$ contributes at
least $1/3$ units of charge to $f$ in Step~6 by Proposition~\ref{prop:step6}.
Assume therefore the wedge of $t_0$ contains exactly one $0$-quadrilateral and $v_0y_0$ is an edge of $f_0$ (see Figure~\ref{fig:1-in-1-2-in-3_100}).
Note that $f_0$ contributes at most $1/6$ units of charge through this edge, since the recipient of such a charge must be the $1$-triangle $f_4$ that has a neighbor that is incident to two original vertices.
Since $f_0$ contributes at most $1/6$ units of charge through its edges that are incident to $A_2$ and $A_2,y_0 \notin \mathcal{P}'(f_0)$
it follows that if $|f_0| \geq 5$ then $f_0$ contributes at least $\frac{|f_0|+1-4-1/3-3/6-(|f_0|-3)/3}{|f_0|-2} \geq 1/6$  units of charge to $f$ in Step~6.

If $f_0$ is a $1$-quadrilateral, then observe that it does not contribute charge through $y_0A_2$ since its immediate neighbor at this edge
is incident to $A_1$ and $A_2$.
$f_0$ does not contribute charge through $A_2w_1$ as well, since the recipient of such a charge must be a $1$-triangle whose
vertices are $A_2$, $w_1$ and $p$. 
However, since $ch_3(f_0)=1/3$ and the other neighbor of this triangle is incident to both $A_2$ and $B_0$,
it follows from Step~4 that $f_0$ does not contribute charge in this step.
Furthermore, $f_0$ does not contribute charge through $v_0y_0$ (to the $1$-triangle $f_4$), for otherwise
both neighbors of $t_0$ would be incident to two original vertices and hence $t_0$ would not receive charge from $f$ in Step~5 (see Figure~\ref{fig:1-in-1-2-in-3_100}).
Thus, $ch_5(f_0) \geq 1/3$ and $f_0$ contributes at least $1/6$ units of charge to $f$ also when $f_0$ is a $1$-quadrilateral.

\medskip
\noindent\underline{Subcase 2.2:} $f$ sends $1/3$ units of charge through $e_2$ and $e_4$ in Step~3.
Suppose that $e_1$ is not an edge of $t_1$ and refer to Figure~\ref{fig:1-in-1-2-in-3d}.
Since $(A_0,B_0)$ has four crossings, it follows that $e_4$ is an edge of $t_4$.
The edge $(A_2,B_2)$ also has four crossings, which implies that $v_2B_2$ is an edge in $M(G)$.
Therefore, the face that shares $e_3$ with $f$ is of size at least four and is incident to $B_2$.
Thus, $f$ does not contribute charge through $e_3$, and hence must contribute
$1/6$ units of charge through $e_0$.
This implies that $A_1$ and $B_4$ coincide, which in turn implies that $e_0$
is not an edge of $t_0$ and therefore the wedge of $t_2$ contains exactly one $0$-quadrilateral
(it cannot contain no $0$-quadrilaterals by Proposition~\ref{prop:e_1-e_2}).
Consider the face $f_2$ and observe that it does not contribute charge through $B_2v_2$ and $v_2y_2$, and that it contributes at most $1/6$ units of charge
through its other edge that is incident to $B_2$ and through $y_2y'_2$. % and through the edge that it shares with the $1$-triangle that is a neighbor of $t_2$.
Note also that $f_2$ is not a vertex-neighbor of a $0$-pentagon at $B_2$, $y_2$, and $y'_2$.
It follows that in Step~6 it contributes at least $\frac{|f_2|+1-4-1/3-2/6-(|f_2|-4)/3}{|f_2|-3} \geq 1/6$ units of charge to $f$ and thus $ch_6(f) \geq 0$.
% Considering the face $f_2$, it is not hard to see that it contributes 
% at least $1/6$ units of charge to $f$ and thus $ch_6(f) \geq 0$.

Assume therefore that $e_1$ is an edge of $t_1$.
We may also assume that the wedge of $t_2$ contains exactly one $0$-quadrilateral.
Indeed, by Proposition~\ref{prop:e_1-e_2} it must contain at least one $0$-quadrilateral.
Suppose that the wedge of $t_2$ contains two $0$-quadrilaterals and refer to Figure~\ref{fig:1-in-1-2-in-3_3}.
\begin{figure}[ht]
    \centering
    \subfigure[If $e_1$ is not an edge of $t_1$, then $f_2$ sends at least $1/6$ units of charge to $f$ in Step~6.]{\label{fig:1-in-1-2-in-3d}
    {\includegraphics[width=4.5cm]{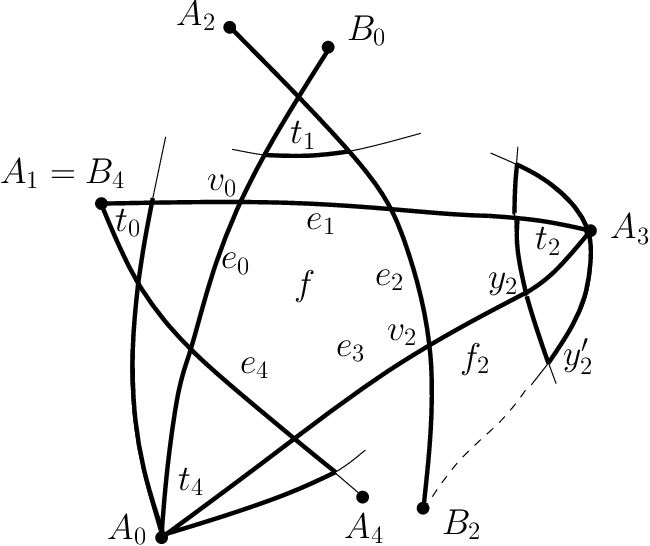}}}
    \hspace{2mm}
    \subfigure[If the wedge of $t_2$ contains two $0$-quadrilaterals,
				then $f$ does not contribute charge through $e_3$ and $e_0$.]{\label{fig:1-in-1-2-in-3_3}
    {\includegraphics[width=4.5cm]{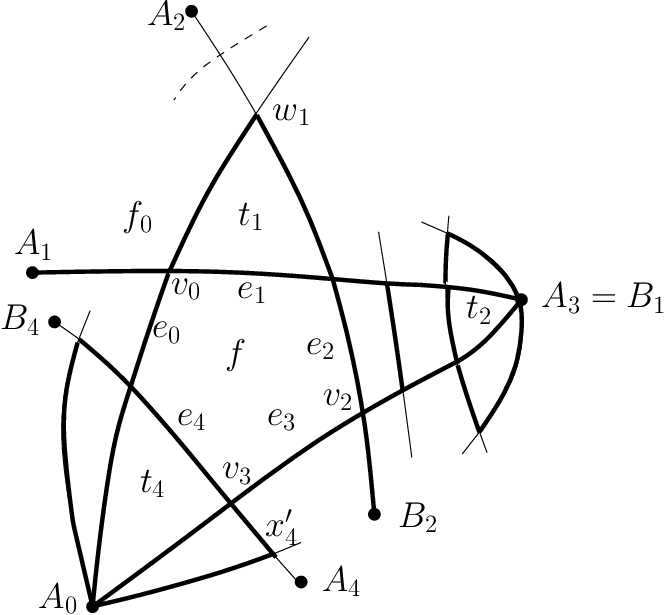}}}
    \hspace{2mm}
    \subfigure[If $e_4$ is not an edge of $t_4$, then $f_1$ sends charge to $f$ in Step~6.]{\label{fig:1-in-1-2-in-3_4}
    {\includegraphics[width=4.5cm]{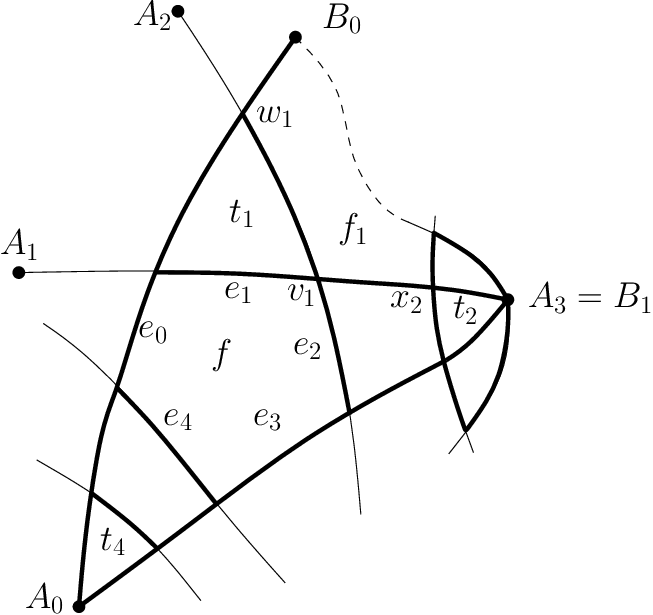}}}
	\caption{Subcase 2.2 in the proof of Lemma~\ref{lem:ch_6-0-pentagon-2/3}: $f$ contributes charge through $e_1$ in Step~1 and through $e_2$ and $e_4$ in Step~3.}
	\label{fig:1-in-1-2-in-3-2.1}
\end{figure}
Then $e_4$ must be an edge of $t_4$ and $v_0A_1$ must be an edge of $f_0$.
If $A_2w_1$ is also an edge of $f_0$, then $|f_0| \geq 4$ and $|V(f_0)| \geq 2$,
and by Proposition~\ref{prop:step6} $f_0$ contributes at least $1/3$ units of charge to $f$ in Step~6.
Therefore, we may assume that the (open) segment $A_2w_1$ contains a crossing point.
It follows that $v_2B_2$ is an edge in $M(G)$
and that $f$ does not contribute charge through $e_3$,
since the face that shares this edge with $f$ is incident to $B_2$ and its size is at least four (it is also incident to $v_2$, $v_3$ and $x'_4$).
Thus, $f$ must contribute charge through $e_0$.
However, the face that shares $e_0$ with $f$ is also of size at least four
and is incident to an original vertex ($A_1$), and therefore $f$ does not contribute
charge through $e_0$ either which implies that $ch_5(f) \geq 0$.

We may assume therefore that the wedge of $t_2$ contains exactly one $0$-quadrilateral.
Consider now the case that $e_4$ is not an edge of $t_4$, and refer to Figure~\ref{fig:1-in-1-2-in-3_4}.
It follows that $B_0 \in V(f_1)$ and therefore by Proposition~\ref{prop:2/3} $f_1$ sends at least $2/3$ units of charge
to $f$ in Step~6.
% Notice that $|f_1| \geq 5$ and $|V(f_1)| \geq 1$.
% $f_1$ contributes at most $1/6$ units of charge through each of its two edges that are incident to $x_2$
% and through each of its two edges that are incident to $B_0$.
% Note also that $f_1$ is not a vertex-neighbor of a $0$-pentagon at $B_0$, $x_2$, 
% and at the vertex that precedes $x_2$ in a clockwise order of the vertices of $f_1$.
% Therefore, $f_1$ contributes in Step~6 at least $\frac{|f_1|+1-4-1/3-1/3-4/6-(|f_1|-5)/3}{|f_1|-3} \geq 1/3$
% units of charge to every face in $\mathcal{P}(f_1)$, including $f$.

It remains to consider the case that $e_4$ is an edge of $t_4$.
If $ch_5(f)<0$ then $f$ must have contributed charge through $e_3$ or $e_0$ in Step~5.
Suppose that $f$ sends $1/6$ units of charge through $e_3$ in Step~5,
and refer to Figure~\ref{fig:1-in-1-2-in-3_5}.
\begin{figure}[ht]
    \centering
    \subfigure[If $f$ contributes charge through $e_3$, then $f_2$ sends charge to $f$ in Step~6.]{\label{fig:1-in-1-2-in-3_5}
    {\includegraphics[width=4.5cm]{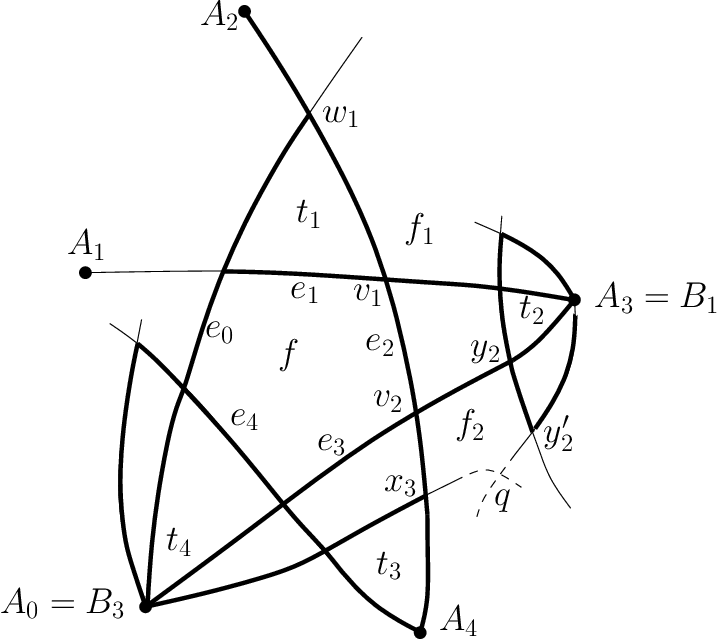}}}
    \hspace{2mm}
    \subfigure[$f$ sends charge through $e_0$ and there is a crossing points between $A_2$ and $w_1$.
				Then $f_2$ sends at least $1/3$ units of charge to $f$ in Step~6.]{\label{fig:1-in-1-2-in-3_6a}
    {\includegraphics[width=4.5cm]{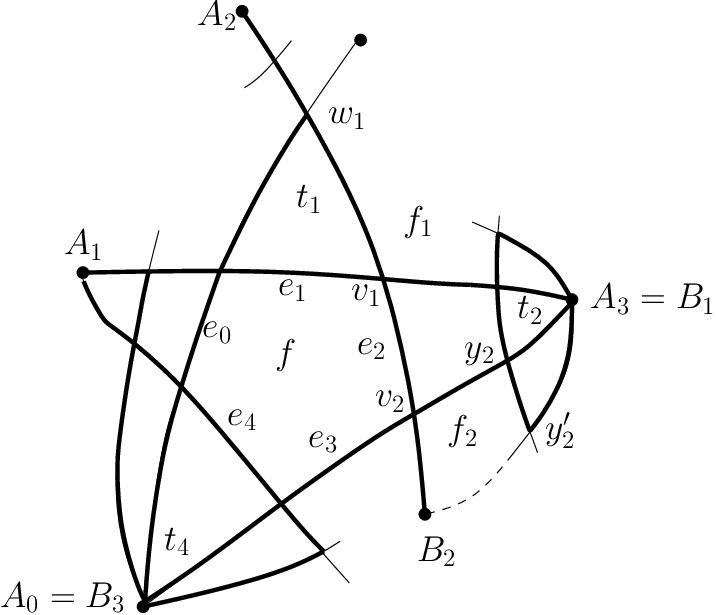}}}
    \hspace{2mm}
    \subfigure[$f$ sends charge through $e_0$ and $w_1A_2$ is an edge in $M(G)$.
				Then $f_1$ sends charge to $f$ in Step~6.]{\label{fig:1-in-1-2-in-3_6}
    {\includegraphics[width=4.5cm]{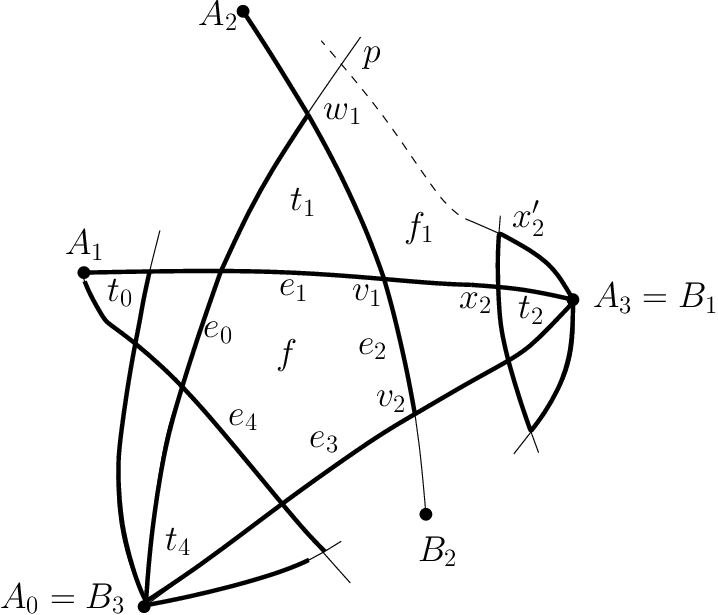}}}
	\caption{Subcase 2.2 in the proof of Lemma~\ref{lem:ch_6-0-pentagon-2/3}. $e_4$ is an edge of $t_4$.}
	\label{fig:1-in-1-2-in-3-2.2b}
\end{figure}
Note that the size of $f_2$ is at least five (if $|f_2|=4$, then there are two parallel edges between $A_3$ and $B_3$) 
and let $q$ be its vertex that follows $y'_2$.
Observe that $f_2$ contributes no charge through $v_2y_2$ (since $|f_1| \geq 5$) 
and at most $1/6$ units of charge through $x_3v_2$ and $y_2y'_2$ by Proposition~\ref{prop:1-triangle-neighbor}.
Note that $f_2$ is not a vertex-neighbor of a $0$-pentagon at $x_3$, $y_2$ and $y'_2$.
If $|f_2|=5$, then if $f_2$ contributes charge through $qx_3$ or $y'_2q$,
then the recipients of this charge must be $1$-triangles.
It follows that in such a case $q \notin \mathcal{P}'(f_2)$.
Moreover, if $f_2$ contributes $1/3$ units of charge through one of these two edges, then it does not contribute charge through the other edge.
%However, if $f_2$ contributes $1/3$ units of charge through $qx_3$, then it does not contribute charge through $pq$.
% However, if $f_2$ contributes $1/3$ charge through one of these edges, then it does not contribute charge through the other edge
% and moreover through $q$ in Step~6.
It follows that $f_2$ sends at least $1/6$ units of charge to $f$ in Step~6 and $ch_6(f) \geq 0$.
% Then $|f_2| \geq 5$  and it is not hard to see that $f_2$ sends at least $1/3$ units of
% charge to $f$ in Step~6.
If $|f_2| \geq 6$, then again $f_2$ sends at least 
$\frac{|f_2|-4-2/6-(|f_2|-3)/3}{|f_2|-3} \geq 1/6$ units of charge to $f$ in Step~6. 
% If $|f_2| \geq 6$ then consider the clockwise chain from $p$ to $x_3$, 
% and observe that it contains $|f_2|-3$ edges
% and at most $|f_2|-4$ vertices through which $f_2$ sends charge in Step~6.
% However, if $f_2$ contributes ($1/6$ units of) charge through $pq$ then it does not contribute charge through $q$ in Step~6.
% Therefore, every face in $\mathcal{P}(f_2)$ (including $f$) receives from $f_2$ in Step~6 at least
% $\min\left\{\frac{|f_2|-4-2/6-(|f_2|-3)/3}{|f_2|-4},\frac{|f_2|-4-3/6-(|f_2|-4)/3}{|f_2|-3}\right\} \geq 1/3$ units of charge.

Finally, suppose that $f$ %does not contribute (at least $1/6$ units of) charge through $e_3$ and 
sends $1/6$ units of charge through $e_0$ in Step~5.
If there is a crossing point between $A_2$ and $w_1$ on $(A_2,B_2)$,
then $B_2 \in V(f_2)$ and $|f_2| \geq 4$ (see Figure~\ref{fig:1-in-1-2-in-3_6a}).
%Let $z$ be the vertex that follows $y_2$ in $f_2$.
Observe that $f_2$ does not contribute charge through $B_2v_2$ (since its immediate neighbor at this edge is not a $1$-triangle)
and $v_2y_2$ (since $|f_1| \geq 5$), and at most $1/6$ units of charge through $y_2y_2'$ (by Proposition~\ref{prop:1-triangle-neighbor}) and through 
its other edge that is incident to $B_2$.
Since $B_2,y_2,y'_2 \notin \mathcal{P}'(f_2)$ it follows that $f_2$ contributes at least
$\frac{|f_2|+1-4-1/3-2/6-(|f_2|-4)/3}{|f_2|-3} \geq 1/3$ units of charge to $f$ in Step~6 and so $ch_6(f)\geq 0$.

If $w_1A_2$ is an edge in $M(G)$, then
consider the face $f_1$ and note that its size is at least five.
Let $p$ be the other vertex of $f_1$ that is adjacent to $w_1$ but $v_1$.
If $p=B_0$, then it follows from Proposition~\ref{prop:2/3} that $f_1$ sends at least $2/3$ units of charge to $f$ in Step~6.
Assume therefore that $p$ is a crossing point and refer to Figure~\ref{fig:1-in-1-2-in-3_6}. 
Observe that $f_1$ sends at most $1/6$ units
of charge through $x'_2x_2$ and $x_2v_1$ by Proposition~\ref{prop:1-triangle-neighbor}.
Since $w_1A_2$ is an edge in $M(G)$ and $(A_0,B_0)$ already has four crossings it follows
that $f$ contributes at most $1/6$ units of charge through $w_1p$ as well.
Note also that $x_2,x'_2,w_1 \notin \mathcal{P}'(f_1)$.
If $|f_1|=5$, then by Proposition~\ref{prop:1-triangle-neighbor} $f$ sends at most $1/6$ units of charge through $px'_2$.
Moreover, $f_1$ cannot send charge through both $px'_2$ and $w_1p$,
because then the $1$-triangle that gets the charge through $px'_2$ would have two neighbors 
such that each of them is incident to two original vertices.
Furthermore, if $f_1$ sends charge through one of these edges, then $p \notin \mathcal{P}'(f_1)$.
% Therefore $ch_5(f_1) \geq 1/6$.
% Note that $f_1$ is not a vertex-neighbor of a $0$-pentagon at $x_2$, $q$, $p$ and $w_1$.
% If $ch_5(f_1)=1/6$ then $\mathcal{P}(f_1)=\{f\}$.
% If $\mathcal{P}(f_1)$ contains another face then this face must intersect $f_1$ exactly at $p$,
% but in this case $f_1$ does not contribute charge through $pq$ and $w_1p$ and so $ch_5(f_1) \geq 1/3$.
It follows that if $|f_1|=5$, then $f_1$ sends at least $1/6$ units of charge to $f$ in Step~6.
If $|f_1| \geq 6$, then $f$ sends at least 
$\frac{|f_1|-4-1/3-3/6-(|f_1|-4)/3}{|f_1|-4} \geq 1/6$ units of charge to $f$ in Step~6.
Recall that if $f$ sends $1/6$ units of charge through $e_3$, then it gets at least $1/6$ units of charge from $f_2$ in Step~2.
Therefore, $ch_6(f) \geq 0$.
%
% If $|f_1| \geq 5$, then consider the clockwise chain from $w_1$ to $q$, 
% and observe that it contains $|f_1|-3$ edges
% and at most $|f_1|-4$ vertices through which $f_1$ sends charge in Step~6.
% Recall that $f_1$ contributes at most $1/6$ units of charge through $w_1p$.
% Therefore, every face in $\mathcal{P}(f_1)$ (including $f$) receives from $f_1$ in Step~6 at least
% $\frac{|f_1|-4-3/6-1/3-(|f_1|-4)/3}{|f_1|-3} \geq 1/6$ units of charge.
% If $|f_1| > 5$, then every `extra' edge might add a $0$-pentagon to $\mathcal{P}(f_1)$ and $f_1$ might contribute at most $1/3$ through it,
% but it also increases the initial charge of $f_1$ by one, and therefore, still every $0$-pentagon in $\mathcal{P}(f_1)$
% receives at least $1/6$ units of charge from $f_1$.
%Note that one neighbor of $t_0$ is a $2$-triangle, and therefore $ch_5(f) \geq -1/6$.
%Thus, $f$ ends up with a non-negative charge.

\medskip
\noindent\underline{Subcase 2.3:} $f$ sends $1/3$ units of charge through $e_2$ and $e_0$ in Step~3.
It follows from Proposition~\ref{prop:e_1-e_2} and the maximum number of crossings per edge
that each of the wedges of $t_2$ and $t_0$ contains exactly one $0$-quadrilateral.

Suppose first that $e_1$ is not an edge of $t_1$ and refer to Figure~\ref{fig:1-in-1-2-in-3_7}.
Consider the face $f_2$ and observe that its size is at least four and it is incident to $B_2$.
% Let $p$ be the vertex that follows $y_2$ in a clockwise order of the vertices of $f_2$.
Note that $f_2$ does not contribute charge through $v_2y_2$, for otherwise $f_1$ and its immediate neighbor at $y_1x'_2$ must be $0$-quadrilaterals, which would imply that the edge of $G$ that contains $x_2y_2$ would have more than four crossings. 
$f_2$ contributes at most $1/6$ units of charge through $y_2y'_2$ (by Proposition~\ref{prop:1-triangle-neighbor}) and each of its edges that are incident to $B_2$.
Observe also that $f_2$ is not a vertex-neighbor of a $0$-pentagon at $B_2$, $y_2$ and $y'_2$.
Therefore $f_2$ contributes at least $\frac{|f_2|+1-4-1/3-3/6-(|f_2|-4)/3}{|f_2|-3} \geq 1/6$ units of charge to $f$ in Step~6. 
By symmetry, so does $f_4$ and therefore $ch_6(f) \geq 0$.
\begin{figure}[ht]
    \centering
    \subfigure[If $e_1$ is not an edge of $t_1$, then $f_2$ sends charge to $f$ in Step~6.]{\label{fig:1-in-1-2-in-3_7}
    {\includegraphics[width=5cm]{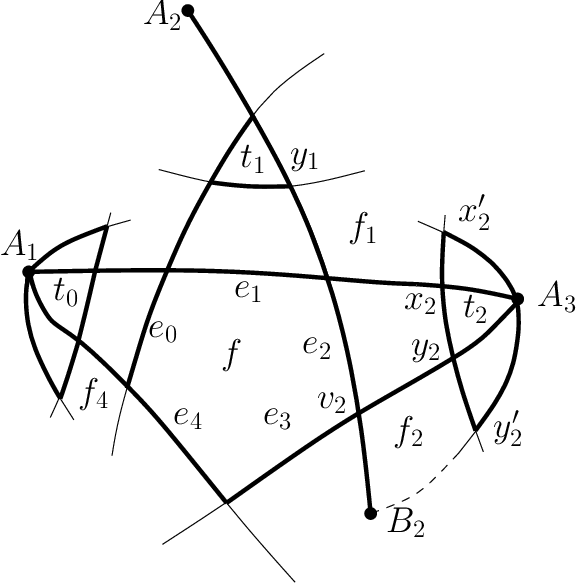}}}
    \hspace{10mm}
    \subfigure[If $e_1$ is an edge of $t_1$ then $f_0$ or $f_2$ send charge to $f$ in Step~6. ]{\label{fig:1-in-1-2-in-3_8}
    {\includegraphics[width=5cm]{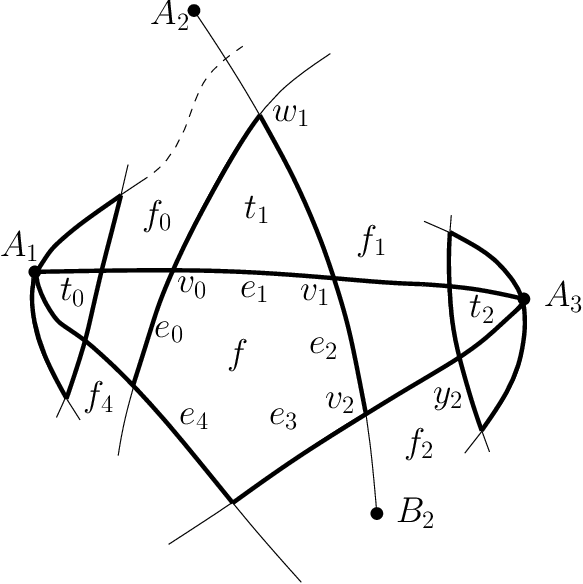}}}
	\caption{Subcase 2.3 in the proof of Lemma~\ref{lem:ch_6-0-pentagon-2/3}: $f$ contributes $1/3$ units of charge to $t_1$ in Step~1,
			and $1/3$ units of charge to each of $t_0$ and $t_2$ in Step~3.}
	\label{fig:1-in-1-2-in-3-2.2}
\end{figure}

Suppose now that $e_1$ is an edge of $t_1$, and refer to Figure~\ref{fig:1-in-1-2-in-3_8}.
Consider the face $f_0$ and observe that its size is at least five.
If $A_2$ is a vertex of $f_0$, then by Proposition~\ref{prop:2/3} $f_0$ sends at least $2/3$ units of charge to $f$ in Step~6.
Assume therefore that there is a crossing point between $A_2$ and $w_1$ on $(A_2,B_2)$
and thus $B_2v_2$ is an edge of $f_2$ (see Figure~\ref{fig:1-in-1-2-in-3_8}).
In this case, as in the case that $e_1$ is not an edge of $t_1$, 
it follows, that $f_2$ contributes at least $1/6$ units of charge to $f$ in Step~6 (note that $|f_1|\geq 5$ and so as before $f_2$ does not contribute charge through $v_2y_2$).
By symmetry, $f$ also receives at least $1/6$ units of charge from $f_1$ or $f_4$
and therefore ends up with a non-negative charge.
%
%
% \noindent\underline{Subcase 2.4:} $f$ sends $1/3$ units of charge to $t_3$ and $t_4$ in Step~3.
% We may assume that $ch_5(g)=-1/6$, for otherwise one of the previous cases applies.
% It follows from Proposition~\ref{prop:e_1-e_2} that $e_3$ (resp., $e_4$) is not an edge of $t_3$ (resp., $t_4$).
% Since there are at most four crossings per edge, the wedge of $t_3$ (resp., $t_4$) contains exactly one $1$-quadrilateral,
% and $e_1$ is an edge of $t_1$ (see Figure~\ref{fig:1-in-1-2-in-3-2.4}).
% \begin{figure}[ht]
%     \centering
%     \includegraphics[width=4cm]{1-in-1-2-in-3_9}
% 	\caption{Subcase 2.4 in the proof of Lemma~\ref{lem:ch_6-0-pentagon-2/3}.}
% 	\label{fig:1-in-1-2-in-3-2.4}
% \end{figure}
% Denote by $f'$ the face whose intersection with $f$ is exactly $x$.
% Observe that $f'$ may contribute at most $1/6$ units of charge through each of its depicted (bold) edges.
% It follows that $f'$ sends at least $1/6$ units of charge to $f'$ in Step~6 and $f$ ends up with a non-negative charge.

\medskip
\noindent\underline{Subcase 2.4:} $f$ sends $1/3$ units of charge to $t_3$ and $t_4$ in Step~3.
By Proposition~\ref{prop:e_1-e_2} and the maximum number of crossings per edge,
each of the wedges of $t_3$ and $t_4$ contains exactly one $0$-quadrilateral.
It follows that $|f_3| \geq 4$.
If $|f_3| \geq 5$, then by Proposition~\ref{prop:e_1-e_2-step3} $f_3$ sends at least $1/3$ units of charge % *** vvv
to $f$ in Step~6 and thus $ch_6(f) \geq 0$.

Suppose therefore that $|f_3|=4$, that is, $y'_3$ and $x'_4$ coincide (see Figure~\ref{fig:1-in-1-2-in-3-2.4}).
\begin{figure}[ht]
    \centering
    \includegraphics[width=5cm]{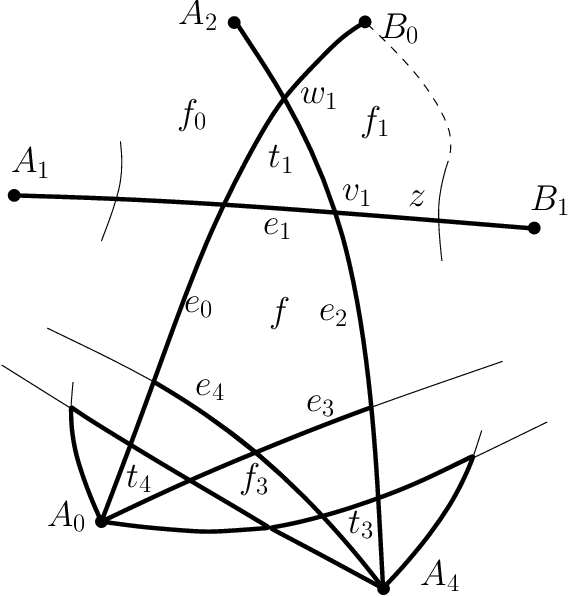}
	\caption{An illustration for Subcase 2.4 in the proof of Lemma~\ref{lem:ch_6-0-pentagon-2/3}: $f$ contributes charge through $e_3$ and $e_4$ in Step~3 and $|f_3|=4$.}
	\label{fig:1-in-1-2-in-3-2.4}
\end{figure}
Consider the face $f_1$ and observe that
its size is at least four and it is incident to $B_0$.
If it is also incident to $B_1$, then by Proposition~\ref{prop:step6} it sends at least $1/3$ units of charge to $f$ in Step~6 and thus $ch_6(f) \geq 0$.
Assume therefore that $f_1$ is not incident to $B_1$ and let $z$ be its other vertex but $v_1$ that lies on $(A_1,B_1)$.
By symmetry, we may also assume that $A_1$ is not a vertex of $f_0$, and therefore $z$ is the only crossing point between $v_1$ and $B_1$ on $(A_1,B_1)$.
Note that $w_1,B_0,z \notin \mathcal{P}'(f_1)$.
Therefore, if $|f_1| \geq 5$, then it follows from Proposition~\ref{prop:step6} that $f_1$ contributes at least $1/6$ units of charge to $f$ in Step~6.

Suppose that $f_1$ is a $1$-quadrilateral. 
Note that $f_1$ contributes $1/3$ units of charge through $v_1w_1$ in Step~1 and 
does not contribute charge through $w_1B_0$ and $B_0z$ since each of the faces that share these edges with $f_1$ is incident to two original vertices.
Considering the edge $zv_1$, observe that $f_1$ does not contribute charge through it in Step~1 as this would imply two parallel edges between $A_0$ and $B_0$.
Therefore, if $f_1$ contributes charge through $zv_1$, then the $1$-triangle that receives this charge is a neighbor of $t_3$
and hence $f_1$ contributes at most $1/6$ units of charge through $zv_1$ (since this neighbor is incident to a face with two original vertices).
Thus, $ch_5(f_1) \geq 1/6$ and $f$ gets all of this extra charge in Step~6.
By symmetry, $f$ also gets at least $1/6$ units of charge from $f_0$ and ends up with a non-negative charge.

\bigskip
\noindent\underline{Case 3:} $f$ contributes charge through at least three edges in Step~3.
Suppose first that $f$ contributes charge through three consecutive edges on the boundary of $f$, say $e_2$, $e_3$, and $e_4$.
Then it follows from Proposition~\ref{prop:e_1-e_2} and the maximum number
of crossings per edge that the wedge of each of the corresponding $1$-triangles contains exactly one $0$-quadrilateral
and that the wedge of $t_1$ contains no $0$-quadrilaterals (see Figure~\ref{fig:1-in-1-3-in-3z}).
\begin{figure}[ht]
    \centering
    \subfigure[If $f$ contributes through $e_2$, $e_3$ and $e_4$ in Step~3,
				then it receives at least $2/3$ units of charge from $f_1$ in Step~6.]{\label{fig:1-in-1-3-in-3z}
    {\includegraphics[width=5cm]{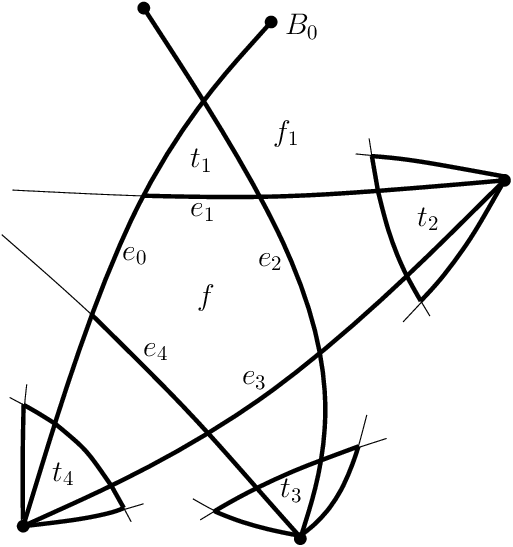}}}
    \hspace{10mm}
    \subfigure[If $f$ contributes through $e_2$, $e_3$ and $e_0$ in Step~3, then it receives at least $2/3$ units of charge from $f_0$ in Step~6.]{\label{fig:1-in-1-3-in-3y}
    {\includegraphics[width=5cm]{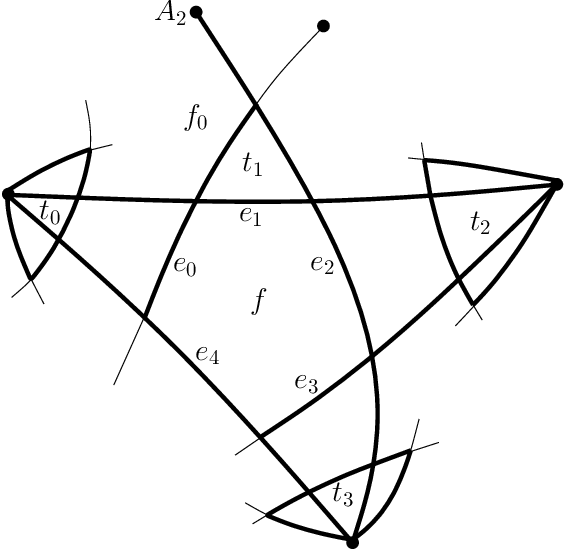}}}
	\caption{Case~3 in the proof of Lemma~\ref{lem:ch_6-0-pentagon-2/3}: $f$ sends $1/3$ units of charge to $t_1$ in Step~1
			and $1/3$ units of charge to three $1$-triangles in Step~3.}
	\label{fig:1-in-1-3-in-3yz}
\end{figure}
Moreover, $B_0 \in V(f_1)$ and therefore, by Proposition~\ref{prop:2/3} $f_1$ sends at least $2/3$ units of charge to $f$
and $ch_6(f) \geq 0$.

Suppose now that $f$ contributes charge in Step~3 through three edges that are not consecutive on its boundary.
By symmetry, we may assume that these edges are $e_2$, $e_3$ and $e_0$.
Note that we may also assume that $f$ does not contribute charge through $e_4$ in Step~3,
for the case of $f$ contributing in Step~3 through three consecutive edges was already considered.
It follows from Proposition~\ref{prop:e_1-e_2} and the maximum number
of crossings per edge that the wedge of each of the $1$-triangles $t_2$, $t_3$ and $t_0$ contains exactly one $0$-quadrilateral
and that the wedge of $t_1$ contains no $0$-quadrilaterals (see Figure~\ref{fig:1-in-1-3-in-3y}).
Moreover, $A_2 \in V(f_0)$ and therefore, by Proposition~\ref{prop:2/3} $f_0$ sends at least $2/3$ units of charge to $f$
and $ch_6(f) \geq 0$.

This concludes Case~3 and the proof of Lemma~\ref{lem:ch_6-0-pentagon-2/3}.
\end{proof}

It remains to consider the final charge of a $0$-pentagon that does not contribute charge in Step~1.

\begin{lem}\label{lem:ch_6-0-pentagon-1}
Let $f$ be a $0$-pentagon such that $ch_1(f)=1$ and $ch_5(f)<0$. Then $ch_6(f) \geq 0$.
\end{lem}

\begin{proof}
Suppose that $ch_1(f)=1$ and $ch_5(f)<0$.
Then $f$ contributes charge to exactly two, three, four or five $1$-triangles in Step~3.
We consider each of these cases separately.

\smallskip
\noindent\underline{Case 1:} $ch_3(f)=1/3$ and $ch_5(f)=-1/6$.
That is, $f$ contributes $1/3$ units of charge to two $1$-triangles in Step~3 
and contributes $1/6$ units of charge to three $1$-triangles in Step~5.
We may assume without loss of generality that in Step~3 either $f$ contributes charge to $t_1$ and $t_2$,
or it contributes charge to $t_1$ and $t_3$.

\medskip
\noindent\underline{Subcase 1.1:} $f$ contributes charge to $t_1$ and $t_2$ in Step~3
and to $t_3,t_4,t_0$ in Step~5.
Recall that by Proposition~\ref{prop:e_1-e_2} $e_1$ cannot be an edge of $t_1$ and $e_2$ cannot be an edge of $t_2$.

We claim that neither of the wedges of $t_1$ and $t_2$ contains two $0$-quadrilaterals.
Suppose, for contradiction, that the wedge of $t_2$ contains two $0$-quadrilaterals and refer to Figure~\ref{fig:2-in-3_3-in-5_1.1}.
\begin{figure}[ht]
    \centering
    \includegraphics[width=6cm]{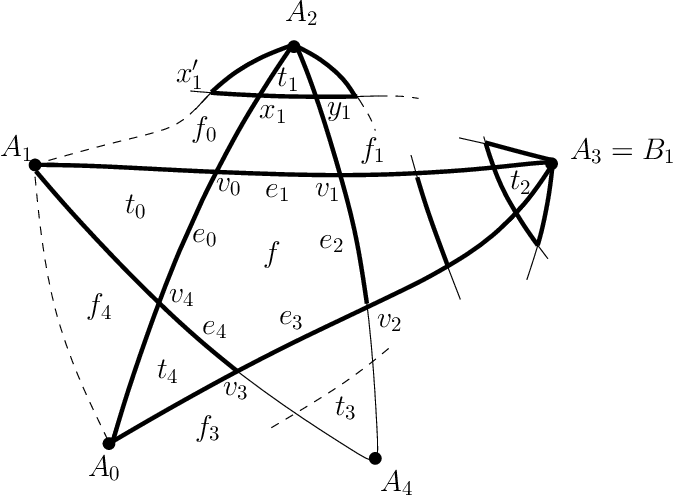}
	\caption{Subcase 1.1 in the proof of Lemma~\ref{lem:ch_6-0-pentagon-1}.
			If the wedge of $t_2$ contains two $0$-quadrilaterals then $t_0$ receives charge from both of its
				neighbors in Step~4, and no charge from $f$.}
	\label{fig:2-in-3_3-in-5_1.1}
\end{figure}
Since $(A_1,B_1)$ has four crossings it follows that $e_0$ is an edge of $t_0$.
Similarly, $e_4$ is an edge of $t_4$.
Note that $f_4$ is incident to $A_0$ and $A_1$ and therefore it contributes charge to $t_4$ in Step~4.
Therefore the other neighbor of $t_4$, $f_3$, must not contribute charge to $t_4$ in Step~4
(otherwise $f$ need not contribute charge to $t_4$ in Step~5).
It follows that $v_3A_4$ and $v_2A_4$ are not edges of $M(G)$.
Thus, the wedge of $t_1$ contains exactly one $0$-quadrilateral
and the size of $f_0$ is at least four.
Recall that $t_0$ receives charge from $f_4$ and from $f$, and therefore it should not receive charge from $f_0$ in Step~4,
which implies that $|f_0|=4$ and $ch_3(f_0) \leq 1/3$.
However, it follows from Proposition~\ref{prop:1-triangle-neighbor} that 
$f_0$ does not contribute charge through $x'_1x_1$ and $x_1v_0$ in Steps~1 and~3,
and therefore if $|f_0|=4$ then $ch_3(f_0) = 2/3$ and $f_0$ does contribute charge
to $t_0$ in Step~4.
%Indeed, $f_0$ does not contribute charge to the $1$-triangle that is a neighbor of $t_1$ in Step~3,
%since one neighbor of this $1$-triangle is incident to two original vertices.
%Also, $f_0$ does not contribute charge through $v_0x_1$ to a $0$- or $1$-triangle in Step~1 or in Step~3,
%for otherwise $f_1$ must be $0$-quadrilateral and the edge of $G$ that contains $x_1y_1$ either crosses %$(A_1,B_1)$
%(in case $f_0$ contributes to a $0$-triangle) or has $A_3$ as an endpoint (in case $f_0$ contributes to a %$1$-triangle).
%However, in both cases there is an edge of $G$ with more than four crossings.
%Thus, $ch_3(f_0)=2/3$ and so $f_0$ contributes charge to $t_0$ in Step~4 (as does $f_4$),
%which implies that $f$ does not contribute charge to $t_0$ in Step~5.

Therefore, each of the wedges of $t_1$ and $t_2$ contains exactly one $0$-quadrilateral.
It follows that $|f_1| \geq 4$.
If $|f_1| \geq 5$, then by Proposition~\ref{prop:e_1-e_2-step3} $f_1$ contributes at least $1/3$ units of charge to $f$  % ***
in Step~6, and so $ch_6(f) \geq 0$.
We assume therefore that $|f_1|=4$, that is, $y'_1$ and $x'_2$ coincide.
If $e_3$ is an edge of $t_3$, then by Proposition~\ref{prop:1/6} $f_2$ sends at least $1/6$ units of charge to $f$ in Step~6 and so $ch_6(f) \geq 0$.
%Suppose that $e_3$ is an edge of $t_3$ and consider the face $f_2$ (see Figure~\ref{fig:0-in-0-2-in-3a}).
\begin{figure}[ht]
    \centering
%    \subfigure[If $e_3$ is an edge of $t_3$, then $f_2$ sends at least $1/6$ units of charge to $f$ in Step~6.]{\label{fig:0-in-0-2-in-3a}
%    {\includegraphics[width=4.5cm]{0-in-0-2-in-3a}}}
%	\hspace{2mm}
    \subfigure[If $f_2$ is a $0$-pentagon, then it does not contribute charge through both $y'_2p$ and $px_3$.]{\label{fig:0-in-0-2-in-3k}
    {\includegraphics[width=4.5cm]{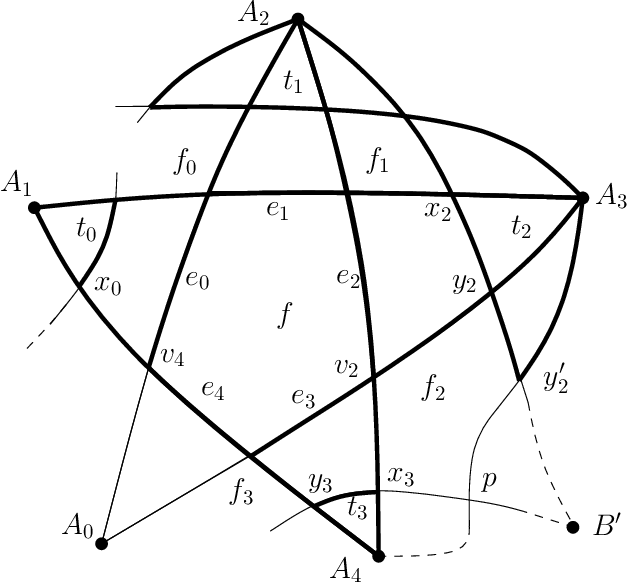}}}
	\hspace{2mm}
    \subfigure[If $f_2$ is a $0$-pentagon that contributes charge through $x_3v_2$ in Step~3,
    then it does not contribute charge through $px_3$ in Step~3.]{\label{fig:0-in-0-2-in-3kk}
    	{\includegraphics[width=4.5cm]{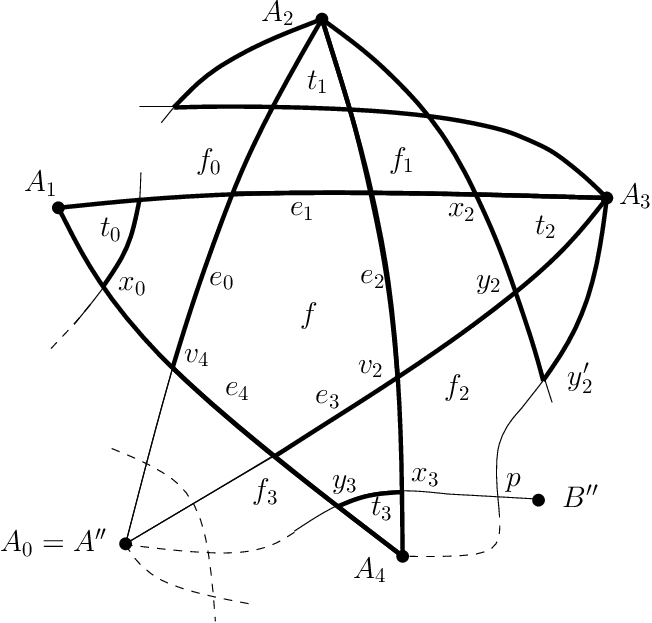}}}
%    \hspace{2mm}
%    \subfigure[If $f_2$ is a $0$-pentagon that contributes charge through $px_3$,%
%				then it does not contribute charge through $qp$.]{\label{fig:0-in-0-2-in-3m}
%    {\includegraphics[width=4.5cm]{0-in-0-2-in-3m}}}
	\caption{Illustrations for Subcase 1.1 in the proof of Lemma~\ref{lem:ch_6-0-pentagon-1}. 
			Both wedges of $t_1$ and $t_2$ contain exactly one $0$-quadrilateral and $|f_1|=4$.}
	\label{fig:0-in-1-2-in-3-1.1a}
\end{figure}
%$|f_2| \geq 4$ for otherwise there would be two parallel edges between $A_2$ and $A_4$.
%Since $e_3$ is an edge of $t_3$, $f_2$ is incident to $A_4$.
%Let $q$ be the vertex that follows $y_2$ in a clockwise order of the vertices of $f_2$.
%Note that $f_2$ contributes at most $1/6$ units of charge through each of its edges that is incident to $A_4$ and also through each of $v_2y_2$ and $y_2y'_2$ by Proposition~\ref{prop:1-triangle-neighbor}.
%Observe also that $f_2$ is not a vertex-neighbor of a $0$-pentagon at $A_4$, $y_2$ and $y'_2$.
%Hence, if $|f_2| \geq 5$, then it follows from Proposition~\ref{prop:step6} 
%that $f_2$ contributes at least $1/6$ units of charge to $f$ and thus $ch_6(f) \geq 0$.
%If $|f_2|=4$, then observe that it does not contribute charge through $y'_2A_4$ since the other face that
%shares this edge with $f_2$ is incident to two original vertices.
%Therefore $f$ also receives $1/6$ units of charge from $f_2$ in Step~6.
%Thus, if $e_3$ is an edge of $t_3$ we have $ch_6(f) \geq 0$.
By symmetry, $ch_6(f) \geq 0$ also if $e_0$ is an edge of $t_0$.

Assume, therefore that $e_3$ is not an edge of $t_3$ and that $e_0$ is not an edge of $t_0$.
It follows that the wedge of each of these $1$-triangles contains exactly one $0$-quadrilateral.
Consider the face $f_2$ and observe that its size is at least four.
We claim that if $|f_2| \geq 5$, then $f_2$ contributes at least $1/6$ units of charge to $f$ in Step~6.
Indeed, suppose that $|f_2| \geq 5$, let $p$ be its vertex that follows $y'_2$ (see Figure~\ref{fig:0-in-0-2-in-3k}).
Note that $f_2$ contributes at most $1/6$ units of charge through $v_2y_2$, $y_2y'_2$ and $y'_2p$ by Proposition~\ref{prop:1-triangle-neighbor} and that $x_3,y_2,y'_2 \notin \mathcal{P}'(f_2)$.
%It also follows from Observation~\ref{obs:extreme} that $f_2$ contributes at most $1/6$ units of charge through $qp$.
Therefore, if $|f_2| \geq 6$, then $f_2$ contributes at least 
$\frac{|f_2|-4-3/6-(|f_2|-3)/3}{|f_2|-3} \geq 1/6$ units of charge to $f$ in Step~6.
If $p \in V(G)$, then $f_2$ contributes at most $1/6$ units of charge through the edges that are incident to $p$ and at least $\frac{|f_2|+1-4-1/3-4/6-(|f_2|-4)/3}{|f_2|-4} \geq 1/6$ units of charge to $f$ in Step~6.

Suppose now that $|f_2|=5$ and $p \notin V(G)$, and let $(A_2,B')$ be the edge of $G$ that contains $x_2y_2$.
Note that $f_2$ may contribute charge through $px_3$ and $y'_2p$ only to $1$-triangles
since each of $(A_2,B_2)$ and $(A_2,B')$ already contains four
crossings among the vertices of $f_1$ and $f_2$.
Moreover, if $f_2$ is a wedge-neighbor of a $1$-triangle at $px_3$ and at $y'_2p$,
then the two neighbors of its wedge-neighbor at $y'_2p$ are incident to two original vertices,
and therefore $f_2$ does not contribute charge through $y'_2p$ (see Figure~\ref{fig:0-in-0-2-in-3k}).
Furthermore, if one of these two wedge-neighbors exists, then $p \notin \mathcal{P}'(f_2)$.
Thus if $p \in \mathcal{P}'(f_2)$, then $ch_5(f_2) \geq 5-4-1/3-2/6 = 1/3$
and so $f$ receives at least $1/6$ units of charge from $f_2$ in Step~6.

Assume therefore that $p \notin \mathcal{P}'(f_2)$, that is $\mathcal{P}(f_2) = \{f\}$.
If $f_2$ contributes at most $1/6$ units of charge through $x_3v_2$, then $ch_5(f_2) \geq 1/6$
and $f$ gets all this excess charge.
Otherwise, suppose that $f_2$ contributes $1/3$ units of charge through $x_3v_2$ and let
$(A'',B'')$ be the edge of $G$ that contains $px_3$ (see Figure~\ref{fig:0-in-0-2-in-3kk}).
Since $(A'',B'')$ has already three crossings ($p$, $x_3$ and $y_3$) it is impossible that $f_2$ contributes charge through $x_3v_2$ in Step~1 (to the $0$-triangle $f_3$), for otherwise
$(A'',B'')$ would have two more crossings (with $(A_3,A_0)$ and $(A_0,A_2)$).
Therefore $f_2$ must contribute through $x_3v_2$ to a $1$-triangle in Step~3 and thus $A''$ and $A_0$ coincide.
Since $(A_4,A_1)$ already has four crossings, it follows that $f_3$ cannot be the $1$-triangle
to which $f_2$ contributes charge through $x_3v_2$,
which in turn implies that $(A'',B'')$ has a crossing point between $A''$ and $y'_3$.
Thus, $p$ is an extreme crossing point on $(A'',B'')$ and therefore $f_2$ does not contribute
charge through $px_3$ in Step~3.
Recall that if $f_2$ contributes charge through $px_3$ then it does not contribute
charge through $y'_2p$.
It follows that $ch_5(f_2) \geq 1/6$ and that $f$ gets all of this excess charge.

By symmetry, if $|f_0| \geq 5$ then $ch_6(f) \geq 0$.
Assume therefore that both $f_0$ and $f_2$ are ($0$-) quadrilaterals.
Consider first the case that $e_4$ is an edge of $t_4$ (see Figure~\ref{fig:0-in-0-2-in-3q}).
\begin{figure}[ht]
    \centering
    \includegraphics[width=5cm]{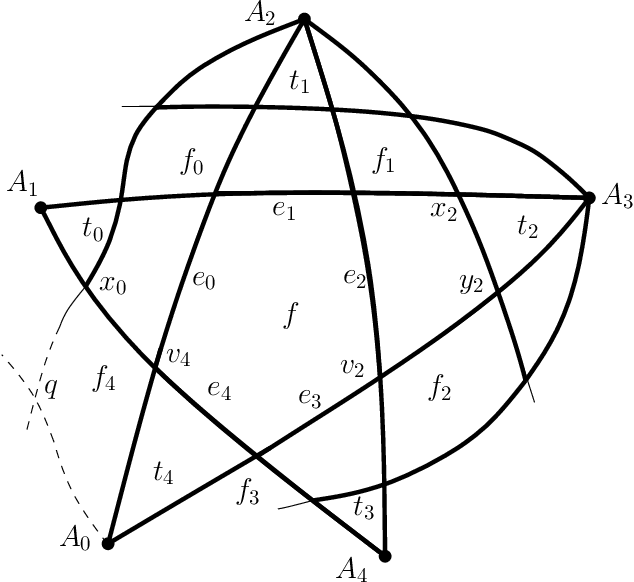}
	\caption{An illustration for Subcase 1.1 in the proof of Lemma~\ref{lem:ch_6-0-pentagon-1}. 
			Each of the wedges of $t_0$, $t_1$, $t_2$ and $t_3$ contains exactly one $0$-quadrilateral, $|f_0|=|f_1|=|f_2|=4$,
			and $e_4$ is an edge of $t_4$.}
	\label{fig:0-in-0-2-in-3q}
\end{figure}
In this case $f_3$ and $f_4$ are the neighbors of $t_4$ and at least one of them does not contribute charge to $t_4$ in Step~4 for otherwise $f$ does not contribute charge to $t_4$ in Step~5.
Assume without loss of generality that $f_4$ does not contribute charge to $t_4$ in Step~4.
Since $f_4$ is not a $1$-triangle (otherwise there would be two parallel edges between $A_0$ and $A_2$),
it follows that $f_4$ is a $1$-quadrilateral and that $ch_3(f_4) \leq 1/3$.
Let $x_0$, $v_4$, $A_0$ and $q$ be the vertices of $f_4$.
$f_4$ cannot contribute charge through $x_0v_4$ in Step~3 by Proposition~\ref{prop:1-triangle-neighbor}.
Similarly, $f_4$ does not contribute charge through $qx_0$ in Step~1 or~3 since $(A_1,A_3)$ would have more than four crossings in the first case
and the edge of $G$ that contains $qx_0$ would have more than four crossings in the second case. 
% Then $f_4$ does not contribute charge through $v_4A_0$. 
% It also contributes at most $1/6$ units of charge through its other edge that is incident to $A_0$ and through  of $x_0v_4$.
% Let $q$ be the vertex of $f_4$ that precedes $x_0$ and note that $f_4$ contributes at most $1/6$ units of charge
% through $qx_0$ and if it contributes that amount of charge then $q \notin \mathcal{P}'(f_4)$.
% Clearly, $x_0,A_0 \notin \mathcal{P}'(f_4)$.
% Therefore $f_4$ contributes at least 
% $\min \{\frac{|f_4|+1-4-1/3-3/6-(|f_2|-4)/3}{|f_2|-3}, \frac{|f_4|+1-4-1/3-2/6-(|f_4|-4)/3}{|f_4|-2} \} \geq 1/6$ 
% units of charge to $f$ in Step~6.
Therefore, $ch_3(f_4) > 1/3$.

It remains to consider the case that $e_4$ is not an edge of $t_4$ (see Figure~\ref{fig:0-in-0-2-in-3r}).
\begin{figure}[ht]
    \centering
    \subfigure[If both $f_3$ and $f_4$ are $0$-quadrilaterals, then each neighbor of $t_4$ is incident to two original vertices
				and therefore $t_4$ does not get charge from $f$.]{\label{fig:0-in-0-2-in-3r}
    {\includegraphics[width=4.5cm]{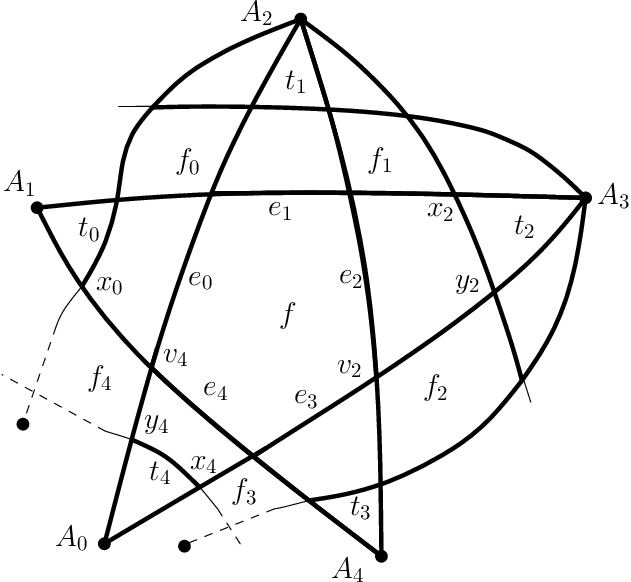}}}
	\hspace{2mm}
    \subfigure[If $f_4$ is a $1$-quadrilateral, then it sends $f$ at least $1/3$ units of charge in Step~6.]{\label{fig:0-in-0-2-in-3u}
    {\includegraphics[width=4.5cm]{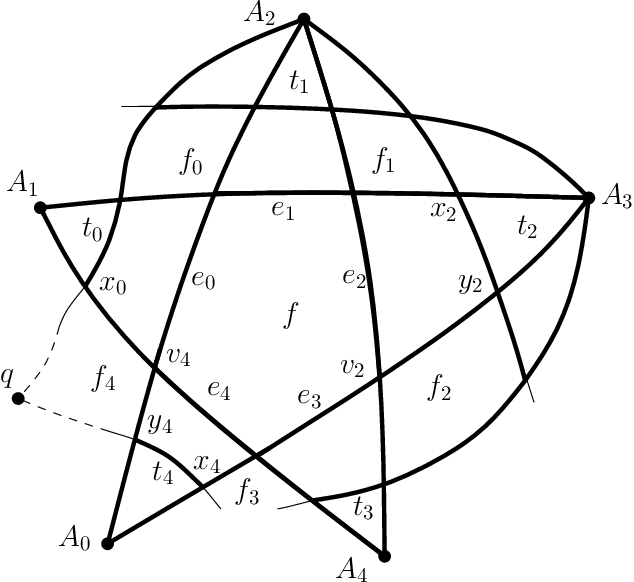}}}
	\hspace{2mm}
    \subfigure[If $|f_4| \geq 5$, then $f_4$ sends $f$ at least $1/6$ units of charge in Step~6.]{\label{fig:0-in-0-2-in-3w}
    {\includegraphics[width=4.5cm]{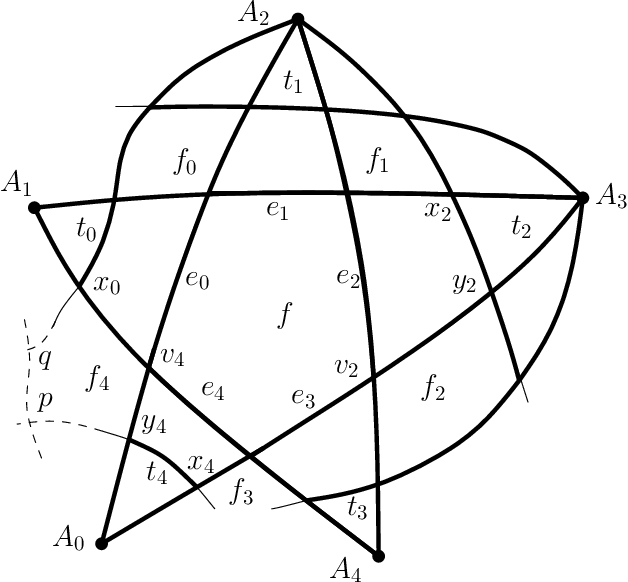}}}
	\caption{Illustrations for Subcase 1.1 in the proof of Lemma~\ref{lem:ch_6-0-pentagon-1}. 
			Each of the wedges of $t_0$, $t_1$, $t_2$, $t_3$ and $t_4$ contains exactly one $0$-quadrilateral and $|f_0|=|f_1|=|f_2|=4$.}
	\label{fig:0-in-1-2-in-3-1.1b}
\end{figure}
In this case the wedge of $t_4$ contains exactly one $0$-quadrilateral (two quadrilaterals would yield more than four crossings on $(A_0,A_2)$).
Note that the size of $f_3$ and $f_4$ is at least four, for otherwise the edge of $G$ that contains $x_4y_4$ would have more than four crossings.
Observe also that if $f_3$ (resp., $f_4$) is a $0$-quadrilateral, 
then the face that is incident to it and to $t_4$ has two original vertices on its boundary.
Therefore, it is impossible that both $f_3$ and $f_4$ are $0$-quadrilaterals, since then
$t_4$ would receive $1/6$ units of charge from each of its neighbors and no charge from $f$.
Assume without loss of generality that $f_4$ is not a $0$-quadrilateral and let $q$ be its vertex that precedes $x_0$.
Note that $f_4$ contributes at most $1/6$ units of charge through $x_0v_4$ and $v_4y_4$.
If $f_4$ is a $1$-quadrilateral (see Figure~\ref{fig:0-in-0-2-in-3u}), then it contributes no charge through $qx_0$ and $y_4q$ (since each of its neighbors at these edges is incident to two original vertices)
and it follows that $f_4$ sends at least $4+1-4-1/3-2/6=1/3$ units of charge to $f$ in Step~6.
Assume therefore that $|f_4| \geq 5$ and let $p$ be its vertex that follows $y_4$ (see Figure~\ref{fig:0-in-0-2-in-3w}).
Note that $f_4$ contributes at most $1/6$ units of charge through $qx_0$,
and if it does contribute charge through this edge, then $q \notin \mathcal{P}'(f_4)$.
Clearly, $x_0,y_4 \notin \mathcal{P}'(f_4)$.
Therefore if the size of $f_4$ is at least six, then it contributes at least
$\min \{\frac{|f_4|-4-3/6-(|f_4|-3)/3}{|f_4|-3}, \frac{|f_4|-4-2/6-(|f_4|-3)/3}{|f_4|-2} \} \geq 1/6$ 
units of charge to $f$ in Step~6.

Suppose that $|f_4|=5$.
If $f_4$ contributes charge through $y_4p$, then it must be to a $1$-triangle whose vertices are $y_4$, $p$ and $A_0$ and thus $p \notin \mathcal{P}'(f_4)$.
Similarly, if $f_4$ contributes charge through $qp$ or $qx_0$, then if must be to a $1$-triangle and then $q \notin \mathcal{P}'(f_4)$.
Therefore, if $f_4$ does not contribute $1/3$ units of charge through one of its edges, then it sends at least $1/6$ units of charge to $f$ in Step~6.

Recall that $f_4$ cannot contributes $1/3$ units of charge through $x_0v_4$, $v_4y_4$ and $qx_0$.
If $f_4$ contributes $1/3$ units of charge through $y_4p$, then it must be to a $1$-triangle whose vertices are $p$, $y_4$ and $A_0$.
This implies that $p$ is not an extreme crossing point on the edge of $G$ that contains $x_4y_4$.
It follows that $f_3$ cannot be a $0$-quadrilateral, and thus $f_4$ does not contribute charge through $v_4y_4$.
If $f_4$ contributes $1/3$ units of charge through $pq$, then it must be to a $1$-triangle in Step~3
since the edge of $G$ that contains $qx_0$ already contains four crossings.
This implies that the neighbor of $t_0$ that is also an immediate neighbor of $f_4$ at $qx_0$ is not a $1$-triangle or a $0$-quadrilateral,
and so $f_4$ does not contribute charge through $qx_0$.
In all of these cases we conclude that $f_4$ sends at least $1/6$ units of charge to $f$ in Step~6, and so $ch_6(f) \geq 0$.
This concludes Subcase~1.1.

\medskip\noindent\underline{Subcase 1.2:} $f$ contributes charge to $t_1$ and $t_3$ in Step~3
and to $t_2,t_4,t_0$ in Step~5.
Consider first the case that $e_1$ is an edge of $t_1$ and refer to Figure~\ref{fig:2-in-3_3-in-5_4}.
\begin{figure}[ht]
    \centering
    \subfigure[$e_1$ is an edge of $t_1$. $f_2$ contributes charge to $f$ in Step~6.]{\label{fig:2-in-3_3-in-5_4}
    {\includegraphics[width=5cm]{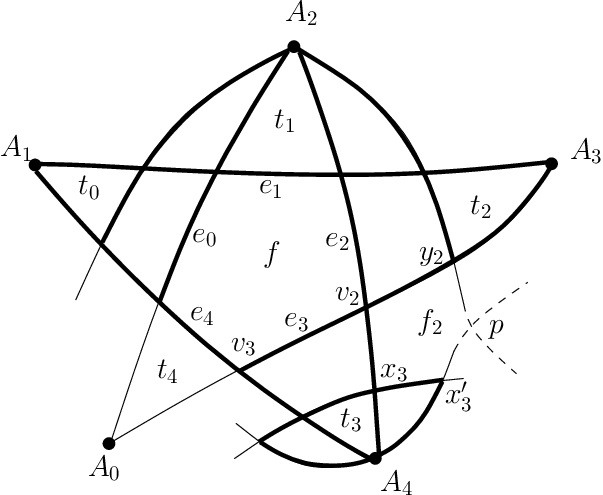}}}
	\hspace{2mm}
%    \subfigure[If $f_1$ and $f_2$ are both $1$-quadrilaterals then $f$ does not contribute charge through $e_2$.]{\label{fig:2-in-3_3-in-5_6}
%    {\includegraphics[width=4.5cm]{2-in-3_3-in-5_6}}}
%	\hspace{2mm}
%    \subfigure[If the wedge of $t_2$ contains two $0$-quadrilaterals,
%				then $f$ does not contribute charge to $t_0$ in Step~5.]{\label{fig:2-in-3_3-in-5_7}
%    {\includegraphics[width=4.5cm]{2-in-3_3-in-5_7}}}
    \subfigure[The wedges of all wedge-neighbors of $f$ contain one $0$-quadrilateral. $f_1$ or $f_2$ is not a $0$-quadrilateral and contributes charge to $f$ in Step~6.]{\label{fig:2-in-3_3-in-5_9}
	    {\includegraphics[width=5cm]{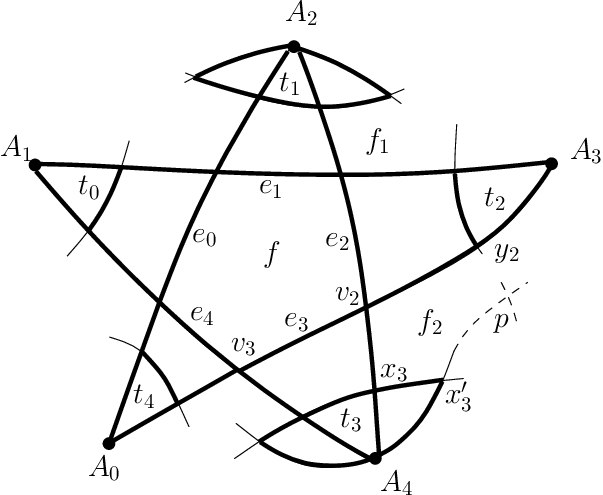}}}	\caption{Subcase 1.2 in the proof of Lemma~\ref{lem:ch_6-0-pentagon-1}:
				$f$ contributes $1/3$ units of charge in Step~3 through each of $e_1$ and $e_3$.}
	\label{fig:2-in-3_3-in-5_1.2}
\end{figure}
Since the neighbors of $t_1$ should be $1$-triangles, it follows that
the wedges of $t_2$ and $t_0$ each contain exactly one $0$-quadrilateral.
It is impossible that $e_3$ is an edge of $t_3$,
as this would imply two parallel edges between $A_2$ and $A_4$ (note that $B_2$ and $A_4$ coincide).
If the wedge of $t_3$ contains two $0$-quadrilaterals, then there are more than four crossings on $(A_1,A_4)$.
Therefore, the wedge of $t_3$ contains exactly one $0$-quadrilateral.

Consider the face $f_2$ and refer to Figure~\ref{fig:2-in-3_3-in-5_4}.
Observe that $|f_2|>4$, for if $|f_2|=4$ then $G$ has two parallel edges between $A_2$ and $A_4$.
Let $p$ be the vertex of $f_2$ that follows $y_2$.
It follows from Proposition~\ref{prop:1-triangle-neighbor} that $f_2$
contributes at most $1/6$ units of charge through each of $v_2y_2$, $x_3v_2$, $x'_3x_3$
and its other edge that is incident to $x'_3$.
Note also that $f_2$ is not a vertex-neighbor of a $0$-pentagon at $x'_3$, $x_3$ and $y_2$.
Therefore, if $|f_2| \geq 6$, then $f_2$ contributes at least $\frac{|f_2|-4-4/6-(|f_2|-4)/3}{|f_2|-3} \geq 1/6$ units of charge to $f$ in Step~6.

Suppose that $|f_2|=5$.
If $p$ is a vertex of $G$, then $\mathcal{P}'(f_2)=\{v_2\}$ and $f_2$ contributes at most $1/6$ units of charge through each of its edges that are incident to $p$.
Thus, $f_2$ contributes at least $|f_2|+1-4-1/3-5/6 \geq 1/6$ units of charge to $f$ in Step~6.
Assume therefore that $p$ is a crossing point.
Note that if $f_2$ contributes charge through one of the edges that are incident to $p$, then it must be to a $1$-triangle and it follows that $p \notin \mathcal{P}'(f_2)$.
Furthermore, if $f_2$ contributes $1/3$ units of charge through $y_2p$, then
it must be contributed to a $1$-triangle whose vertices are $y_2$, $p$ and $A_3$, and whose neighbors are $1$-triangles.
Therefore, in such a case $f_2$ does not contribute charge through $px'_3$,
because the face that shares this edge with $f_2$ has at least three crossing points and one original vertex of $G$ as vertices. 
It follows that in Step~6 $f_2$ contributes to $f$ at least $1/6$ units of charge and $f$ ends up with a non-negative charge.

% If $|f_2| \geq 6$ then on the clockwise chain from $y_2$ to $q$ there are $|f_2|-3$ edges
% and at most $|f_2|-4$ vertices through which $f_2$ contributes charge in Step~6.
% Since $f_2$ contributes at most $1/6$ units of charge through $v_2y_2$, $x_3q$, and $x_3v_2$,
% it contributes at least $\frac{|f_2|-4-3/6-(|f_2|-3)/3}{|f_2|-3} \geq 1/6$
% units of charge to every face in $\mathcal{P}(f_2)$ in Step~6.
% If $|f_2|>5$ then every `extra' edge might add a $0$-pentagon to $\mathcal{P}(f_2)$ and $f_2$ might contribute at most $1/3$ through it,
% but it also increases the initial charge of $f_2$ by one, and therefore, still every $0$-pentagon in $\mathcal{P}(f_2)$
% (including $f$) will receive at least $1/6$ units of charge from $f_2$ in Step~6.

The case that $e_3$ is an edge of $t_3$ is symmetric, therefore
we assume now that $e_1$ is not an edge of $t_1$ and $e_3$ is not an edge of $t_3$. %(refer to Figure~\ref{fig:2-in-3_3-in-5_6}).
Observe that the wedges of $t_1$ and $t_3$ must contain exactly one $0$-quadrilateral each, for otherwise $(A_2,A_4)$ has more than four crossings.
If one of the wedges of $t_0$, $t_2$ and $t_4$ contains no $0$-quadrilaterals,
then it follows from Proposition~\ref{prop:1/6} that $f$ gets at least $1/6$ units
of charge from one of its vertex-neighbors at Step~6 and ends up with a non-negative charge.

Assume therefore that each of the wedges of $t_0$, $t_2$ and $t_4$ contains exactly one $0$-quadrilateral (if one of them contains two $0$-quadrilaterals, then another one contains no $0$-quadrilaterals or there is an edge with more than four crossings).
Refer to Figure~\ref{fig:2-in-3_3-in-5_9} and note that each of $f_1$ and $f_2$ has at least four crossing points as vertices 
(e.g., $f_2$ is incident to $x'_3$, $x_3$, $v_2$ and $y_2$).
At least one of these faces is not a $0$-quadrilateral, for otherwise there would be two parallel edges between $A_2$ and $A_4$.
Assume without loss of generality that $|f_2| \geq 5$ and observe that similarly
to the case above in which $e_1$ was an edge of $t_1$, it follows that $f_2$ contributes at least $1/6$ units of charge to $f$ in Step~6.

\bigskip
\noindent\underline{Case 2:} $ch_3(f) = 0$ and $ch_5(f) < 0$.
That is, $f$ contributes $1/3$ units of charge to exactly three $1$-triangles in Step~3,
and contributes $1/6$ units of charge to one or two $1$-triangles in Step~5.
We may assume without loss of generality that $f$ contributes charge through $e_1$ and $e_2$ in Step~3,
and consider two subcases.

\medskip
\noindent\underline{Subcase 2.1:} $f$ contributes charge to $t_3$ in Step~3.
Observe that none of the $1$-triangles $t_1,t_2,t_3$ can share an edge (of $M(G)$) with $f$
according to Proposition~\ref{prop:e_1-e_2}.
Moreover, the wedges of $t_1$ and $t_3$ must contain exactly one $0$-quadrilateral,
for otherwise $(A_2,A_4)$ has more than four crossings.
If there is exactly one $0$-quadrilateral in the wedge of $t_2$,
then by Corollary~\ref{cor:f_1-or-f_2} $f_1$ or $f_2$ contributes 
at least $1/3$ units of charge to $f$ in Step~6 and so $ch_6(f) \geq 0$.

Therefore, assume that there are two $0$-quadrilaterals in the wedge of $t_2$ and
refer to Figure~\ref{fig:3-in-3-2.1}.
\begin{figure}[ht]
    \centering
    \subfigure[Subcase 2.1: $f$ contributes charge to $t_1$, $t_2$ and $t_3$ in Step~3.
				If there are two $0$-quadrilaterals in the wedge of $t_2$, then each of  $f_0$ and $f_3$ contribute at least $1/6$ units of charge to $f$ in Step~6.]{\label{fig:3-in-3-2.1}
    {\includegraphics[width=5cm]{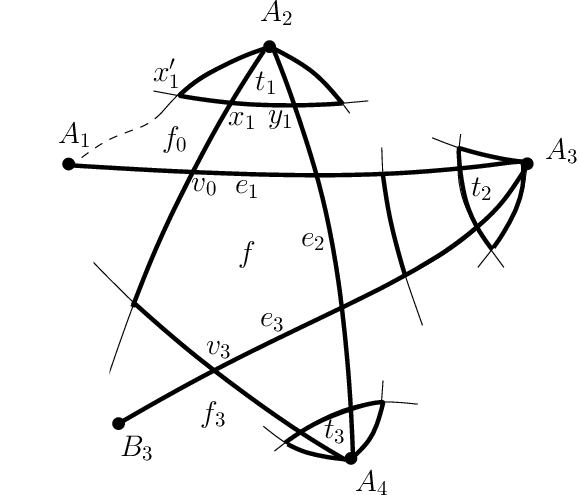}}}
	\hspace{2mm}
    \subfigure[Subcase 2.2: $f$ contributes charge to $t_1,t_2,t_4$ in Step~3.
				If there are two $0$-quadrilaterals in the wedge of $t_2$, then $f_0$ 
				contributes at least $1/6$ units of charge to $f$ in Step~6.]{\label{fig:3-in-3-2.2}
    {\includegraphics[width=5cm]{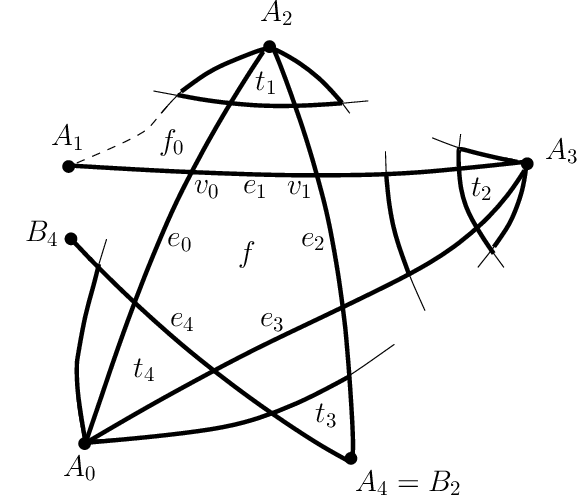}}}
	\caption{Illustrations for Case 2 in the proof of Lemma~\ref{lem:ch_6-0-pentagon-1}.}
	\label{fig:3-in-3-2}
\end{figure}
It follows that $A_1v_0$ is an edge of $f_0$ and $B_3v_3$ is an edge of $f_3$.
Consider $f_0$ and observe that $|f_0| \geq 4$.
Note also that $f_0$ does not contribute any charge through $x_1v_0$, as this would imply 
that the edge of $G$ that contains $x_1y_1$ has more than four crossings.
Note that $f_0$ contributes at most $1/6$ units of charge through each of its edges that are incident to $A_1$, and, by Proposition~\ref{prop:1-triangle-neighbor}, also through $x'_1x_1$.
Since $A_1,x'_1,x_1 \notin \mathcal{P}'(f_0)$, $f_0$ contributes at least
$\frac{|f_0|+1-4-1/3-3/6-(|f_0|-4)/3}{|f_0|-3} \geq 1/6$ units of charge to $f$ in Step~6.
By symmetry, so does $f_3$, and therefore $f$ ends up with a non-negative charge.
%
%If $f_0$ is a $1$-quadrilateral (as in Figure~\ref{fig:3-in-3-2.1}), then it
%does not contribute charge through $A_1x'_1$ since its neighbor at this edge is incident to two original vertices. Therefore $ch_5(f_0) \geq 1/3$ in this case,
%and since $A_1,x'_1,x_1 \notin \mathcal{P}'(f_0)$, $f_0$ sends $1/3$ units of charge to $f$ is Step~6.
%
%If $|f_0| \geq 5$ then consider the clockwise chain from $A_1$ to $z$, 
%and observe that it contains $|f_0|-3$ edges
%and at most $|f_0|-4$ vertices through which $f_0$ sends charge in Step~6.
%Therefore, every face in $\mathcal{P}(f_0)$ (including $f$) receives from $f_0$ in Step~6 at least
%$\frac{|f_0|-4+1-2/6-1/3-(|f_0|-3)/3}{|f_0|-3} \geq 1/3$ units of charge.
%
% If $|f_0| > 4$, then every `extra' edge might add a $0$-pentagon to $\mathcal{P}(f_0)$ and $f_0$ might contribute at most $1/3$ through it,
% but it also increases the initial charge of $f_0$ by one, and therefore, still every $0$-pentagon in $\mathcal{P}(f_0)$
% (including $f$) will receive at least $1/3$ units of charge from $f_0$ in Step~6.
%By symmetry $f_3$ also contributes at least $1/3$ units of charge to $f$ in Step~6
%and therefore $f$ ends up with a non-negative charge.
%

\smallskip\noindent\underline{Subcase 2.2:} 
$f$ contributes charge through each of $e_1$, $e_2$ and $e_4$ in Step~3,
and through at least one of $e_0$ and $e_3$ in Step~5.
Observe that none of the $1$-triangles $t_1,t_2$ can share an edge (of $M(G)$) with $f$
according to Proposition~\ref{prop:e_1-e_2}.
We claim that we may assume that each of the wedges of $t_1$ and $t_2$ contains exactly one $0$-quadrilateral.
Indeed, assume without loss of generality that the wedge of $t_2$ contains two $0$-quadrilaterals
and refer to Figure~\ref{fig:3-in-3-2.2}.
Since each of $(A_0,A_3)$ and $(A_1,A_3)$ contains at most four crossings, $e_4$ must be an edge of $t_4$ and $v_0A_1$ must be an edge of $f_0$.
It follows that $f$ can not contribute charge through $e_0$ (since its immediate neighbor at this edge is incident to three crossing points and to $A_1$) and thus it must contribute charge through $e_3$.
Hence, $A_4=B_2$ and the wedge of $t_1$ must contain exactly one $0$-quadrilateral.
It follows, as in the analysis in Subcase~2.1 (see Figure~\ref{fig:3-in-3-2.1}), that $f_0$ contributes at least $1/6$ units of charge to $f$ in Step~6, and thus $ch_6(f) \geq 0$. 

Therefore, each of the wedges of $t_1$ and $t_2$ contains exactly one $0$-quadrilateral.
If $f_1$ is not a $0$-quadrilateral, then it follows from Proposition~\ref{prop:e_1-e_2-step3}
that $f_1$ sends at least $1/3$ units of charge to $f$ in Step~6 and so $ch_6(f) \geq 0$.
Assume therefore that $f_1$ is a $0$-quadrilateral.

Consider the case that the wedge of $t_4$ contains no $0$-quadrilaterals and refer to Figure~\ref{fig:0-in-0-3-in-3a}.
Suppose that $f$ contributes charge through $e_3$ in Step~5.
We claim that in this case $f_2$ contributes at least $1/6$ units of charge to $f$ in Step~6.
Observe that $|f_2| \geq 5$, since otherwise if $|f_2|=4$ there are two parallel edges between $A_0$ and $A_3$.
Let $q$ be the vertex of $f_2$ that follows $y'_2$.
By Proposition~\ref{prop:1-triangle-neighbor} $f_2$ contributes at most $1/6$ units of charge through each of $x_3v_2$, $v_2y_2$, $y_2y'_2$, and $y'_2q$.
Note also that $x_3,y_2,y'_2 \notin \mathcal{P}'(f_2)$.
Therefore, if $|f_2| \geq 6$, then $f_2$ contributes at least $\frac{|f_2|-4-4/6-(|f_2|-4)/3}{|f_2|-3} \geq 1/6$ units of charge to $f$ in Step~6.

If $|f_2|=5$, then observe that if $f_2$ contributes $1/3$ units of charge through $qx_3$ then it implies that $f_2$ does not contribute any charge through $y'_2q$.
Furthermore, if $f_2$ contributes charge through any of these two edges, then $q \notin \mathcal{P}'(f_2)$.
It follows that if $|f_2|=5$, then $f_2$ also contributes at least $1/6$ units of charge to $f$ in Step~6.

By symmetry, if $f$ contributes $1/6$ units of charge through $e_0$ in Step~5, then it receives at least $1/6$ units of charge from $f_0$ in Step~6,
and so $f$ ends up with a non-negative charge.

It remains to consider the case that the wedge of $t_4$ contains exactly one $0$-quadrilateral
(note that it cannot contain two quadrilaterals, for otherwise $(A_0,A_3)$ would contain five crossings).
Suppose that $f$ contributes charge through $e_3$ in Step~5.
Then the wedge of $t_3$ contains at most one $0$-quadrilateral, for otherwise $(A_2,A_4)$ would contain five crossings.
If the wedge of $t_3$ contains no $0$-quadrilateral, then by Proposition~\ref{prop:1/6} $f_3$ sends at least $1/6$ units of charge to $f$ in Step~6.

%We claim that the wedge of $t_3$ must contain one $0$-quadrilateral.
%Indeed, suppose it does not, that is, $e_3$ is an edge of $t_3$ and refer to %Figure~\ref{fig:0-in-0-3-in-3b}.
\begin{figure}[ht]
    \centering
    \subfigure[$e_4$ is an edge of $t_4$, and $f$ contributes charge through $e_3$]{\label{fig:0-in-0-3-in-3a}
    {\includegraphics[width=4.5cm]{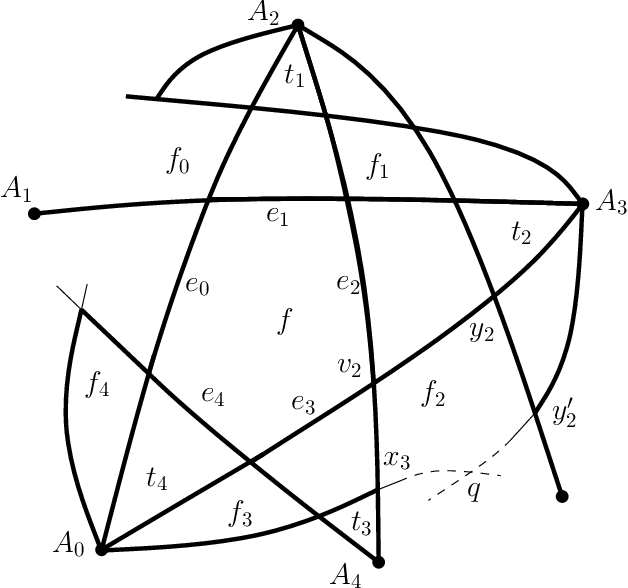}}}
	\hspace{2mm}
%    \subfigure[$e_4$ is not an edge of $t_4$ and $e_3$ is not an edge of $t_3$.]{\label{fig:0-in-0-3-in-3b}
%    {\includegraphics[width=4.5cm]{0-in-0-3-in-3b}}}
%	\hspace{2mm}
    \subfigure[Each of the wedges of $t_1$, $t_2$, $t_3$ and $t_4$ contains one $0$-quadrilateral. 
			If $|f_3| \geq 4$ then $f_3$ contributes at least $1/6$ units of charge to $f$.]{\label{fig:0-in-0-3-in-3c}
    {\includegraphics[width=4.5cm]{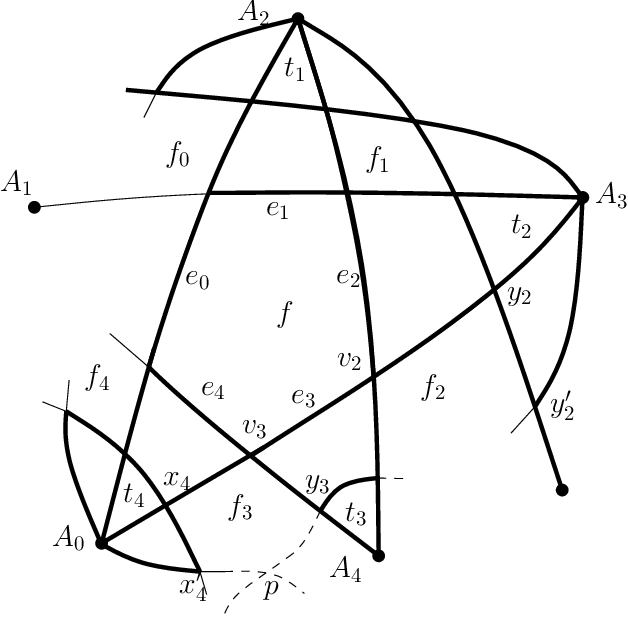}}}
	\caption{Illustrations for Subcase 2.2 in the proof of Lemma~\ref{lem:ch_6-0-pentagon-1}.
				$f$ contributes charge through $e_1$, $e_2$ and $e_4$ in Step~3;
				the wedges of $t_1$ and $t_2$ each contains one $0$-quadrilateral;
				and $f_1$ is a $0$-quadrilateral.}
	\label{fig:3-in-3-3}
\end{figure}
%Note that the neighbors of $t_3$ are $f_2$ and $f_3$, and that the size of each of them is at least four.
%$f_2$ does not contribute charge through any of its edges that are not incident to $A_4$ is Step~1 or~3,
%and therefore it follows that $f_2$ contributes charge to $t_4$ in Step~4.
%Considering $f_3$, observe that it does not contribute charge in Step~1 or~3 through the edge
%that it shares with a neighbor of $t_4$, since this neighbor is a $1$-triangle that has a neighbor 
%which is incident to two original vertices.
%$f_3$ does not contribute charge in Step~1 or~3 through $x_4v_3$ either, since if it does,
%then it must be to a $1$-triangle $t'$ in Step~3 and not to a $0$-triangle in Step~1 (otherwise there would be five crossings on the edge of $G$ that contains $x_4y_4$).
%However, this $1$-triangle has a neighbor that is incident to two original vertices.
%Thus, $f_3$ also contributes charge to $t_3$ in Step~4 and therefore $f$ does not contribute charge to $t_3$ in Step~4.

Therefore, assume that the wedge of $t_3$ contains exactly one $0$-quadrilateral and refer to Figure~\ref{fig:0-in-0-3-in-3c}.
Note that the size of $f_2$ and $f_3$ is at least four, and it is impossible that both of them are $0$-quadrilaterals,
for then there would be two parallel edges between $A_0$ and $A_3$.
Suppose that $f_3$ is not a $0$-quadrilateral, and let $p$ be its vertex that precedes $x'_4$.
It follows from Proposition~\ref{prop:1-triangle-neighbor} that $f_3$ contributes at most $1/6$ units of charge through each of $px'_4$, $x'_4x_4$ and $x_4v_3$.
Note also that $f_3$ may contribute at most $1/6$ units of charge through $v_3y_3$
(to the $1$-triangle whose vertices are $A_3$, $y_2$ and $y'_2$),
and that $x'_4,x_4,y_3 \notin \mathcal{P}'(f_3)$.
Therefore, if $|f_3| \geq 6$, then $f_3$ contributes at least $\frac{|f_3|-4-4/6-(|f_3|-4)/3}{|f_3|-3} \geq 1/6$ units of charge to $f$ in Step~6.
If $|f_3|=5$, then observe that $f_3$ contributes at most $1/6$ units of charge through $px'_4$ and that in such a case it cannot contribute $1/3$ units of charge through $y_3p$.
Furthermore, if $f_3$ contributes charge through any of these two edges, then $p \notin \mathcal{P}'(f_2)$.
It follows that if $|f_3|=5$, then $f_3$ contributes at least $1/6$ units of charge to $f$ in Step~6.

If $f_3$ is a $0$-quadrilateral then $f_2$ is not a $0$-quadrilateral.
In this case, as in Subcase~1.2 (see Figure~\ref{fig:2-in-3_3-in-5_4}),
we conclude that $f_2$ sends at least $1/6$ units of charge to $f$ in Step~5.

By symmetry, if $f$ contributes $1/6$ units of charge through $e_0$ in Step~5, then it receives at least $1/6$ units of charge from $f_0$ or $f_4$ in Step~6,
and so $f$ ends up with a non-negative charge.
This concludes Case~2.

\bigskip
\noindent\underline{Case 3:} $ch_3(f) = -1/3$.
That is, $f$ contributes $1/3$ units of charge to exactly four $1$-triangles in Step~3,
and contributes $1/6$ units of charge to zero or one $1$-triangles in Step~5.
We may assume without loss of generality that $f$ contributes charge through $e_0$, $e_1$, $e_2$ and $e_3$ in Step~3.

It follows from Proposition~\ref{prop:e_1-e_2} and the maximum number of crossings per edge,
that each of the wedges of $t_0$, $t_1$, $t_2$ and $t_3$ contains exactly one $0$-quadrilateral.
Therefore, by Corollary~\ref{cor:f_1-or-f_2} $f_0$ or $f_1$ contributes 
at least $1/3$ units of charge to $f$ in Step~6 and so does $f_1$ or $f_2$.
Suppose that $ch_6(f)<0$. 
Then it follows that $f$ receives charge from $f_1$ and no charge from $f_0$ and $f_2$ and thus these faces must be $0$-quadrilaterals (following Proposition~\ref{prop:e_1-e_2-step3}).
Furthermore, $f$ must contribute charge through $e_4$ in Step~5.

If the wedge of $t_4$ contains no $0$-quadrilaterals, then it follows from Proposition~\ref{prop:1/6} that $f$ receives at least $1/6$ units of charge from each of $f_3$ and $f_4$ and thus $ch_6(f) \geq 0$.

%Consider first the case that $e_4$ is an edge of $t_4$ and refer to %Figure~\ref{fig:0-in-0-4-in-3a}.
%\begin{figure}[ht]
%    \centering
%    \includegraphics[width=5cm]{0-in-0-4-in-3a}
%	\caption{An illustration for Case 3 in the proof of Lemma~\ref{lem:ch_6-0-pentagon-1}:
%			$f$ contributes charge through $e_0$, $e_1$, $e_2$ and $e_3$ in Step~3,
%			$f_0$ and $f_2$ are $0$-quadrilaterals, and $e_4$ is an edges of $t_4$.}
%	\label{fig:0-in-0-4-in-3a}
%\end{figure}
%It follows that $f_3$ and $f_4$ are the neighbors of $t_4$.
%Since $f$ receives charge from $f$ in Step~5, at least one of the neighbors of $t_4$ does not contribute charge to $t_4$ in Step~4.
%Assume without loss of generality that $f_4$ does not contribute charge to $t_4$ in Step~4.
%Note that $f_4$ cannot be a $1$-triangle, for otherwise there would be two parallel edges between $A_0$ and $A_2$.
%It follows that $f_4$ must be a $1$-quadrilateral and $ch_3(f_4)=1/3$.
%Denote by $q$ its vertex that precedes $x_0$ and let $(A',A_2)$ be the edge of $G$ that contains $qx_0$.
%Note that $f_4$ does not contribute charge through $x_0v_4$ in Steps~1 or~3,
%since it is the wedge-neighbor of a $1$-triangle that has one neighbor which is incident to two original vertices ($A_1$ and $A_2$).
%Similarly, $f_4$ does not contribute charge through $qx_0$ in Steps~1 or~3,
%since the face that shares this edge with $qx_0$ is incident to $A_1$,
%and if this face is a $1$-triangle, then one of its neighbors is incident to both $A_1$ and $A'$.
%Therefore $ch_3(f_4)=2/3$, a contradiction.

If the wedge of $t_4$ contains two $0$-quadrilaterals, 
then $(A_0,A_3)$ would contain five crossings.
Assume therefore that the wedge of $t_4$ contains exactly one $0$-quadrilateral (see Figure~\ref{fig:0-in-0-4-in-3c}).
\begin{figure}[ht]
    \centering
    \includegraphics[width=5cm]{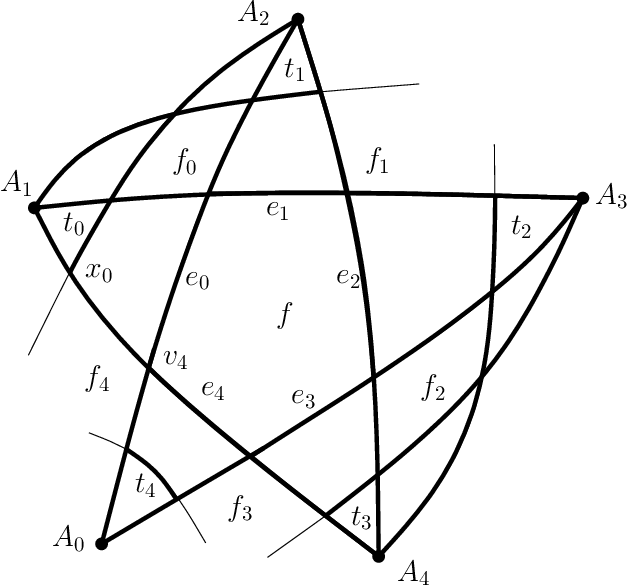}
	\caption{An illustration for Case 3 in the proof of Lemma~\ref{lem:ch_6-0-pentagon-1}:
			$f$ contributes charge through $e_0$, $e_1$, $e_2$ and $e_3$ in Step~3,
			$f_0$ and $f_2$ are $0$-quadrilaterals, and the wedge of $t_4$ contains one $0$-quadrilateral.}
	\label{fig:0-in-0-4-in-3c}
\end{figure}
In this case, similarly to Subcase~1.1 (see Figure~\ref{fig:0-in-1-2-in-3-1.1b}),
it is impossible that both $f_3$ and $f_4$ are $0$-quadrilaterals,
and we conclude that one of these faces contributes at least $1/6$ units of charge to $f$ in Step~6, and thus $ch_6(f) \geq 0$.

\bigskip
\noindent\underline{Case 4:} $ch_3(f) = -2/3$.
That is, $f$ contributes $1/3$ units of charge through each of its edges in Step~3.
It follows from Proposition~\ref{prop:e_1-e_2} and the maximum number of crossings per edge,
that each of the wedges of $t_0$, $t_1$, $t_2$, $t_3$ and $t_4$ contains exactly one $0$-quadrilateral.
Therefore, by Corollary~\ref{cor:f_1-or-f_2} $f_0$ or $f_1$ contributes 
at least $1/3$ units of charge to $f$ in Step~6 and so does $f_2$ or $f_3$.
Therefore $f$ ends up with a non-negative charge.

This concludes the proof of Lemma~\ref{lem:ch_6-0-pentagon-1}.
\end{proof}

It follows from Proposition~\ref{prop:ch_5} and Lemmas~\ref{lem:ch_6-0-pentagon-negative}, \ref{lem:ch_6-0-pentagon-0}, 
\ref{lem:ch_6-0-pentagon-1/3}, \ref{lem:ch_6-0-pentagon-2/3}, and~\ref{lem:ch_6-0-pentagon-1}
that the final charge of every face in $M(G)$ is non-negative.
Recall that the charge of every original vertex of $G$ is $1/3$,
and that the total charge is $4n-8$.
Therefore, $2|E(G)|/3=\sum_{A \in V(G)} \deg(v)/3 \leq 4n-8$ and thus $|E(G)| \leq 6n-12$.

\paragraph{A lower bound.}
To see that the bound in Theorem~\ref{thm:4crossings} is tight up to an additive constant
we use the following construction of Pach et al.~\cite[Proposition~2.8]{PR+06}.
For an integer $l \geq 2$, we set $n=6l$ and tile a vertical cylindrical surface with $l-1$ horizontal layers
each consisting of three hexagonal faces that are wrapped around the cylinder.
The top and bottom of the cylinder are also tiled with hexagonal faces.
See Figure~\ref{fig:cylinder} for an illustration of this construction.
\begin{figure}[ht]
    \centering
    \subfigure[Tiling a vertical cylinder surface with horizontal layers each consisting of three hexagons.
				The top and bottom are also tiled with hexagons.]{\label{fig:cylinder}
    {\includegraphics[width=4.5cm]{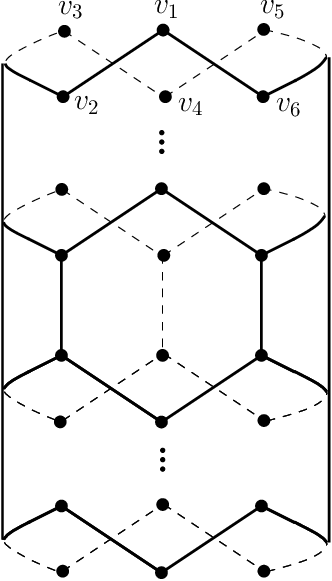}}}
    \hspace{20mm}
    \subfigure[Drawing edges in the top face to get an almost tight lower bound for Theorem~\ref{thm:4crossings}.]{\label{fig:top-face}
    {\includegraphics[width=4cm]{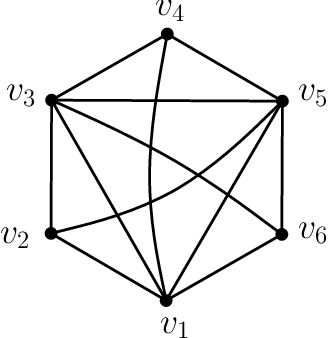}}}
	\caption{A lower bound construction.}
	\label{fig:lower-bound}
\end{figure}
Note that every vertex is adjacent to exactly three hexagons,
except for three vertices of the top face ($v_1,v_3,v_5$ in Figure~\ref{fig:cylinder}) and three vertices of the bottom face
that are adjacent to two hexagons.
On the top and bottom hexagons we draw edges between every two vertices
unless both of them are adjacent to exactly three hexagons (see Figure~\ref{fig:top-face}).
For the rest of the hexagons we draw all the possible edges.
Thus, every edge in the drawing is indeed crossed at most four times.
Note that the degree of every vertex is $12$, except for six vertices whose degree is $8$, and six vertices whose degree is $10$.
Hence, the number of edges is $(12(n-12)+10\cdot6+8\cdot6)/2=6n-18$.

\paragraph{Remark.}
As noted in the Introduction, Theorem~\ref{thm:4crossings} actually holds in a more general setting
where edges may intersect several times (including the case of two common endpoints),
as long as there are no \emph{balanced lenses}.
A \emph{lens} is formed by two (closed) edge-segments in a topological (multi)graph that intersect at their endpoints and at no other points:
these edge-segments define a closed curve that divides the plane into two regions (bounded and unbounded),
each of which is called a \emph{lens}.
If every edge of the topological (multi)graph crosses one of the edge-segments 
the same number of times it crosses the other edge-segment, then the lenses defined by these edge-segments are \emph{balanced}.
The proof of Theorem~\ref{thm:4crossings} can be modified to show
that a topological (multi)graph with $n \geq 3$ vertices, at most four crossings per edge
and no balanced lenses has at most $6n-12$ vertices.\footnote{
Note that the construction above with the ``missing diagonals'' at the top and bottom hexagons
shows that this bound is tight for infinitely many values of $n$.}
% However, since this stronger statement is not needed to improve the Crossing Lemma,
% and because proving it would complicate an already long and technical proof, 
% we chose not to prove this generalization of Theorem~\ref{thm:4crossings} here.

%%%%%%%%%%%%%%%%%%%%%%%%%%%%%%%%%%%%%%%%%%%%%%%%%%%%%%%%%%%%%%%%%%%%%%%%%%%%%
%\section{Applications of Theorem~\ref{thm:4crossings}}
\section{Improvements for the Crossing Lemma and Albertson Conjecture}
\label{sec:applications}
%%%%%%%%%%%%%%%%%%%%%%%%%%%%%%%%%%%%%%%%%%%%%%%%%%%%%%%%%%%%%%%%%%%%%%%%%%%%%

%\subsection{A better Crossing Lemma}
%\label{sec:crossing-lemma}

Let $G$ be a graph with $n>2$ vertices and $m$ edges.
Recall that in a drawing of $G$ with $\X(G)$ crossings $G$ is drawn as a simple topological graph.
Therefore, we may consider only such drawings.

The following linear bounds on the crossing number $\X(G)$ appear in~\cite{PR+06} and~\cite{PT97}.
\begin{eqnarray}
\X(G) & \geq & m - 3(n-2) 	 					\label{eq:m-3(n-2)} \\
\X(G) & \geq & \frac{7}{3}m - \frac{25}{3}(n-2)	\label{eq:7/3m-25/3(n-2)} \\
\X(G) & \geq & 4m - \frac{103}{6}(n-2) 			\label{eq:4m-103/6(n-2)} \\
\X(G) & \geq & 5m - 25(n-2) 					\label{eq:5m-25(n-2)}
\end{eqnarray}

Using Theorem~\ref{thm:4crossings} we can obtain a similar bound, as stated in Theorem~\ref{thm:linear-new}:
\begin{equation}\label{eq:5m-139/6(n-2)}
\X(G)  \geq  5m - \frac{139}{6}(n-2)
\end{equation}

\begin{proofof}{Theorem~\ref{thm:linear-new}}
%If $n=3$ or $n=4$ the statement trivially holds since $\X(G) \geq 0$.
If $m \leq 6(n-2)$, then the statement holds by~(\ref{eq:4m-103/6(n-2)}).
Suppose now that $m > 6(n-2)$ and consider a drawing of $G$ with $\X(G)$ crossings.
Remove an edge of $G$ with the most crossings,
and continue doing so as long as the number of remaining edges is greater than $6(n-2)$.
It follows from Theorem~\ref{thm:4crossings} that each of the $m-6(n-2)$ removed edges
was crossed by at least $5$ other edges at the moment of its removal.
By~(\ref{eq:4m-103/6(n-2)}), the number of crossings in the remaining graph is
at least $4(6(n-2))-\frac{103}{6}(n-2)$.
Therefore, $\X(G) \geq 5(m-6(n-2)) + 4(6(n-2))-\frac{103}{6}(n-2) = 5m-\frac{139}{6}(n-2)$.
\end{proofof}

Using the new linear bound it is now possible to obtain a better Crossing Lemma,
by plugging it into its probabilistic proof, as in~\cite{Mon05,PR+06,PT97}.

\begin{proofof}{Theorem~\ref{thm:crossing-lemma}}
Let $G$ be a graph with $n$ vertices and $m \geq 6.95n$ edges
and consider a drawing of $G$ with $\X(G)$ crossings.
Construct a random subgraph of $G$ by selecting every vertex independently with probability $p=6.95n/m\leq 1$.
Let $G'$ be the subgraph of $G$ that is induced by the selected vertices.
Denote by $n'$ and $m'$ the number of vertices and edges in $G'$, respectively.
Clearly, $\E[n']=pn$ and $\E[m']=p^2m$.
Denote by $x'$ the number of crossings in the drawing of $G'$ inherited from the drawing of $G$.
Then $\E[\X(G')] \leq \E[x'] = p^4\X(G)$.
It follows from Theorem~\ref{thm:linear-new} that
$\X(G') \geq 5m'-\frac{139}{6}n'$ (note that this it true for any $n' \geq 0$), 
and this holds also for the expected values:
$\E[\X(G')] \geq 5\E[m']-\frac{139}{6}\E[n']$.
Plugging in the expected values we get that
$\X(G) \geq \left(\frac{5}{6.95^2}-\frac{139}{6\cdot 6.95^3}\right)\frac{m^3}{n^2} = 
\frac{2000}{57963} \frac{m^3}{n^2} \geq \frac{1}{29}\frac{m^3}{n^2}$.

Consider now the case that $m < 6.95n$.
Comparing the bounds~(\ref{eq:m-3(n-2)})--(\ref{eq:5m-139/6(n-2)})
one can easily see that (\ref{eq:m-3(n-2)}) is best when $3(n-2) \leq m < 4(n-2)$,
(\ref{eq:7/3m-25/3(n-2)}) is best when $4(n-2) \leq m < 5.3(n-2)$,
(\ref{eq:4m-103/6(n-2)}) is best when $5.3(n-2) \leq m < 6(n-2)$,
and (\ref{eq:5m-139/6(n-2)}) is best when $6(n-2) \leq m$.
If we consider the possible values of $m < 6.95n$ according to these intervals
and use the best bound for each interval, then we get that $\X(G) \geq \frac{1}{29} \frac{m^3}{n^2}-\frac{35}{29}n$. 
\end{proofof}

The new bound for the Crossing Lemma immediately implies better bounds in all of its applications.
We recall three such improvements from~\cite{PR+06} and~\cite{PT97}.
Since the computations are almost verbatim to the proofs in~\cite{PR+06}, we omit them.

\begin{cor}\label{cor:multigraph}
Let $G$ be an $n$-vertex multigraph with $m$ edges and edge multiplicity $t$.
Then $\X(G) \geq \frac{1}{29}\frac{m^3}{tn^2}-\frac{35}{29}nt^2$.
\end{cor}

\begin{cor}\label{thm:k-crossings}
Let $G$ be an $n$-vertex simple topological graph.
If every edge of $G$ is crossed by at most $k$ other edges, for some $k \geq 2$, 
then $G$ has at most $3.81\sqrt{k}n$ edges.
\end{cor}

\begin{cor}\label{cor:incidences}
The number of incidences between $m$ lines and $n$ points in the Euclidean plane
is at most $2.44m^{2/3}n^{2/3}+m+n$.
\end{cor}

The previous best constant in the last upper bound was $2.5$.
It is known~\cite{PT97} that this constant should be greater than $1.27$.\footnote{
The constant $0.42$ is mentioned in~\cite{PT97} due to miscalculation. It was pointed out recently in~\cite{BCV18} that the actual constant is $1.27$.}

%\begin{comment}

\subsection{Albertson conjecture}

Recall that according to Albertson conjecture if $\chi(G)=r$ then $\X(G) \geq \X(K_r)$.
This has been proven to hold for $r \leq 16$~\cite{ACF09,BT10,OZ09}.
The values $13 \leq r \leq 16$ where verified Bar\'at and T\'oth~\cite{BT10}
who followed and refined the approach of Albertson et al.~\cite{ACF09}.
By using the new bound in Theorem~\ref{thm:linear-new} and following the same approach
we can now verify Albertson conjecture for $r \leq 18$.
For completeness we repeat the arguments from~\cite{BT10}.

A graph $G$ is \emph{$r$-critical} if $\chi(G)=r$ and 
the chromatic number of every proper subgraph of $G$ is less than $r$.
Obviously, if $H$ is a subgraph of $G$ then $\X(H) \leq \X(G)$.
Therefore, it is enough to prove Albertson conjecture for $r$-critical graphs.
Recall also that it suffice to consider graphs with no subdivision of $K_r$.
The next result shows that we may consider only graphs with at least $r+5$ vertices.

\begin{lem}[{\cite[Corollary~11]{BT10}}] \label{lem:r+4}
An $r$-critical graph with at most $r+4$ vertices contains a subdivision of $K_r$
(and thus satisfies Albertson conjecture).
\end{lem}

The approach of~\cite{ACF09} and~\cite{BT10} for proving Albetson conjecture
is to plug lower bounds on the minimum number of edges in $r$-critical graphs
into lower bounds on the crossing number and compare the results to an upper bound on $\X(K_r)$.
By using the same method with the new bounds on the crossing number, 
we can verify Albertson conjecture for further values of $r$.

Let $f_r(n)$ be the minimum number of edges in an $n$-vertex $r$-critical graph.
Since $K_r$ is the only $r$-critical graph with $r$ vertices we have $f_r(r)=r(r-1)/2$.
Another trivial bound is $f_r(n) \geq n(r-1)/2$, because the degree of every vertex in 
an $r$-critical graph must be at least $r-1$. 
The study of $f_r(n)$ goes back to Dirac~\cite{Dir57}.
He proved that there is no $r$-critical graph on $r+1$ vertices and 
that if $r \geq 4$ and $n \geq r+2$ then 
\begin{equation}\label{eq:Dirac}
f_r(n) \geq n(r-1)/2 + (r-3)/2.
\end{equation}
This was improved by Kostochka and Stiebitz~\cite{KS99} to
\begin{equation}\label{eq:KS}
f_r(n) \geq n(r-1)/2 + (r-3),
\end{equation}
when $n \neq 2r-1$.
Considering the case $n=2r-1$, Bar\'at and T\'oth~\cite{BT10} concluded

\begin{lem}[{\cite[Corollary~7]{BT10}}] \label{lem:KS}
Let $G$ be an $n$-vertex $r$-critical graph with $m$ edges, such that $r \geq 4$.
If $G$ does not contain a subdivision of $K_r$ then $m \geq n(r-1)/2 + (r-3)$.
\end{lem}

Gallai~\cite{Gal63} found exact values of $f_r(n)$ for $6 \leq r+2 \leq n \leq 2r-1$:
\begin{equation}\label{eq:Gallai}
f_r(n) = \frac{1}{2}(n(r-1)+(n-r)(2r-n)-2).
\end{equation}
He also characterized the graphs obtaining this bound.
His results yield:

\begin{lem}[{\cite[Corollary~5]{BT10}}] \label{lem:Gallai}
Let $G$ be an $n$-vertex $r$-critical graph with $m$ edges, such that $6 \leq r+2 \leq n \leq 2r-1$.
If $G$ does not contain a subdivision of $K_r$ then $m \geq \frac{1}{2}(n(r-1)+(n-r)(2r-n)-1)$.
\end{lem}

Instead of using the linear bound of Theorem~\ref{thm:linear-new} directly,
we will use a more refined bound obtained from it using the probabilistic argument
(as is done in~\cite{BT10}).

\begin{lem} \label{lem:detailed-cr}
Let $\X(n,m,p) = \frac{5m}{p^2} - \frac{139n}{6p^3} + \frac{139}{3p^4} - \frac{6n^2(1-p)^{n-2}}{p^4}$.
For every graph $G$ with $n \geq 9$ vertices and $m$ edges and every $0 < p \leq 1$ we have $\X(G) \geq \X(n,m,p)$.
\end{lem}

\begin{proof}
We will use the linear bound of Theorem~\ref{thm:linear-new}, however it does not hold for $n \leq 2$.
Therefore, for every graph $G$ we define
$$
\X'(G) = \begin{cases} 
\X(G) & \mbox{if } n \geq 3 \\
5 & \mbox{if } n = 2 \\
24 & \mbox{if } n = 1 \\
47 & \mbox{if } n = 0 
\end{cases}
$$
Thus, for every graph $G$ we have 
\begin{equation}\label{eq:cr'-linear}
\X'(G) \geq 5m-\frac{139}{6}(n-2).
\end{equation}

Let $G$ be a graph with $n$ vertices and $m$ edges and let $0 < p \leq 1$.
Consider a drawing of $G$ with $\X(G)$ crossings.
Construct a random subgraph of $G$ by selecting every vertex independently with probability $p$.
Let $G'$ be the subgraph of $G$ that is induced by the selected vertices.
Denote by $n'$ and $m'$ the number of vertices and edges in $G'$, respectively.
Consider the drawing of $G'$ as inherited from the drawing of $G$, and let $x'$ be the number of crossings in this drawing.
Clearly, $\E[n']=pn$, $\E[m']=p^2m$, and $\E[x']=p^4\X(G)$.
From~(\ref{eq:cr'-linear}) and the linearity of expectation we get:
\begin{eqnarray}
\E[x'] & \geq & \E[\X(G')] - 5\cdot\Pr(n'=2) - 24\cdot\Pr(n'=1) - 47\cdot\Pr(n'=0) \nonumber \\
& \geq & 5p^2m - \frac{139}{6}pn + \frac{139}{3} - 5{{n}\choose{2}}p^2(1-p)^{n-2} - 24np(1-p)^{n-1} - 47(1-p)^{n} \nonumber  \\
& \geq & 5p^2m - \frac{139}{6}pn + \frac{139}{3} - 6n^2p^2(1-p)^{n-2}. \nonumber
\end{eqnarray}
Dividing by $p^4$, the lemma follows.
\end{proof}

Before proving Theorem~\ref{thm:Albertson}, let us recall the best known
upper bound on the crossing number of $K_r$~\cite{Guy60}:
\begin{equation}\label{eq:cr(K_r)}
\X(K_r) \leq Z(r) = \frac{1}{4}\left\lfloor\frac{r}{2}\right\rfloor\left\lfloor\frac{r-1}{2}\right\rfloor\left\lfloor\frac{r-2}{2}\right\rfloor\left\lfloor\frac{r-3}{2}\right\rfloor.
\end{equation}

\begin{proofof}{Theorem~\ref{thm:Albertson}}
We follow the proof of Theorem~2 in~\cite{BT10}.
Given $r$ let $G$ be an $r$-critical graph with $n$ vertices and $m$ edges.
We assume that $G$ does not contain a subdivision of $K_r$ for otherwise we are done.
By Lemma~\ref{lem:r+4} we may assume that $n \geq r+5$.
Lemma~\ref{lem:KS} is used to get a lower bound on $m$, namely $m \geq (r-1)n/2+(r-3)$.
This bound is plugged into Lemma~\ref{lem:detailed-cr} and for an appropriate
value of $p$ we get a lower bound on $\X(G)$ that is greater than $Z(r)$ for $n \geq n'$.
Then it remains to verify the conjecture for each $n$ in the range $r+5,\ldots,n'$.
This is done using a lower bound on $m$ that we get from either Lemma~\ref{lem:KS} or Lemma~\ref{lem:Gallai}
and picking $p$ such that $\X(n,m,p) \geq Z(r)$.
We will always have $n \geq 22$ and $p \geq 0.5$, therefore we may assume that
\begin{equation} \label{eq:cr(n,m,p)}
\X(n,m,p) \geq \frac{5m}{p^2} - \frac{139n}{6p^3} + \frac{139}{3p^4} - 0.05
\end{equation}

1. Suppose that $r=17$ and let $G$ be an $n$-vertex $17$-critical graph with $m$ edges.
By~(\ref{eq:cr(K_r)}) we have $\X(K_r) \leq 784$.
It follows from Lemmas~\ref{lem:r+4} and~\ref{lem:KS} that we may assume that $n \geq 22$ and $m \geq 8n+14$.
From~(\ref{eq:cr(n,m,p)}) we have $\X(G) \geq \X(n,8n+14,0.727) \geq 15.38n+298.25$.
Therefore, if $n \geq \frac{784-298.25}{15.38} \geq 31.58$ the conjecture holds.
Since Bar\'at and T\'oth~\cite{BT10} have already verified Albertson conjecture for $r=17$ and $n \leq 31$, we are done.

2. Suppose that $r=18$ and let $G$ be an $n$-vertex $18$-critical graph with $m$ edges.
By~(\ref{eq:cr(K_r)}) we have $\X(K_r) \leq 1008$.
It follows from Lemmas~\ref{lem:r+4} and~\ref{lem:KS} that we may assume that $n \geq 23$ and $m \geq 8.5n+15$.
From~(\ref{eq:cr(n,m,p)}) we have $\X(G) \geq \X(n,8.5n+15,0.69) \geq 18.74n+361.88$.
Therefore, if $n \geq \frac{1008-361.88}{18.74} \geq 34.47$ the conjecture holds.
It remains to verify the conjecture for $n=23,\ldots,34$.
Table~1 (left) shows the lower bound on $m$ for each $n$,
the value of $p$ we choose, and the corresponding lower bound on the crossing number that we get.
Note that since we are interested in values of $n$ such that $r+2 \leq n \leq 2r-1$,
we may use Lemma~\ref{lem:Gallai} instead of the Lemma~\ref{lem:KS}.

\begin{table} \label{tab:18-19}
\begin{center}
  \begin{tabular}{ | c | c | c | c |}
  \hline
  \multicolumn{4}{|c|}{$r=18$, $\X(K_{18})\leq 1008$} \\
  \hline
    $n$ & $m$ & $p$ & $\lceil \X(n,m,p)\rceil$ \\ \hline
    23 & 228 & 0.555 & 1073 \\
    24 & 240 & 0.556 & 1132 \\
    25 & 251 & 0.560 & 1176 \\
    26 & 261 & 0.567 & 1204 \\
    27 & 270 & 0.576 & 1217 \\
    28 & 278 & 0.586 & 1218 \\
    29 & 285 & 0.599 & 1206 \\
    30 & 291 & 0.613 & 1183 \\
    31 & 296 & 0.628 & 1151 \\
    32 & 300 & 0.646 & 1111 \\
    33 & 303 & 0.665 & 1064 \\
    34 & 305 & 0.686 & 1010 \\ 
     &  &  &  \\ 
     &  &  &  \\ 
     &  &  &  \\ 
     &  &  &  \\ \hline
  \end{tabular}
\quad\quad\quad
  \begin{tabular}{ | c | c | c | c |}
  \hline
  \multicolumn{4}{|c|}{$r=19$, $\X(K_{19})\leq 1296$} \\
  \hline
    $n$ & $m$ & $p$ & $\lceil \X(n,m,p)\rceil$ \\ \hline
     &  &  &  \\ 
    24 & 251 & 0.523 & 1321 \\
    25 & 264 & 0.524 & 1397 \\
    26 & 276 & 0.527 & 1455 \\
    27 & 287 & 0.533 & 1495 \\
    28 & 297 & 0.540 & 1518 \\
    29 & 306 & 0.548 & 1527 \\
    30 & 314 & 0.558 & 1520 \\
    31 & 321 & 0.570 & 1501 \\
    32 & 327 & 0.583 & 1471 \\
    33 & 332 & 0.597 & 1430 \\
    34 & 336 & 0.613 & 1380 \\
    35 & 339 & 0.631 & 1322 \\
    36 & 341 & 0.650 & {\bf 1259} \\
    37 & 349 & 0.656 & {\bf 1269} \\
    38 & 358 & 0.659 & {\bf 1292} \\ \hline
  \end{tabular}
  \caption{Lower bounds on the number of edges and crossing numbers for specific values of $n$ for $r=18$ (left) and $r=19$ (right).}
\end{center}
\end{table}

3. Suppose that $r=19$ and let $G$ be an $n$-vertex $19$-critical graph with $m$ edges.
By~(\ref{eq:cr(K_r)}) we have $\X(K_r) \leq 1296$.
It follows from Lemmas~\ref{lem:r+4} and~\ref{lem:KS} that we may assume that $n \geq 23$ and $m \geq 9n+16$.
From~(\ref{eq:cr(n,m,p)}) we have $\X(G) \geq \X(n,9n+16,0.66) \geq 22.72n+427.78$.
Therefore, if $n \geq \frac{1296-427.78}{22.72} \geq 38.21$ then the conjecture holds.
It remains to verify the conjecture for $n=24,\ldots,38$.
Table~1 (right) shows the lower bound on $m$ for each $n$,
the value of $p$ we choose, and the corresponding lower bound 
on the crossing number that we get.\footnote{The code of our calculations appears in Appendix~\ref{sec:code}.}

Therefore, the conjecture holds for $r=19$ and every $n \notin \{36,37,38\}$.
As is done in~\cite{BT10}, we can handle the case $n=36$ by using a result of Gallai~\cite{Gal63},
who proved that an $r$-critical graph with $2r-2$ vertices is the \emph{join}\footnote{A \emph{join}
of two graphs $G_1=(V_1,E_1)$ and $G_2=(V_2,E_2)$ consists of the two graphs and the edges
$\{(v_1,v_2) \mid v_1 \in V_1, v_2 \in V_2 \}$.}
of two smaller critical graphs.
Therefore, if $n=36$ then $G$ is the join of an $r_1$-critical graph $G_1=(V_1,E_1)$ 
and an $r_2$-critical graph $G_2=(V_2,E_2)$, such that $r_1+r_2=19$.
Let $n_i=|V_i|$ and $m_i=|E_i|$, for $i=1,2$.
Then $n_1+n_2=36$ and $m = m_1+m_2+n_1 n_2$.

We assume without loss of generality that $r_1 \leq r_2$,
and therefore have to consider the cases $r_1=1,\ldots,9$.
Suppose that $r_1=1$, which implies that $G_1=K_1$. 
If $G_2$ contains a subdivision of $K_{18}$ then $G$ contains a subdivision of $K_{19}$ and we are done.
Otherwise, by Lemma~\ref{lem:KS} we get $m_2 \geq 313$.
Therefore, $m = n_1n_2+ m_2 \geq 348$ when $r_1=1$.

Since $G_1$ is $r_1$-critical and $G_2$ is $r_2$-critical we have 
$m \geq f_{r_1}(n_1) + f_{r_2}(n_2) + n_1n_2$.
Note that $n_1=36-n_2 \leq 36-r_2 = r_1+17$ and if $r_1=2$ then $G_1=K_2$. 
A computer calculation using the trivial bound for $f_r(n)$ along with (\ref{eq:KS}),
reveals that $m \geq 348$ for every $r_1=2,\ldots,9$ and every $n_1 = r_1,\ldots,r_1+17$
(ignoring cases where $n_1=r_1+1$ or $n_2=r_2+1$ since there are no such critical graphs).
Therefore, we conclude that $G$ has at least $348$ edges.
Picking $p=0.635$ we get that $\X(G) \geq \X(36,348,0.635) = 1343 \geq \X(K_{19})$. 
% Suppose that $r_1=2$, which implies that $G_2=K_2$.
% If $G_2$ contains a subdivision of $K_{17}$ then $G$ contains a subdivision of $K_{19}$ and we are done.
% Otherwise, by Lemma~\ref{lem:KS} we get $m_2 \geq 286$.
% Therefore, $m = n_1n_2+ m_2 \geq 354$.
%
% Suppose that $r_1=3$, which implies that $G_2=K_3$ or an odd cycle.
% By the trivial bound $m_1 \geq n_1$ and by Lemma~\ref{lem:KS} we have $m_2 \geq (15n_2+26)/2$.
% Thus, $m \geq \lceil n_1 + 15(n_2+26)/2 + n_1n_2 \rceil = \lceil 29.5n_1-n_1^2+283 \rceil$.
% The last expression attains its minimum value $363$ for $n_1=3,\ldots,20$ when $n_1=3$
% (note that since $r_2 = 16$ we have $n_2 \geq 16$ and so $n_1 \leq 20$).
\end{proofof}

Recall that Bar\'at and T\'oth~\cite{BT10} showed that if Albertson conjecture is false, 
then the minimal counter-example is an $r$-critical graph with at least $r+5$ vertices (Lemma~\ref{lem:r+4}).
They also gave an upper bound of $3.57r$ on the number of vertices in such a minimal counter-example
(improving a $4r$ bound due to Albertson et al.~\cite{ACF09}).
Using Theorem~\ref{thm:linear-new} we can improve upon this bound as well.

\begin{lem}
\label{lem:AC-smallest-counterexample}
If $G$ is an $r$-critical graph with $n \geq 3.03r$ vertices, then $\X(G) \geq \X(K_r)$.
\end{lem}

\begin{proof}
The proof is similar to the proof of Lemma~3 in~\cite{BT10}.
We repeat it here for completeness, and because there is a small typo
in the calculation in~\cite{BT10}.
%(a multiplicative factor $\alpha^3$ is missing in the LHS of the last inequality).

Let $G$ be an $r$-critical graph with $n$ vertices drawn in the plane with $\X(G)$ crossings.
We may assume that $r \geq 19$, since for $r \leq 18$ the conjecture holds.
If $n \geq 3.57r$ then it follows from~\cite{BT10} that $\X(G) \geq \X(K_r)$.
Therefore, we assume that $n = \alpha r$ for some $3.03 \leq \alpha < 3.57$.
Note that $n \geq 3r \geq 57$.
Let $5 \leq k \leq n$ be an integer and let $G_1,G_2,\ldots,G_t$, $t = {{n}\choose{k}}$,
be all the (inherited drawings of) subgraphs induced by exactly $k$ vertices in $G$.
%(notice that $n \geq 3r \geq 57$ so there are indeed subgraphs of $G$ with $40$ vertices).
Denote by $m_i$ the number of edges in $G_i$, and note that by Theorem~\ref{thm:linear-new}
we have $\X(G_i) \geq 5m_i-\frac{139}{6}(k-2)$.
Observe also that every crossing in $G$ appears in ${{n-4}\choose{k-4}}$ subgraphs
and every edge in $G$ appears in ${{n-2}\choose{k-2}}$ subgraphs.
Finally, recall that $m \geq n(r-1)/2$ since $G$ is $r$-critical.
Thus we have,
\begin{eqnarray}
\X(G) & \geq & \frac{1}{{{n-4}\choose{k-4}}} \sum_{i=1}^{t} \X(G_i) \geq \frac{1}{{{n-4}\choose{k-4}}} \sum_{i=1}^{t} \left( 5m_i-\frac{139(k-2)}{6}\right) \nonumber \\
& = & 5m\frac{{{n-2}\choose{k-2}}}{{{n-4}\choose{k-4}}} - \frac{139(k-2){{n}\choose{k}}}{6{{n-4}\choose{k-4}}} \nonumber \\
% & = & 5m\frac{(n-2)(n-3)}{(k-2)(k-3)} - \frac{139n(n-1)(n-2)(n-3)}{6k(k-1)(k-3)} \nonumber \\
& \geq & \frac{5(r-1)n}{2}\frac{(n-2)(n-3)}{(k-2)(k-3)} - \frac{139n(n-1)(n-2)(n-3)}{6k(k-1)(k-3)} \nonumber \\
%& = & \frac{5(r-1)\alpha^3 r(r-\frac{2}{\alpha})(r-\frac{3}{\alpha})}{2(k-2)(k-3)} - 
%	\frac{139 \alpha^4 r(r-\frac{1}{\alpha})(r-\frac{2}{\alpha})(r-\frac{3}{\alpha})}{6k(k-1)(k-3)} \\ \nonumber
& = & \frac{n(n-2)(n-3)}{2(k-3)} \left( \frac{5(r-1)}{k-2} - \frac{139(n-1)}{3k(k-1)} \right) \nonumber \\
& = & \frac{\alpha^3 r(r-\frac{2}{\alpha})(r-\frac{3}{\alpha})}{2(k-3)} \left( \frac{5(r-1)}{k-2} - \frac{139(\alpha r-1)}{3k(k-1)}\right)\nonumber \\
& \geq & \frac{\alpha^3 r(r-2)((r-3)+2)}{2(k-3)} \left( \frac{5(r-1)}{k-2} - \frac{139(r-1)(\alpha+\frac{\alpha-1}{r-1})}{3k(k-1)} \right)\nonumber\\
& = & \frac{\alpha^3 r(r-1)(r-2)(r-3)}{2(k-3)} \left( \frac{5}{k-2} - \frac{139\alpha}{3k(k-1)}\right) + h(\alpha,r,k),\nonumber
\end{eqnarray}
where
\begin{eqnarray}
h(\alpha,r,k) & = & \frac{\alpha^3 r(r-1)(r-2)}{2(k-3)}\left(\frac{10}{k-2}- \frac{139}{3k(k-1)}\left(2\alpha + 2\frac{\alpha-1}{r-1} + \frac{r-3}{r-1}(\alpha-1)\right)\right) \nonumber \\
& \geq & \frac{\alpha^3 r(r-1)(r-2)}{2(k-3)}\left(\frac{10}{k-2}- \frac{139}{3k(k-1)}\left(2\alpha + \frac{\alpha-1}{9} + (\alpha-1)\right)\right). \nonumber
\end{eqnarray}

Suppose now that $3.17 \leq \alpha \leq 3.57$.
Then for $k=47<n$ we have $h(\alpha,r,47) \geq 0$ and therefore
\begin{eqnarray}
\X(G) & \geq & \frac{\alpha^3}{2\cdot 44} \left( \frac{5}{45} - \frac{139\alpha}{3\cdot 47 \cdot 46}\right)r(r-1)(r-2)(r-3) \nonumber\\
      & \geq & \frac{1}{64} r(r-1)(r-2)(r-3) \geq \X(K_r). \nonumber
\end{eqnarray}  

Suppose now that $3.05 \leq \alpha \leq 3.17$.
Then for $k=41<n$ we have $h(\alpha,r,41) \geq 0$ and therefore
\begin{eqnarray}
\X(G) & \geq & \frac{\alpha^3}{2\cdot 38} \left( \frac{5}{39} - \frac{139\alpha}{3\cdot 41 \cdot 40}\right)r(r-1)(r-2)(r-3) \nonumber\\
      & \geq & \frac{1}{64} r(r-1)(r-2)(r-3) \geq \X(K_r). \nonumber
\end{eqnarray}  

Finally, suppose that $3.03 \leq \alpha \leq 3.05$.
Then for $k=40<n$ we have $h(\alpha,r,40) \geq 0$ and therefore
\begin{eqnarray}
\X(G) & \geq & \frac{\alpha^3}{2\cdot 37} \left( \frac{5}{38} - \frac{139\alpha}{3\cdot 40 \cdot 39}\right)r(r-1)(r-2)(r-3) \nonumber\\
      & \geq & \frac{1}{64} r(r-1)(r-2)(r-3) \geq \X(K_r). \nonumber
\end{eqnarray}  
\end{proof}

%\end{comment}
\paragraph{Acknowledgements.} I thank an anonymous referee for carefully reading the paper, suggesting ways to improve its presentation and bringing the remark in~\cite{BCV18} to my attention.

\bibliographystyle{abbrv}

\begin{thebibliography}{10}

\footnotesize

\bibitem{Ac09}
E.~Ackerman, 
On the maximum number of edges in topological graphs with no four pairwise crossing edges, 
{\em Disc.\ Compu.\ Geometry}, 41~(2009), 365–-375.

%\bibitem{CGTA}
%E.~Ackerman, 
%On topological graphs with at most four crossings per edge, 
%{\em Computational Geometry: Theory and Applications}, to appear.

\bibitem{AT07}
E.~Ackerman and G.~Tardos, 
On the maximum number of edges in quasi-planar graphs, 
{\em J.\ Combinatorial Theory, Ser.\ A.}, 114:3~(2007), 563-571.

%\bibitem{Adam}
%A.~Sheffer,
%Incidences: lower bounds (part 2),
%{\em Some Plane Truths},
%Retrieved on July 20, 2014 from {\tt %	http://adamsheffer.wordpress.com/2014/07/01/incidences-lower-bounds-part-2/}.

\bibitem{AZ04}
M.~Aigner and G.~Ziegler, 
{\em Proofs from the Book}, Springer-Verlag, Heidelberg, 2004.

\bibitem{AH77}
K.~Appel and W.~Haken, 
Every planar map is four colorable. Part I. Discharging, 
{\em Illinois J.\ Math.}, 21~(1977), 429-–490.

\bibitem{AFKMT12}
K.~Arikushi, R.~Fulek, B.~Keszegh, F.~Moric, and C.D.~T\'oth,
Graphs that admit right angle crossing drawings,
{\em Comput.\ Geom.\ Theory Appl.}, 45:7~(2012), 326--333.

\bibitem{ACNS82}
{M. Ajtai, V. Chv\'atal, M. Newborn, and E. Szemer\'edi,} 
{Crossing-free subgraphs},
{\it Theory and Practice of Combinatorics}, North-Holland Math. Stud. 60,
North-Holland, Amsterdam, 1982, 9--12.

\bibitem{ACF09}
M.O.~Albertson, D.W.~Cranston, and J.~Fox,
Crossings, colorings and cliques,
{\em Elec.\ J.\ Combinatorics} 16~(2010), \#R45.

\bibitem{BCV18}
M.~Balki, J.~Cibulka and P.~Valtr,
Covering lattice points by subspaces and counting point- hyperplane incidences,
{\em Disc.\ Compu.\ Geometry}, to appear.


\bibitem{BT10}
J.~Bar\'at and G.~T\'oth,
Towards the Albertson conjecture,
{\em Elec.\ J.\ Combinatorics} 17:1~(2010), \#R73.

\bibitem{BMP05}
P.~Brass, W.~Moser, J.~Pach,
{\em Research Problems in Discrete Geometry}, Springer, 2005.

\bibitem{Cat79}
P.A.~Catlin, 
Haj\'os' graph-coloring conjecture: variations and counterexamples, 
{\em J.\ Combin.\ Theory Ser.\ B}, 26~(1979), 268--274.

\bibitem{Dir57}
G.A.~Dirac, 
A theorem of R.L.~Brooks and a conjecture of H.~Hadwiger, 
{\em Proc.\ London Math.\ Soc.} 7 (1957), 161--195.

\bibitem{EF81}
P.~Erd\H{o}s and S.~Fajtlowicz, 
On the conjecture of Haj\'os, 
{\em Combinatorica}, 1~(1981), 141--143.

\bibitem{Gal63}
T.~Gallai, 
Kritische Graphen II, 
{\em Publ.\ Math.\ Inst.\ Hungar.\ Acad.\ Sci.} 8~(1963), 373--395.

\bibitem{Guy60}
R.K.~Guy, 
A combinatorial problem, 
{\em Nabla (Bulletin of the Malayan Mathematical Society)}, 7~(1960), 68--72.

\bibitem{KS99}
A.V.~Kostochka and M.~Stiebitz, 
Excess in colour-critical graphs, 
in: Graph Theory and Combinatorial Biology, 
Balatonlelle (Hungary), 1996, 
{\em Bolyai Society, Mathematical Studies} 7, Budapest, 1999, 87--99.

\bibitem{KP11}
S.~Kurz and R.~Pinchasi, 
Regular matchstick graphs,
{\em American Mathematical Monthly}, 118:3~(2011), 264--267. 

\bibitem{L83}
F.T. Leighton, 
{\it Complexity Issues in VLSI: Optimal Layouts for the Shuffle-Exchange Graph and Other Networks}, 
MIT Press, Cambridge, MA, 1983.

\bibitem{Mon05}
B.~Montaron,
An improvement of the crossing number bound,
{\em J.\ Graph Theory}, 50:1~(2005), 43--54.

\bibitem{OZ09}
B.~Oporowskia and D.~Zhao,
Coloring graphs with crossings,
{\em Disc.\ Math.}, 309:6~(2009), 2948--2951.

\bibitem{PR+06}
J.~Pach, R.~Radoi\v{c}i\'{c}, G.~Tardos, G.~T\'oth,
Improving the crossing lemma by finding more crossings in sparse graphs,
{\em Disc.\ Compu.\ Geometry}, 36:4~(2006), 527--552.

\bibitem{PT97}
J.~Pach and G.~T\'oth,
Graphs drawn with few crossings per edge,
{\em Combinatorica}, 17:3~(1997), 427--439.

\bibitem{RT08}
R.~Radoi\v{c}i\'{c} and G.~T\'oth, 
The discharging method in combinatorial geometry and the Pach--Sharir conjecture,
{\em Proc.\ Summer Research Conference on Discrete and Computational Geometry}, (J. E. Goodman, J. Pach, J. Pollack, eds.), 
Contemporary Mathematics, AMS, 453 (2008), 319--342.

\end{thebibliography}

%\begin{comment}

\appendix

\section{{\sf sage} code of the calculations in the proof of Theorem~\ref{thm:Albertson}}
\label{sec:code}

\begin{lstlisting}
sage: Dirac(n,r)=((r-1)*n+r-3)/2
sage: KS(n,r)=((r-1)*n+2*r-6)/2
sage: Gallai(n,r)=((r-1)*n+(n-r)*(2*r-n)-2)/2
sage: BT_Gal(n,r)=Gallai(n,r)+0.5
sage: cr_prime(n,m,p)=5*m/p^2-139*n/(6*p^3)+139/(3*p^4)-0.05
sage: Z(r)=floor(r/2)*floor((r-1)/2)*floor((r-2)/2)*floor((r-3)/2)/4
sage: def proc1(r):
...       sols = solve([cr_prime(n,KS(n,r),p).diff(p)==0, cr_prime(n,KS(n,r),p)==Z(r)],n,p,solution_dict=True)
...       for s in sols:
...           if (s[n].imag()==0 and s[p].imag()==0): # output only real solutions
...               print "p=",s[p].n(),",n=",s[n].n()
sage: proc1(17)
p= 0.727523979840676 ,n= 31.5627659574468
sage: cr_prime(n,KS(n,17),0.727)
15.3896636507376*n + 298.258502516192
sage: proc1(18)
p= 0.690689920492434 ,n= 34.4659498207885
sage: cr_prime(n,KS(n,18),0.69)
18.7463154231188*n + 361.887598221377
sage: def proc2(n,r):
...       if n <= 2*r-2:
...           m = ceil(BT_Gal(n,r))
...       else:
...           m = ceil(KS(n,r))
...       sols = solve(diff(cr_prime(n,m,p),p)==0, p, solution_dict=True)
...       best_p= round(sols[1][p],3)
...       best_cr = ceil(cr_prime(n,m,best_p))
...       str = '\t\t'+repr(n)+' & '+repr(m)+' & '+repr(best_p.n())+' & '+repr(best_cr) + ' \\\\'
...       print str
sage: for n in range(23,35):
...       proc2(n,18)
		23 & 228 & 0.555000000000000 & 1073 \\
		24 & 240 & 0.556000000000000 & 1132 \\
		25 & 251 & 0.560000000000000 & 1176 \\
		26 & 261 & 0.567000000000000 & 1204 \\
		27 & 270 & 0.576000000000000 & 1217 \\
		28 & 278 & 0.586000000000000 & 1218 \\
		29 & 285 & 0.599000000000000 & 1206 \\
		30 & 291 & 0.613000000000000 & 1183 \\
		31 & 296 & 0.628000000000000 & 1151 \\
		32 & 300 & 0.646000000000000 & 1111 \\
		33 & 303 & 0.665000000000000 & 1064 \\
		34 & 305 & 0.686000000000000 & 1010 \\
sage: proc1(19)
p= 0.659831121833534 ,n= 38.2051696284330
sage: cr_prime(n,KS(n,19),0.66)
22.7249538544304*n + 427.789066289688
sage: for n in range(24,39):
...       proc2(n,19)
		24 & 251 & 0.523000000000000 & 1321 \\
		25 & 264 & 0.524000000000000 & 1397 \\
		26 & 276 & 0.527000000000000 & 1455 \\
		27 & 287 & 0.533000000000000 & 1495 \\
		28 & 297 & 0.540000000000000 & 1518 \\
		29 & 306 & 0.548000000000000 & 1527 \\
		30 & 314 & 0.558000000000000 & 1520 \\
		31 & 321 & 0.570000000000000 & 1501 \\
		32 & 327 & 0.583000000000000 & 1471 \\
		33 & 332 & 0.597000000000000 & 1430 \\
		34 & 336 & 0.613000000000000 & 1380 \\
		35 & 339 & 0.631000000000000 & 1322 \\
		36 & 341 & 0.650000000000000 & 1259 \\
		37 & 349 & 0.656000000000000 & 1269 \\
		38 & 358 & 0.659000000000000 & 1292 \\
sage: def f(n,r): # lower bound for the number of edge in n-vertex r-critical graph
...       best=0
...       if n==r: # K_r
...           best=n*(n-1)/2
...       elif n>r+1:
...           best=ceil(n*(r-1)/2) # trivial
...           if (r>=4 and n>=r+2):
...               best=max(best,ceil(Dirac(n,r)))
...               if n!=2*r-1:
...                   best=max(best,ceil(KS(n,r)))
...               if n<=2*r-1:
...                   best=max(best,ceil(Gallai(n,r)))
...       return best
sage: # considering the case r=19, n=36
sage: min_m=348
sage: for r1 in range(2,10):
...       r2 = 19-r1
...       if r1==2:
...           max_n1=2
...       else:
...           max_n1=36-r2
...       for n1 in range(r1,max_n1+1):
...           n2 = 36-n1
...           if (n1!=r1+1 and n2!=r2+1):
...               curr = f(n1,r1)+f(n2,r2)+n1*n2
...               min_m = min(min_m,curr)
...
sage: print min_m
348
\end{lstlisting}

%\end{comment}

\end{document}